\documentclass[reqno]{amsart}
\usepackage{amssymb}
\usepackage{graphicx}

\usepackage[usenames, dvipsnames]{color}
\usepackage{verbatim}
\usepackage{mathrsfs}
\usepackage{bm}
\usepackage{cite}

\numberwithin{equation}{section}

\newtheorem{theorem}{Theorem}[section]
\newtheorem{corollary}[theorem]{Corollary}
\newtheorem{lemma}[theorem]{Lemma}

\newtheorem{example}[theorem]{Example}

\theoremstyle{definition}
\newtheorem{remark}[theorem]{Remark}

\theoremstyle{definition}

\theoremstyle{definition}

\makeatletter
\def\dashint{\operatorname%
{\,\,\text{\bf-}\kern-.98em\DOTSI\intop\ilimits@\!\!}}
\makeatother

\def\\det{\text{det}}

\def\.5{\frac{1}{2}}

\newcommand{\RN}[1]{%
  \textup{\uppercase\expandafter{\romannumeral#1}}%
}

\renewcommand{\epsilon}{\varepsilon}

\newcounter{marnote}

\begin{document}
\title[Singularities of the stress concentration]{Singularities of the stress concentration in the presence of $C^{1,\alpha}$-inclusions with core-shell geometry}

\author[Z.W. Zhao]{Zhiwen Zhao}

\address[Z.W. Zhao]{Beijing Computational Science Research Center, Beijing 100193, China.}
%\address{2. School of Mathematical Sciences, Beijing Normal University, Beijing 100875, China.}
\email{zwzhao365@163.com}

\author[X. Hao]{Xia Hao}
\address[X. Hao]{School of Mathematical Sciences, Beijing Normal University, Beijing 100875, China. }
\email{xiahao0915@163.com}

%\thanks{}
%\footnote{ }

\date{\today} % delete this line to display the current date

%%% BEGIN DOCUMENT

%\tableofcontents

\begin{abstract}
In high-contrast composites, if an inclusion is in close proximity to the matrix boundary, then the stress, which is represented by the gradient of a solution to the Lam\'{e} systems of linear elasticity, may exhibits the singularities with respect to the distance $\varepsilon$ between them. In this paper, we establish the asymptotic formulas of the stress concentration for core-shell geometry with $C^{1,\alpha}$ boundaries in all dimensions by precisely capturing all the blow-up factor matrices, as the distance $\varepsilon$ between interfacial boundaries of a core and a surrounding shell goes to zero. Further, a direct application of these blow-up factor matrices gives the optimal gradient estimates.
\end{abstract}

\maketitle

%\noindent{\bf{Keywords}}: Stress concentration; asymptotic formulas; blow-up factor matrices; core-shell structure; $C^{1,\alpha}$-inclusions.

\section{Introduction}
In high-contrast fiber-reinforced composite materials, which comprise of inclusions and a matrix, it is common that hard inclusions are closely located or close to touching the boundary of background medium. When the distance between inclusions or between the inclusions and the matrix boundary tends to zero, the stress field always concentrates highly in the narrow regions between them. To give a clear understanding of this high concentration, there has been a sustained effort to develop quantitative theories over the past twenty years since the great work of Babu\u{s}ka et al. \cite{BASL1999}, where numerical computation of damage and fracture in linear composite systems was studied. The numerical investigation showed that the size of the strain tensor retains bounded as the distance between two inclusions tends to zero. This observation was demonstrated in the subsequent work \cite{LN2003} completed by Li and Nirenberg. It is worth emphasizing that they established stronger $C^{1,\alpha}$ estimates for general divergence form elliptic systems including the Lam\'{e} system with piecewise H\"{o}lder continuous coefficients for domains of arbitrary smooth shapes in all dimensions. The corresponding results for scalar elliptic equations can refer to \cite{BV2000} and \cite{LV2000}. Recently, Dong and Li \cite{DL2019} established the optimal upper and lower bound estimates on the gradient of a solution to a class of non-homogeneous elliptic equations with discontinuous coefficients. These estimates especially showed the clear dependence on elliptic coefficients, which answered open problem $(b)$ proposed by Li and Vogelius \cite{LV2000} in the case of circular inclusions.

In the context of electrostatics, Ammari et al. \cite{AKL2005} were the first to investigate the conductivity equation for two circular inclusions, also called the simplified scalar model of the elasticity problem. They showed that the blow-up rate of the gradient is $\varepsilon^{-1/2}$ by constructing a lower bound on the gradient in two dimensions. After that, it brought a long list of literature on this topic, for instance, see \cite{BLY2010,Y2007,Y2009,BLY2009,LY2009,AKLLL2007} and the references therein. It has been proved there that the blow-up rates of the gradients are $\varepsilon^{-1/2}$ in two dimensions, $\varepsilon|\log\varepsilon|^{-1}$ in three dimensions and $\varepsilon^{-1}$ in dimensions greater than or equal to four, respectively. However, since there is no maximum principle holding for the systems, it prevented us extending the results of the perfect conductivity equation to the Lam\'{e} systems until Bao, Li and Li \cite{BLL2015,BLL2017} applied an iteration technique with respect to the energy which was created in \cite{LLBY2014} to establish the pointwise upper bounds on the gradients of solutions to the Lam\'{e} systems with partially infinite coefficients. Their results indicated that the blow-up appears in the shortest line segment between inclusions. Recently, Kang and Yu \cite{KY2019} gave a complete description for the singularities of the stress concentration in the thin gap between two inclusions in dimension two by introducing singular functions constructed by nuclei of strain. Beside these aforementioned work related to interior estimates of the gradients, there is another direction to investigate the boundary estimates such as \cite{BJL2017}, which showed that the boundary data only contributes to the blow-up factor with no singularities in the presence of $2$-convex inclusions. There is no intrinsic difference in terms of the role of the boundary data played in the interior and boundary estimates until the boundary data of $k$-order growth was investigated in \cite{LZ2020}, where Li and Zhao found that this type of boundary data can increase the stress blow-up rate and change the position of optimal blow-up rate simultaneously. Miao and Zhao \cite{MZ202102} further improved the results in \cite{BJL2017,LZ2020} by capturing the blow-up factor matrices in all dimensions.

Since the modern engineering trends are going toward more reliance on computational predictions, it is significantly important to give a precise description for the singular behavior of the concentrated field to assess the level of accuracy in numerical results. Indeed, the last ten years has witnessed a growing scientific literature investigating the asymptotics of the concentrated field for the perfect conductivity problem such as \cite{KLY2013,ACKLY2013,KLY2014,LLY2019,l2020,ZH202101,HZ2021}. However, it is not easy to extend the asymptotic results for the scalar equation to the Lam\'{e} system due to the increase of the number of free constants, which makes elastic problem much more complex than the perfect conductivity problem. Additionally, for nonlinear $p$-Laplace equation, Gorb and Novikov \cite{GN2012} captured a stress concentration factor for the purpose of establishing the $L^{\infty}$ estimate of the gradient. The results were subsequently extended to the Finsler $p$-Laplacian by Ciraolo and Sciammetta \cite{CS2019,CS20192}. It is worth pointing out that the smoothness of inclusions in the above-mentioned work is required for at least $C^{1,\alpha}$, $0<\alpha<1$. Kang and Yun \cite{KY2019002} recently gave a quantitative characterization for the enhanced field due to presence of the bow-tie structure, which is a special Lipschitz domain. For more related issues and investigations, see \cite{G2015,KLY2015,ABTV2015,KY2020,KL2019,GB2005} and the reference therein.

In the present work, we consider the elasticity problem in the presence of a $C^{1,\alpha}$-inclusion with the core-shell structure which comprises of a core and a surrounding shell, where the core is close to touching the interfacial boundary of the shell. The primary objective of this paper is to establish the precise asymptotic formulas of the stress concentration in any dimension by capturing all the blow-up factor matrices. As an immediate consequence of these blow-up factor matrices, we obtain the optimal gradient estimates. Our results improve and extend the gradient estimates and asymptotics in \cite{CL2019}. In addition, the results of this paper also indicate that for $C^{1,\alpha}$ inclusions close to touching the external boundary, the boundary data doesn't contribute to the singularities of the stress. This is different from the blow-up phenomena revealed in \cite{LZ2020} for $C^{2,\alpha}$ inclusions. In general, the singularities of the stress will be amplified with the deterioration of smoothness of inclusions, while the singular effect of the boundary data will disappear in this case.

We now present an overview of the rest of this paper. In Section \ref{formulation}, we describe our problem and list the main results. We carry out a linear decomposition \eqref{Le2.015} of the gradient of the solution $u$ to problem \eqref{La.002} below in Section \ref{prf results} and then give the proofs of Theorems \ref{Lthm066} and \ref{Lthm06666} in Section \ref{pth results}, which consist of the following asymptotic expansions: $(a)$ asymptotics of $\nabla u_{i}$, $u_{i}$, $i=0,1,...,\frac{d(d+1)}{2}$ are defined by \eqref{P2.005}; $(b)$ asymptotics of the free constants $C^{i}$, $i=1,2,...,\frac{d(d+1)}{2}$ determined by the third line of equation \eqref{La.002}, their proofs are given in Section \ref{SEC005}. Example \ref{CORO001} is presented in Section \ref{SEC006}.

\section{Problem setting and main results}\label{formulation}
\subsection{Formulation of the problem}
Let $D_{1}^{\ast}$ be a convex $C^{1,\alpha}$-subdomain inside a bounded open set $D$ with $C^{1,\alpha}~(0<\alpha<1)$ boundary, which touches the external boundary $\partial D$ only at one point. That is, after translation and rotation if necessary,
\begin{align*}
\partial D_{1}^{\ast}\cap\partial D=\{0\}\in\mathbb{R}^{d},\quad D\subset\{(x',x_{d})\in\mathbb{R}^{d}\,|\,x_{d}>0\}.
\end{align*}
Here and after, we utilize superscript prime to denote ($d-1$)-dimensional domains and variables. By a translation, denote
\begin{align*}
D_{1}^{\varepsilon}:=D_{1}^{\ast}+(0',\varepsilon),
\end{align*}
where $\varepsilon>0$ is a sufficiently small constant. For the sake of convenience, we further simplify the notations as follows:
\begin{align*}
D_{1}:=D_{1}^{\varepsilon},\quad\Omega:=D\setminus\overline{D}_{1},\quad\mathrm{and}\quad\Omega^{\ast}:=D\setminus\overline{D_{1}^{\ast}}.
\end{align*}

Suppose that $\Omega$ and $D_{1}$ are occupied, respectively, by two different isotropic and homogeneous elastic materials with different Lam$\mathrm{\acute{e}}$ constants $(\lambda,\mu)$ and $(\lambda_{1},\mu_{1})$. The elasticity tensors $\mathbb{C}^0$ and $\mathbb{C}^1$ for the background and the inclusion are expressed, respectively, as
$$C_{ijkl}^0=\lambda\delta_{ij}\delta_{kl} +\mu(\delta_{ik}\delta_{jl}+\delta_{il}\delta_{jk}),$$
and
$$C_{ijkl}^1=\lambda_1\delta_{ij}\delta_{kl} +\mu_1(\delta_{ik}\delta_{jl}+\delta_{il}\delta_{jk}),$$
where $i, j, k, l=1,2,...,d$ and $\delta_{ij}$ represents the kronecker symbol: $\delta_{ij}=0$ for $i\neq j$, $\delta_{ij}=1$ for $i=j$. Let $u=(u^{1},u^{2},...,u^{d})^{T}:D\rightarrow\mathbb{R}^{d}$ represent the elastic displacement field. Denote by $\chi_{\Omega}$ the characteristic function of $\Omega\subset\mathbb{R}^{d}$. For a given boundary data $\varphi=(\varphi^{1},\varphi^{2},...,\varphi^{d})^{T}$, we consider the following Dirichlet problem for the Lam$\mathrm{\acute{e}}$ system with piecewise constant coefficients
\begin{align}\label{La.001}
\begin{cases}
\nabla\cdot \left((\chi_{\Omega}\mathbb{C}^0+\chi_{D_{1}}\mathbb{C}^1)e(u)\right)=0,&\hbox{in}~~D,\\
u=\varphi, &\hbox{on}~~\partial{D},
\end{cases}
\end{align}
where $e(u)=\frac{1}{2}\left(\nabla u+(\nabla u)^{T}\right)$ is the strain tensor. Under the hypothesis of the standard ellipticity condition for problem (\ref{La.001}), that is,
\begin{align*}%\label{ellipticity}
\mu>0,\quad d\lambda+2\mu>0,\quad \mu_1>0,\quad d\lambda_1+2\mu_1>0,
\end{align*}
it is well known that there is a unique solution $u\in H^{1}(D;\mathbb{R}^{d})$ to problem (\ref{La.001}) for $\varphi\in H^{1}(D;\mathbb{R}^{d})$. Moreover, $\nabla u$ was proved to be piecewise H\"older continuous in \cite{LN2003}.

Denote by
$$\Psi:=\{\psi\in C^1(\mathbb{R}^{d}; \mathbb{R}^{d})\ |\ \nabla\psi+(\nabla\psi)^T=0\}$$
the linear space of rigid displacement in $\mathbb{R}^{d}$. Let $\{e_{1},...,e_{d}\}$ be the standard basis of $\mathbb{R}^{d}$. It is known that
\begin{align*}
\left\{\;e_{i},\;x_{k}e_{j}-x_{j}e_{k}\;\big|\;1\leq\,i\leq\,d,\;1\leq\,j<k\leq\,d\;\right\}
\end{align*}
is a basis of $\Psi$. Rewrite this basis as $\left\{\psi_{i}\big|\,i=1,2,...,\frac{d(d+1)}{2}\right\}$. For example, in dimension two
$$\psi_{1}=\begin{pmatrix}1\\0\end{pmatrix},\quad\psi_{2}=\begin{pmatrix}0\\1\end{pmatrix},\quad\psi_{3}=\begin{pmatrix}x_{2}\\-x_{1}\end{pmatrix}.$$
For fixed $\lambda$ and $\mu$, let $u_{\lambda_{1},\mu_{1}}$ be the solution of (\ref{La.001}). It has been proved in the appendix of \cite{BLL2015} that
\begin{align*}%\label{limit}
u_{\lambda_1,\mu_1}\rightarrow u\quad\hbox{in}\ H^1(D; \mathbb{R}^{d}),\quad \hbox{as}\ \min\{\mu_1, d\lambda_1+2\mu_1\}\rightarrow\infty,
\end{align*}
with $u\in H^1(D; \mathbb{R}^{d})$ satisfying
\begin{align}\label{La.002}
\begin{cases}
\mathcal{L}_{\lambda, \mu}u:=\nabla\cdot(\mathbb{C}^0e(u))=0,\quad&\hbox{in}\ \Omega,\\
u=C^{\alpha}\psi_{\alpha},&\hbox{on}\ \partial{D}_{1},\\
%u|_{+}=u|_{-},&\hbox{on}\ \partial{D}_{1},\\
%e(u)=0,&\hbox{in}~~D_{1},\\
\int_{\partial{D}_{1}}\frac{\partial u}{\partial \nu_0}\Big|_{+}\cdot\psi_{\alpha}=0,&\alpha=1,2,...,\frac{d(d+1)}{2},\\
u=\varphi,&\hbox{on}\ \partial{D},
\end{cases}
\end{align}
where the free constants $C^{\alpha}$, $\alpha=1,2,...,\frac{d(d+1)}{2}$ will be determined later by the third line and
\begin{align*}
\frac{\partial u}{\partial \nu_0}\Big|_{+}&:=(\mathbb{C}^0e(u))\nu=\lambda(\nabla\cdot u)\nu+\mu(\nabla u+(\nabla u)^\mathrm{T})\nu,
\end{align*}
and $\nu$ denotes the unit outer normal of $\partial D_{1}$. Here and below, we let the subscript $\pm$ represent the limit from outside and inside the domain, respectively. We would like to remark that there has established the existence, uniqueness and regularity of weak solutions to (\ref{La.002}) in \cite{BLL2015}. Moreover, the $H^{1}$-regularity of solution $u$ to problem (\ref{La.002}) is improved to that of $C^1(\overline{\Omega};\mathbb{R}^{d})\cap C^1(\overline{D}_{1};\mathbb{R}^{d})$.

Suppose further that there exists a positive constant $R$, independent of $\varepsilon$, such that the top and bottom boundaries of the thin gap between $D_{1}$ and $D$ can be formulated, respectively, by
\begin{align*}
x_{d}=\varepsilon+h_{1}(x')\quad\mathrm{and}\quad x_{d}=h(x'),\quad\quad x'\in B_{2R}',
\end{align*}
where $h_{1}$ and $h$ verify that for $\beta>0$,
\begin{enumerate}
{\it\item[(\bf{S1})]
$h_{1}(x')-h(x')=\tau|x'|^{1+\alpha}+O(|x'|^{1+\alpha+\beta}),\;\mbox{if}\;\,x'\in B'_{2R},$
\item[(\bf{S2})]
$|\nabla_{x'}h_{1}(x')|,\,|\nabla_{x'}h(x')|\leq \kappa_{1}|x'|^{\alpha},\;\mbox{if}\;\,x'\in B_{2R}',$
\item[(\bf{S3})]
$\|h_{1}\|_{C^{1,\alpha}(B'_{2R})}+\|h\|_{C^{1,\alpha}(B'_{2R})}\leq \kappa_{2},$}
\end{enumerate}
where $\tau,\kappa_{1}$ and $\kappa_{2}$ are all positive constants independent of $\varepsilon$. We additionally assume that for $i=1,...,d-1,$ $h_{1}(x')-h(x')$ is an even function of each $x_{i}$ in $B_{R}'$. We here would like to remark that condition ({\bf{S1}}) allows $\partial D_{1}$ and $\partial D$ to have different convexity, such as $h_{1}(x')=\tau_{1}|x'|^{1+\alpha}$ and $h(x')=\tau_{2}|x'|^{1+\alpha+\beta}$ in $B'_{2R}$ with two positive constants $\tau_{1}$ and $\tau_{2}$ independent of $\varepsilon$.

For $z'\in B'_{R}$ and $0<t\leq2R$, denote
\begin{align*}
\Omega_{t}(z'):=&\left\{x\in \mathbb{R}^{d}~\big|~h(x')<x_{d}<\varepsilon+h_{1}(x'),~|x'-z'|<{t}\right\}.
\end{align*}
For simplicity, we employ the abbreviated notation $\Omega_{t}$ to represent the thin gap $\Omega_{t}(0')$. For $(x',x_{d})\in\Omega_{2R}$, write
\begin{align}\label{KLO011}
\delta(x'):=\varepsilon+h_{1}(x')-h(x').
\end{align}
Introduce a scalar auxiliary function $\bar{v}\in C^{1,\alpha}(\mathbb{R}^{d})$ satisfying $\bar{v}=1$ on $\partial D_{1}$, $\bar{v}=0$ on $\partial D$,
\begin{align*}
%\label{GHA001}
\bar{v}(x',x_{d}):=\frac{x_{d}-h(x')}{\delta(x')},\;\,\mathrm{in}\;\Omega_{2R},\quad\mbox{and}~\|\bar{v}\|_{C^{1,\alpha}(\Omega\setminus\Omega_{R})}\leq C.
\end{align*}
For $i=1,2,...,\frac{d(d+1)}{2}$, define
\begin{align}\label{GWQ}
\bar{u}_{i}(x',x_{d}):=\psi_{i}\bar{v},\quad\bar{u}_{0}(x',x_{d}):=\varphi(x',h(x'))(1-\bar{v}).
\end{align}

\subsection{Main results}
Before listing our main results, we first give some notations. Set
\begin{align*}
\Gamma_{\alpha}:=&\Gamma\left(\frac{1}{1+\alpha}\right)\Gamma\left(\frac{\alpha}{1+\alpha}\right),
\end{align*}
where $\Gamma(s)=\int^{+\infty}_{0}t^{s-1}e^{-t}dt$, $s>0$ is the Gamma function. Introduce a definite
constant as follows:
\begin{align}\label{zwzh001}
\mathcal{M}_{\alpha,\tau}=\frac{2\Gamma_{\alpha}}{(1+\alpha)\tau^{\frac{1}{1+\alpha}}},
\end{align}
where $\tau$ is defined in condition ({\bf{S1}}). Denote
\begin{align}\label{AZ}
(\mathcal{L}_{d}^{1},...,\mathcal{L}_{d}^{d-1},\mathcal{L}_{d}^{d})=(\mu,...,\mu,\lambda+2\mu).
\end{align}
Furthermore, we suppose that for some $\kappa_{3}>0$,
\begin{align}\label{LAPXN001}
\kappa_{3}\leq\mu,d\lambda+2\mu\leq\frac{1}{\kappa_{3}}.
\end{align}

In this paper, we suppose that $\varphi\in C^{1,\alpha}(\partial D;\mathbb{R}^{d})$. Without loss of generality, let $\varphi(0)=0$. Otherwise, we substitute $u-\varphi(0)$ for $u$. The key to establish the asymptotic formula of the stress $\nabla u$ lies in extracting the blow-up factors independent of $\varepsilon$. For that purpose, we focus on the case when the distance between the inclusion $D_{1}$ and the external boundary $\partial{D}$ becomes zero. Introduce a family of bounded linear functionals in relation to $\varphi$,
\begin{equation}\label{BF001}
Q_{j}^{\ast}[\varphi]:=\int_{\partial D_{1}^{\ast}}\frac{\partial u_{0}^{\ast}}{\partial\nu_{0}}\Big|_{+}\cdot\psi_{j},\quad\quad j=1,2,...,\frac{d(d+1)}{2},
\end{equation}
where $u_{0}^{\ast}\in C^{1}(\overline{\Omega^{\ast}};\mathbb{R}^{d})\cap C^{2}(\Omega^{\ast};\mathbb{R}^{d})$ verifies
\begin{align*}
\begin{cases}
\mathcal{L}_{\lambda,\mu}u_{0}^{\ast}=0,\quad\quad\;\,&\mathrm{in}\;\,\Omega^{\ast}:=D\setminus\overline{D_{1}^{\ast}},\\
u_{0}^{\ast}=0,\quad\quad\;\,&\mathrm{on}\;\,\partial D_{1}^{\ast}\setminus\{0\},\\
u_{0}^{\ast}=\varphi(x),\quad\;\,&\mathrm{on}\;\,\partial D.
\end{cases}
\end{align*}
For $i,j=1,2,...,\frac{d(d+1)}{2}$, define
\begin{align}\label{KRD0}
a_{ij}^{\ast}=\int_{\Omega^{\ast}}(\mathbb{C}^{0}e(u_{i}^{\ast}),e(u_{j}^{\ast})),
\end{align}
where, for $i=1,2,...,\frac{d(d+1)}{2}$, $u_{i}^{\ast}$ solves
\begin{equation}\label{l03.001}
\begin{cases}
\mathcal{L}_{\lambda,\mu}u_{i}^{\ast}=0,\quad\;\,&\mathrm{in}\;\,\Omega^{\ast},\\
u_{i}^{\ast}=\psi_{i},\quad\;\,&\mathrm{on}\;\,\partial D_{1}^{\ast}\setminus\{0\},\\
u_{i}^{\ast}=0,\quad\;\,&\mathrm{on}\;\,\partial D.
\end{cases}
\end{equation}
Note that the definition of $a_{ij}^{\ast}$ in \eqref{KRD0} is valid for any $i,j=1,2,...,\frac{d(d+1)}{2}$ in three dimensions but only valid under some cases in two dimensions, see Lemma \ref{lemmabc} below. In dimension two, define
\begin{align}\label{BF002}
\mathbb{B}_{i}^{\ast}[\varphi]=\begin{pmatrix} Q_{i}^{\ast}[\varphi]&a^{\ast}_{i3}\\  Q_{3}^{\ast}[\varphi]&a_{33}^{\ast}\\  \end{pmatrix},\quad i=1,2.
\end{align}
For the order of the rest term, define
\begin{align}\label{ZWZHAO01}
\varepsilon(\alpha,\beta):=&
\begin{cases}
\varepsilon^{\min\{\frac{\beta}{1+\alpha},\frac{(1-\alpha)\alpha}{2(1+2\alpha)},\frac{\alpha^{2}}{2(1+2\alpha)(1+\alpha)^{2}}\}},&\alpha>\beta,\\
\varepsilon^{\min\{\frac{(1-\alpha)\alpha}{2(1+2\alpha)},\frac{\alpha^{2}}{2(1+2\alpha)(1+\alpha)^{2}}\}},&0<\alpha\leq\beta.
\end{cases}
%\bar{\varepsilon}(\alpha,\beta):=&
%\begin{cases}
%\varepsilon^{\min\{\frac{\beta}{1+\alpha},\frac{\alpha}{2(1+2\alpha)},\frac{(2-\alpha)(1+\alpha)}{3(1+2\alpha)}\}},&\alpha>\beta,\\
%\varepsilon^{\min\{\frac{\alpha}{2(1+2\alpha)},\frac{(2-\alpha)(1+\alpha)}{3(1+2\alpha)}\}},&0<\alpha\leq\beta.
%\end{cases}\label{ZWZHAO02}
\end{align}

Unless otherwise stated, in what following $C$ denotes a constant, whose values may vary from line to line, depending only on $\tau,\kappa_{1},\kappa_{2},\kappa_{3},R$ and an upper bound of the $C^{1,\alpha}$ norms of $\partial D_{1}$ and $\partial D$, but not on $\varepsilon$. $O(1)$ denotes some quantity satisfying  $|O(1)|\leq\,C$. Observe that by using the standard elliptic theory (see \cite{ADN1959,ADN1964}), we know
\begin{align*}
\|\nabla u\|_{L^{\infty}(\Omega\setminus\Omega_{R})}\leq\,C\|\varphi\|_{C^{1,\alpha}(\partial D)}.
\end{align*}
Based on this fact, it is sufficient to study the asymptotic behavior of $\nabla u$ in the thin gap $\Omega_{R}$. The first principal result is listed as follows.

\begin{theorem}\label{Lthm066}
Assume that $D_{1}\subset D\subseteq\mathbb{R}^{2}$ are defined as above, conditions $\rm{(}${\bf{S1}}$\rm{)}$--$\rm{(}${\bf{S3}}$\rm{)}$ hold, $Q_{3}^{\ast}[\varphi]\neq0$ and $\det\mathbb{B}_{i}^{\ast}[\varphi]\neq0$, $i=1,2$. For $\varphi\in C^{1,\alpha}(\partial D;\mathbb{R}^{2})$, let $u\in H^{1}(D;\mathbb{R}^{2})\cap C^{1}(\overline{\Omega};\mathbb{R}^{2})$ be the solution of (\ref{La.002}). Then for a sufficiently small $\varepsilon>0$ and $x\in\Omega_{R}$,
\begin{align*}
\nabla u=&\sum^{2}_{i=1}\frac{\det\mathbb{B}_{i}^{\ast}[\varphi]}{a_{33}^{\ast}}\frac{\varepsilon^{\frac{\alpha}{1+\alpha}}(1+O(\varepsilon(\alpha,\beta)))}{\mathcal{L}_{2}^{i}\mathcal{M}_{\alpha,\tau}}\nabla\bar{u}_{i}+\frac{Q_{3}^{\ast}[\varphi]}{a_{33}^{\ast}}(1+O(\varepsilon^{\frac{\alpha}{2(1+2\alpha)}}))\nabla\bar{u}_{3}\notag\\
&+\nabla\bar{u}_{0}+O(1)\delta^{-\frac{1-\alpha}{1+\alpha}}\|\varphi\|_{C^{1}(\partial D)},
\end{align*}
where $\delta$ is defined by \eqref{KLO011}, $\mathcal{M}_{\alpha,\tau}$ is defined in \eqref{zwzh001}, the explicit auxiliary functions $\bar{u}_{i}$, $i=0,1,2,3$ are defined by \eqref{GWQ}, the Lam\'{e} constants $\mathcal{L}_{2}^{i}$, $i=1,2$ are given in \eqref{AZ}, the blow-up factors $Q_{3}^{\ast}[\varphi]$ and $a_{33}^{\ast}$ are defined in \eqref{BF001}--\eqref{KRD0}, $\det\mathbb{B}_{i}^{\ast}[\varphi]$ is the determinant of the blow-up factor matrix $\mathbb{B}_{i}^{\ast}[\varphi]$ defined in \eqref{BF002}, $\varepsilon(\alpha,\beta)$ is defined by \eqref{ZWZHAO01}.
\end{theorem}
\begin{remark}
To begin with, we see from decomposition \eqref{Le2.015} below that $\nabla u$ is divided into three parts as follows: $\sum^{d}_{i=1}C^{i}\nabla\bar{u}_{i}$, $\sum^{\frac{d(d+1)}{2}}_{i=d+1}C^{i}\nabla\bar{u}_{i}$ and $\nabla\bar{u}_{0}$. Note that for $i=1,2,...,\frac{d(d+1)}{2}$, the main singularity of $\nabla\bar{u}_{i}$ lies in $|\psi_{i}|\delta^{-1}$, and the major singularity of $\nabla\bar{u}_{0}$ is determined by $|\partial_{x_{d}}\bar{u}_{0}|\leq\frac{C\|\varphi\|_{C^{1}(\partial D)}|x'|}{\varepsilon+|x'|^{1+\alpha}}$ . Then the asymptotic results in Theorems \ref{Lthm066} and \ref{Lthm06666} indicate that the first part $\sum^{d}_{i=1}C^{i}\nabla\bar{u}_{i}$ blows up, respectively, at the rate of $\varepsilon^{-\frac{1}{1+\alpha}}$ and $\varepsilon^{-1}$ in $(d-1)$-dimensional ball $\{|x'|\leq\varepsilon^{\frac{1}{1+\alpha}}\}\cap\Omega$ in two dimensions and higher dimensions, while the blow-up rates of the latter two parts $\sum^{\frac{d(d+1)}{2}}_{i=d+1}C^{i}\nabla\bar{u}_{i}$ and $\nabla\bar{u}_{0}$ are no greater than $\varepsilon^{-\frac{\alpha}{1+\alpha}}$ on the cylinder surface $\{|x'|=\varepsilon^{\frac{1}{1+\alpha}}\}\cap\Omega$. Consequently, the maximal singularity of $\nabla u$ lies on the fist part and its blow-up rate is $\varepsilon^{-\frac{1}{1+\alpha}}$ if $d=2$ and $\varepsilon^{-1}$ if $d\geq3$, respectively.
\end{remark}
\begin{remark}
For every $j=1,2,...,\frac{d(d+1)}{2}$, the blow-up factor $Q_{j}[\varphi]$ defined in \eqref{AKDM001} below remains bounded for any given boundary data $\varphi$ and converges to $Q_{j}^{\ast}[\varphi]$ defined by \eqref{BF001} as the distance $\varepsilon$ tends to zero. This is different from the results in \cite{LZ2020} for $C^{2,\alpha}$-inclusions close to the matrix boundary, where Li and Zhao \cite{LZ2020} found that the blow-up factor $Q_{j}[\varphi]$ will possess the blow-up rates for some special boundary data classified according to the parity and then strengthen the singularities of the stress. 
\end{remark}

In dimensions greater than or equal to three, we write
\begin{gather}\label{matrix01}
\mathbb{A}^{\ast}=
\begin{pmatrix} a_{11}^{\ast}&a_{12}^{\ast}&\cdots&a_{1\frac{d(d+1)}{2}}^{\ast} \\ a_{21}^{\ast}&a_{22}^{\ast}&\cdots&a_{2\frac{d(d+1)}{2}}^{\ast} \\ \vdots&\vdots&\ddots&\vdots\\a_{\frac{d(d+1)}{2}1}^{\ast}&a_{\frac{d(d+1)}{2}2}^{\ast}&\cdots&a_{\frac{d(d+1)}{2}\frac{d(d+1)}{2}}^{\ast}
\end{pmatrix}.
\end{gather}
For $i=1,2,...,\frac{d(d+1)}{2}$, after replacing the elements of $i$-th column in the matrix $\mathbb{A}^{\ast}$ by column vector $(Q_{1}^{\ast}[\varphi],Q_{2}^{\ast}[\varphi],...,Q^{\ast}_{\frac{d(d+1)}{2}}[\varphi])^{T}$, we get the new matrix $\mathbb{F}_{i}^{\ast}[\varphi]$ as follows:
\begin{gather}\label{LATP001}
\mathbb{F}_{i}^{\ast}[\varphi]=
\begin{pmatrix}
a_{11}^{\ast}&\cdots&Q_{1}^{\ast}[\varphi]&\cdots&a_{1\frac{d(d+1)}{2}}^{\ast} \\\\ \vdots&\ddots&\vdots&\ddots&\vdots\\\\a_{\frac{d(d+1)}{2}1}^{\ast}&\cdots&Q^{\ast}_{\frac{d(d+1)}{2}}[\varphi]&\cdots&a^{\ast}_{\frac{d(d+1)}{2}\frac{d(d+1)}{2}}
\end{pmatrix}.
\end{gather}
Denote
\begin{align}\label{NZKL001}
\bar{\varepsilon}(\alpha,d)=
\begin{cases}
\varepsilon^{\frac{\alpha^{2}(1-\alpha)}{2(1+2\alpha)(1+\alpha)^{2}}},&d=3,\\
\varepsilon^{\frac{\alpha^{2}}{2(1+2\alpha)(1+\alpha)^{2}}\min\{1+\alpha,2-\alpha\}},&d=4,\\
\varepsilon^{\frac{\alpha^{2}}{2(1+2\alpha)(1+\alpha)}},&d\geq5.
\end{cases}
\end{align}

Then we obtain the second principal result as follows.
\begin{theorem}\label{Lthm06666}
Assume that $D_{1}\subset D\subseteq\mathbb{R}^{d}\,(d\geq3)$ are defined as above, conditions $\rm{(}${\bf{S1}}$\rm{)}$--$\rm{(}${\bf{S3}}$\rm{)}$ hold, and $\det\mathbb{F}_{i}^{\ast}[\varphi]\neq0$, $i=1,2,...,\frac{d(d+1)}{2}$. For $\varphi\in C^{1,\alpha}(\partial D;\mathbb{R}^{d})$, let $u\in H^{1}(D;\mathbb{R}^{d})\cap C^{1}(\overline{\Omega};\mathbb{R}^{d})$ be the solution of (\ref{La.002}). Then for a sufficiently small $\varepsilon>0$ and $x\in\Omega_{R}$,
\begin{align*}
\nabla u=&\sum_{i=1}^{\frac{d(d+1)}{2}}\frac{\det\mathbb{F}_{i}^{\ast}[\varphi]}{\det\mathbb{A}^{\ast}}(1+O(\bar{\varepsilon}(\alpha,d)))\nabla\bar{u}_{i}+\nabla\bar{u}_{0}+O(1)\delta^{-\frac{1}{1+\alpha}}\|\varphi\|_{C^{1}(\partial D)},
\end{align*}
where $\delta$ is defined by \eqref{KLO011}, the explicit auxiliary functions $\bar{u}_{i}$, $i=0,1,...,\frac{d(d+1)}{2}$ are defined in \eqref{GWQ}, $\det\mathbb{A}^{\ast}$ and $\det\mathbb{F}_{i}^{\ast}[\varphi]$ are, respectively, the determinants of the blow-up factor matrices $\mathbb{A}^{\ast}$ and $\mathbb{F}_{i}^{\ast}[\varphi]$ defined in \eqref{matrix01}--\eqref{LATP001}, and $\bar{\varepsilon}(\alpha,d)$ is defined by \eqref{NZKL001}.
\end{theorem}
\begin{remark}
We claim that $\varphi\not\equiv0$ on $\partial D$ under the condition of $\det\mathbb{F}_{i}^{\ast}[\varphi]\neq0$, $i=1,2,...,\frac{d(d+1)}{2}$. In fact, if $\varphi\equiv0$ on $\partial D$, then it follows from integration by parts that $Q_{j}^{\ast}[\varphi]=\int_{\partial D}\frac{\partial u_{j}^{\ast}}{\partial\nu_{0}}\big|_{+}\cdot\varphi=0$, which contradicts the assumed condition.
\end{remark}

By applying the proofs of Theorems \ref{Lthm066} and \ref{Lthm06666} with a slight modification, we derive the pointwise upper and lower bounds on the gradients for more general $C^{1,\alpha}$-inclusions as follows:
\begin{align}
\tau_{1}|x'|^{1+\alpha}\leq h_{1}(x')-h(x')\leq \tau_{2}|x'|^{1+\alpha},\quad\mathrm{for}\;x'\in B'_{2R},\;\tau_{j}>0,\,j=1,2.\label{AKDN001}
%h_{1}(x')-h(x')\;\text{is an even function of each}\;x_{j},\quad j=1,...,d-1.\label{AKDN001001}
\end{align}
To be specific,
\begin{corollary}\label{HBZ001}
Assume that $D_{1}\subset D\subseteq\mathbb{R}^{d}\,(d\geq2)$ are defined as above, conditions \eqref{AKDN001} and $\rm{(}${\bf{S2}}$\rm{)}$--$\rm{(}${\bf{S3}}$\rm{)}$ hold. For $\varphi\in C^{1,\alpha}(\partial D;\mathbb{R}^{d})$, let $u\in H^{1}(D;\mathbb{R}^{d})\cap C^{1}(\overline{\Omega};\mathbb{R}^{d})$ be the solution of (\ref{La.002}). Then for a sufficiently small $\varepsilon>0$,

$(a)$ if $d=2$, there exist some integer $1\leq i_{0}\leq2$ such that $\det\mathbb{B}_{i_{0}}^{\ast}[\varphi]\neq0$, then for $x\in\{x'=0'\}\cap\Omega$,
\begin{align*}
\frac{|\det\mathbb{B}_{i_{0}}^{\ast}[\varphi]|}{C\mathcal{L}_{2}^{i_{0}}\tau_{2}^{\frac{1}{1+\alpha}}|a_{33}^{\ast}|}\frac{1}{\varepsilon^{\frac{1}{1+\alpha}}}\leq|\nabla u|\leq \frac{\max\limits_{1\leq i\leq 2}(\mathcal{L}_{2}^{i})^{-1}|\det\mathbb{B}_{i}^{\ast}[\varphi]|}{\tau_{1}^{\frac{1}{1+\alpha}}|a_{33}^{\ast}|}\frac{C}{\varepsilon^{\frac{1}{1+\alpha}}};
\end{align*}

$(b)$ if $d\geq3$, there exist some integer $1\leq i_{0}\leq d$ such that $\det\mathbb{F}_{i_{0}}^{\ast}[\varphi]\neq0$, then for  $x\in\{x'=0'\}\cap\Omega$,
\begin{align*}
\frac{|\det\mathbb{F}_{i_{0}}^{\ast}[\varphi]|}{|\det \mathbb{A}^{\ast}|}\frac{1}{C\varepsilon}\leq|\nabla u|\leq \frac{\max\limits_{1\leq i\leq d}|\det\mathbb{F}_{i}^{\ast}[\varphi]|}{|\det \mathbb{A}^{\ast}|}\frac{C}{\varepsilon},
\end{align*}
where the blow-up factor $a_{33}^{\ast}$ is defined in \eqref{KRD0} in the case of $d=2$, the blow-up factor matrices $\mathbb{B}_{i}^{\ast}[\varphi]$, $i=1,2$, $\mathbb{A}^{\ast}$ and $\mathbb{F}_{i}^{\ast}[\varphi]$, $i=1,2,...,d$ are defined by \eqref{BF002}, \eqref{matrix01} and \eqref{LATP001}, respectively.
\end{corollary}

\begin{remark}
The optimal upper and lower bounds on the gradient obtained in Corollary \ref{HBZ001} not only show the explicit dependence on the curvature parameters $\tau_{i}$, $i=1,2$ and the Lam\'{e} constants $\mathcal{L}_{2}^{i}$, $i=1,2$, but also on the explicit blow-up factor matrices. These information are not presented in the previous work \cite{BJL2017,LZ2020,CL2019}.

\end{remark}

\begin{figure}[htb]
\center{\includegraphics[width=0.45\textwidth]{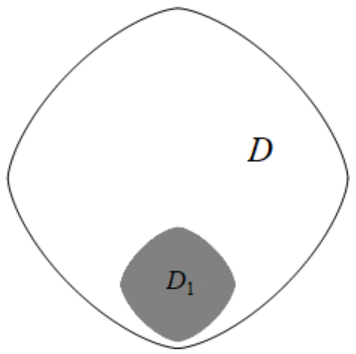}}
\caption{Curvilinear squares with rounded-off angles, $\alpha=\frac{1}{2}$}
\end{figure}

Finally, we consider a special core-shell geometry in dimension two, where the core and the surrounding shell are curvilinear squares with rounded-off angles, see Figure 1. Assume that there exist two positive constants $0<r_{1}<r_{2}$, independent of $\varepsilon$, such that the interfacial boundary of the inclusion $D_{1}$ and the external boundary $\partial D$ can be formulated as
\begin{align}\label{ZZW986}
|x_{1}|^{1+\alpha}+|x_{2}-\varepsilon-r_{1}|^{1+\alpha}=&r_{1}^{1+\alpha}\quad\mathrm{and}\quad|x_{1}|^{1+\alpha}+|x_{2}-r_{2}|^{1+\alpha}=r_{2}^{1+\alpha},
\end{align}
respectively. Denote
\begin{align}\label{AKLGJ001}
\tau_{0}=\frac{1}{1+\alpha}\left(\frac{1}{r_{1}^{\alpha}}-\frac{1}{r_{2}^{\alpha}}\right).
\end{align}

Then, we have
\begin{example}\label{CORO001}
Assume as above, condition \eqref{ZZW986} holds, $Q_{3}^{\ast}[\varphi]\neq0$ and $\det\mathbb{B}_{i}^{\ast}[\varphi]\neq0$, $i=1,2$. For $\varphi\in C^{1,\alpha}(\partial D;\mathbb{R}^{2})$, let $u\in H^{1}(D;\mathbb{R}^{2})\cap C^{1}(\overline{\Omega};\mathbb{R}^{2})$ be the solution of (\ref{La.002}). Then for a sufficiently small $\varepsilon>0$ and $x\in\Omega_{r_{0}}$, $0<r_{0}<\frac{1}{2}\min\{r_{1},r_{2}\}$ is a small constant independent of $\varepsilon$,
\begin{align}\label{LGT315LP}
\nabla u=&\sum^{2}_{i=1}\frac{\det\mathbb{B}_{i}^{\ast}[\varphi]}{a_{33}^{\ast}}\frac{\varepsilon^{\frac{\alpha}{1+\alpha}}}{\mathcal{L}_{2}^{i}\mathcal{M}_{\alpha,\tau_{0}}}\frac{1+O(\varepsilon^{\min\{\frac{\alpha^{2}}{2(1+2\alpha)(1+\alpha)^{2}},\frac{(1-\alpha)\alpha}{2(1+2\alpha)}\}})}{1+\mathcal{G}^{\ast}_{i}\varepsilon^{\frac{\alpha}{1+\alpha}}}\nabla\bar{u}_{i}\notag\\
&+\frac{Q_{3}^{\ast}[\varphi]}{a_{33}^{\ast}}(1+O(\varepsilon^{\frac{\alpha}{2(1+2\alpha)}}))\nabla\bar{u}_{3}+\nabla\bar{u}_{0}+O(1)\delta^{-\frac{1-\alpha}{1+\alpha}}\|\varphi\|_{C^{1}(\partial D)},
\end{align}
where $\delta$ is defined by \eqref{KLO011}, the explicit auxiliary functions $\bar{u}_{i}$, $i=0,1,2,3$ are defined by \eqref{GWQ}, $\mathcal{M}_{\alpha,\tau_{0}}$ is defined in \eqref{zwzh001} with $\tau=\tau_{0}$, the Lam\'{e} constants $\mathcal{L}_{2}^{i}$, $i=1,2$ are given in \eqref{AZ}, the blow-up factors $Q_{3}^{\ast}[\varphi]$ and $a_{33}^{\ast}$ are defined in \eqref{BF001}--\eqref{KRD0}, the blow-up factor matrices $\mathbb{B}_{i}^{\ast}[\varphi]$, $i=1,2$ are defined in \eqref{BF002}, the geometry constants $\mathcal{G}^{\ast}_{i}$, $i=1,2$ are defined by \eqref{QKLP001} below.

\end{example}

\begin{remark}
The geometry constants $\mathcal{G}^{\ast}_{1}$ and $\mathcal{G}^{\ast}_{2}$ captured in asymptotic expansion \eqref{LGT315LP} show the explicit dependence on the radii $r_{1}$ and $r_{2}$ and these geometry constants are independent of the distance parameter $\varepsilon$. In addition, the blow-up factor matrices $\mathbb{B}_{1}^{\ast}[\varphi]$ and $\mathbb{B}_{2}^{\ast}[\varphi]$ can be numerically calculated and analyzed for any given boundary data $\varphi$.

\end{remark}

\section{Preliminary}\label{prf results}
\subsection{Properties of the tensor $\mathbb{C}^{0}$}
To prove Theorem \ref{Lthm066}, we first make note of some properties of the tensor $\mathbb{C}^{0}$. For the isotropic elastic material, let
\begin{align}\label{coeffi001}
\mathbb{C}^{0}:=(C_{ijkl}^{0})=(\lambda\delta_{ij}\delta_{kl}+\mu(\delta_{ik}\delta_{jl}+\delta_{il}\delta_{jk})),\quad \mu>0,\quad d\lambda+2\mu>0,
\end{align}
where $C_{ijkl}^{0}$ satisfies the following symmetry property:
\begin{align}\label{symm}
C_{ijkl}^{0}=C_{klij}^{0}=C_{klji}^{0},\quad i,j,k,l=1,2,...,d.
\end{align}
For every pair of $d\times d$ matrices $\mathbb{A}=(a_{ij})$ and $\mathbb{B}=(b_{ij})$, write
\begin{align*}
(\mathbb{C}^{0}\mathbb{A})_{ij}=\sum_{k,l=1}^{n}C_{ijkl}^{0}a_{kl},\quad\hbox{and}\quad(\mathbb{A},\mathbb{B})\equiv \mathbb{A}:\mathbb{B}=\sum_{i,j=1}^{d}a_{ij}b_{ij}.
\end{align*}
Therefore,
$$(\mathbb{C}^{0}\mathbb{A},\mathbb{B})=(\mathbb{A}, \mathbb{C}^{0}\mathbb{B}).$$
From (\ref{symm}), we see that $\mathbb{C}^{0}$ verifies the ellipticity condition, that is, for every $d\times d$ real symmetric matrix $\xi=(\xi_{ij})$,
\begin{align}\label{ellip}
\min\{2\mu, d\lambda+2\mu\}|\xi|^2\leq(\mathbb{C}^{0}\xi, \xi)\leq\max\{2\mu, d\lambda+2\mu\}|\xi|^2,
\end{align}
where $|\xi|^2=\sum\limits_{ij}\xi_{ij}^2.$  In particular,
\begin{align*}
\min\{2\mu, d\lambda+2\mu\}|\mathbb{A}+\mathbb{A}^T|^2\leq(\mathbb{C}(\mathbb{A}+\mathbb{A}^T), (\mathbb{A}+\mathbb{A}^T)).
\end{align*}
In addition, we know that for any open set $O$ and $u, v\in C^2(O;\mathbb{R}^{d})$,
\begin{align}\label{Le2.01222}
\int_O(\mathbb{C}^0e(u), e(v))\,dx=-\int_O\left(\mathcal{L}_{\lambda, \mu}u\right)\cdot v+\int_{\partial O}\frac{\partial u}{\partial \nu_0}\Big|_{+}\cdot v.
\end{align}
\subsection{Solution decomposition}
As seen in \cite{LZ2020,BJL2017}, the solution of \eqref{La.002} can be split as follows:
\begin{align*}
u=\sum^{d}_{i=1}C^{i}u_{i}+\sum^{\frac{d(d+1)}{2}}_{i=d+1}C^{i}u_{i}+u_{0},\quad\;\,\mathrm{in}\;\Omega,
\end{align*}
where the free constants $C^{i},\,i=1,2,...,\frac{d(d+1)}{2}$ will be determined later by making use of the forth line in (\ref{La.002}), $u_{i}\in C^{1}(\overline{\Omega};\mathbb{R}^{d})\cap C^{2}(\Omega;\mathbb{R}^{d}),\,i=0,1,...,\frac{d(d+1)}{2}$, verify
\begin{equation}\label{P2.005}
\begin{cases}
\mathcal{L}_{\lambda,\mu}u_{0}=0,&\mathrm{in}\;\Omega,\\
u_{0}=0,&\mathrm{on}\;\partial D_{1},\\
u_{0}=\varphi(x),&\mathrm{on}\;\partial D,
\end{cases}\quad
\begin{cases}
\mathcal{L}_{\lambda,\mu}u_{i}=0,&\mathrm{in}\;\Omega,\\
u_{i}=\psi_{i},&\mathrm{on}\;\partial D_{1},\\
u_{i}=0,&\mathrm{on}\;\partial D,
\end{cases}
\end{equation}
respectively. Consequently,
\begin{align}\label{Le2.015}
\nabla u=\sum^{d}_{i=1}C^{i}\nabla u_{i}+\sum^{\frac{d(d+1)}{2}}_{i=d+1}C^{i}\nabla u_{i}+\nabla u_{0},\quad\;\,\mathrm{in}\;\,\Omega,
\end{align}
which indicates that the asymptotics of $\nabla u$ consist of the following two aspects of expansions:
\begin{itemize}
\item [$(i)$] asymptotics of $\nabla u_{i}$, $i=0,1,...,\frac{d(d+1)}{2}$;

\item[$(ii)$] asymptotics of $C^{i}$, $i=1,2,...,\frac{d(d+1)}{2}$.
\end{itemize}

\section{The proofs of Theorems \ref{Lthm066} and \ref{Lthm06666}}\label{pth results}
To begin with, we show that $\nabla\bar{u}_{0}$ is the leading term of $\nabla u_{0}$, where $\bar{u}_{0}$ is defined in \eqref{GWQ}. Note that the solution $u_{0}$ to problem \eqref{P2.005} can be further decomposed as follows:
$$u_{0}=\sum_{l=1}^{d}u_{0l},$$
where $u_{0l}$ satisfies
\begin{align}\label{RTP101}
\begin{cases}
\mathcal{L}_{\lambda,\mu}u_{0l}=0,\quad&
\hbox{in}\  \Omega,  \\
u_{0l}=0&\hbox{on}\ \partial{D}_{1},\\
u_{0l}=(0,...,0,\varphi^{l}(x),0,...,0)^{T},&\hbox{on} \ \partial{D}.
\end{cases}
\end{align}
First, we extend $\varphi\in C^{1,\alpha}(\partial D;\mathbb{R}^{d})$ to $\varphi\in C^{1,\alpha}(\overline{\Omega};\mathbb{R}^{d})$ verifying that $\|\varphi^{l}\|_{C^{1,\alpha}(\overline{\Omega\setminus\Omega_{R}})}\leq C\|\varphi^{l}\|_{C^{1,\alpha}(\partial D)}$, $l=1,2,...,d.$ Construct a smooth cutoff function $\rho\in C^{1,\alpha}(\overline{\Omega})$ such that $0\leq\rho\leq1$, $|\nabla\rho|\leq C$ in $\overline{\Omega}$, and
\begin{align*}
\rho=1\;\mathrm{in}\;\Omega_{\frac{3}{2}R},\quad\rho=0\;\mathrm{in}\;\overline{\Omega}\setminus\Omega_{2R}.
\end{align*}
For $x\in\Omega$, define
\begin{align*}
\bar{u}_{0l}(x):=(0,...,0,[\rho(x)\varphi^{l}(x',h(x'))+(1-\rho(x))\varphi^{l}(x)](1-\bar{v}(x)),0,...,0)^{T}.
\end{align*}
Then, we have
\begin{align*}
\nabla\bar{u}_{0}=\sum_{l=1}^{d}\nabla\bar{u}_{0l},\quad\mathrm{in}\;\Omega_{R},
\end{align*}
where
\begin{align*}
\bar{u}_{0l}(x):=(0,...,0,\varphi^{l}(x',h(x'))(1-\bar{v}(x)),0,...,0)^{T},\quad\mathrm{in}\;\Omega_{R}.
\end{align*}

\begin{theorem}\label{lem89999}
Assume as in Theorems \ref{Lthm066} and \ref{Lthm06666}. Let $u_{0l},\,l=1,2,...,d$ be the weak solution of (\ref{RTP101}). Assume that $\|\varphi\|_{C^{1}(\partial D)}>0$. Then for a sufficiently small $\varepsilon>0$, $l=1,2,...,d$,
\begin{align*}
\nabla u_{0l}=\nabla\bar{u}_{0l}+O(1)\delta^{-\frac{1}{1+\alpha}}\big(|\varphi^{l}(x',h(x'))|+\delta^{\frac{1}{1+\alpha}}\|\varphi^{l}\|_{C^{1}(\partial D)}\big),\quad \mathrm{in}~\Omega_{R}.
\end{align*}
Consequently,
\begin{align}\label{RTP1311}
\nabla u_{0}=\nabla\bar{u}_{0}+O(1)\delta^{-\frac{1}{1+\alpha}}\big(|\varphi(x',h(x'))|+\delta^{\frac{1}{1+\alpha}}\|\varphi\|_{C^{1}(\partial D)}\big),\quad \mathrm{in}~\Omega_{R},
\end{align}
where $\bar{u}_{0}$ is defined in \eqref{GWQ}.
\end{theorem}
\begin{remark}
Due to the assumption of $\varphi(0)=0$ above, we refine the expansion \eqref{RTP1311} as follows:
\begin{align*}
\nabla u_{0}=\nabla\bar{u}_{0}+O(1)\|\varphi\|_{C^{1}(\partial D)},\quad \mbox{in}~\Omega_{R}.
\end{align*}
\end{remark}

In the following, we will use an adapted version of the iteration technique with respect to the energy to prove Theorem \ref{lem89999}, which was developed in \cite{CL2019} by combining the Campanato's approach and $W^{1,p}$ estimates for elliptic systems with right hand side in divergence form. To begin with, we state the following two lemmas, which are Theorem 2.3 and Theorem 2.4 in \cite{CL2019}. For simplicity, in this section we denote $\partial_{j}:=\partial_{x_{j}}$, $j=1,2,...,d$. Let $Q$ be a bounded domain in $\mathbb{R}^{d}$, $d\geq2$, with $C^{1,\alpha}\,(0<\alpha<1)$ boundary portion $\Gamma\subset\partial Q$. Consider the boundary value problem as follows:
\begin{align}\label{ADCo1}
\begin{cases}
-\partial_{j}(C_{ijkl}^{0}\partial_{l}w^{k})=\partial_{j}f_{ij},&in\; Q,\\
w=0,&on\;\Gamma,
\end{cases}
\end{align}
where $f_{ij}\in C^{0,\alpha}(Q)$, $i,j,k,l=1,2,...,d$, and the Einstein summation convention in repeated indices is used.
\begin{lemma}\label{CL001}
Let $w\in H^{1}(Q;\mathbb{R}^{d})\cap C^{1}(Q\cup\Gamma;\mathbb{R}^{d})$ be the solution of \eqref{ADCo1}. Then for any subdomain $Q'\subset\subset Q\cup\Gamma$,
\begin{align}\label{GNA001}
\|w\|_{C^{1,\alpha}(Q')}\leq C\left(\|w\|_{L^{\infty}(Q)}+[F]_{\alpha,Q}\right),
\end{align}
where $F:=(f_{ij})$ and $C=C(d,\alpha,Q',Q)$.
\end{lemma}
The H\"{o}lder semi-norm of matrix-valued function $F=(f_{ij})$ is defined as follows:
\begin{align*}
[F]_{\alpha,Q}:=\max_{1\leq i,j\leq d}[f_{ij}]_{\alpha,Q}\quad\mathrm{and}\quad[f_{ij}]_{\alpha,Q}=\sup_{x,y\in Q,x\neq y}\frac{|f_{ij}(x)-f_{ij}(y)|}{|x-y|^{\alpha}}.
\end{align*}

\begin{lemma}\label{CL002}
Assume that $Q$ and $\Gamma$ are defined as in Lemma \ref{CL001}. Let $w\in H^{1}(Q;\mathbb{R}^{d})$ be the weak solution of \eqref{ADCo1} with $f_{ij}\in C^{0,\alpha}(Q)$, $0<\alpha<1$ and $i,j=1,2,...,d$. Then, for any $2\leq p<\infty$ and $Q'\subset\subset Q\cup\Gamma$,
\begin{align}\label{LNZ001}
\|w\|_{W^{1,p}(Q')}\leq C(\|w\|_{H^{1}(Q)}+\|F\|_{L^{p}(Q)}),
\end{align}
where $C=C(\lambda,\mu,p,Q')$ and $F:=(f_{i}^{k})$. In particular, if $p>d$, we have
\begin{align}\label{LNZ002}
\|w\|_{C^{0,\gamma}(Q')}\leq C(\|w\|_{H^{1}(Q)}+[F]_{\alpha,Q}),
\end{align}
where $0<\gamma\leq1-d/p$ and $C=C(\lambda,\mu,\gamma,p,Q')$.
\end{lemma}
For readers' convenience, the detailed proofs of Lemmas \ref{CL001} and \ref{CL002} are left in the Appendix.

\begin{proof}[The proof of Theorem \ref{lem89999}]
Take $l=1$ for example. Other cases are the same. For simplicity, we denote
$$w:=u_{01}-\bar{u}_{01},$$
where $\bar{u}_{1}$ is defined in \eqref{GWQ}. Then $w$ solves
\begin{align}\label{Zww01}
\begin{cases}
-\mathcal{L}_{\lambda, \mu}w=\nabla\cdot(\mathbb{C}^0e(\bar{u}_{01})),\quad&\hbox{in}\ \Omega,\\
w=0,&\hbox{on}\ \partial{\Omega}.
\end{cases}
\end{align}
Obviously, $w$ also satisfies that for any constant matrix $\mathcal{M}=(\mathfrak{a}_{ij})$,
\begin{align}\label{ZWZW001}
-\mathcal{L}_{\lambda, \mu}w=\nabla\cdot(\mathbb{C}^0e(\bar{u}_{01})-\mathcal{M}),\quad&\hbox{in}\ \Omega.
\end{align}

We next divide into three parts to prove Theorem \ref{lem89999}. For simplicity, we use $\|\varphi^{1}\|_{C^{1}}$ to denote $\|\varphi^{1}\|_{C^{1}(\partial D)}$ in this section.

{\bf Part 1.} Proof of
\begin{align}\label{zzwad01}
\|\nabla w\|_{L^{2}(\Omega)}\leq C\|\varphi^{1}\|_{C^{1}}.
\end{align}
From \eqref{Zww01}, we see
\begin{align}\label{ZH001}
\int_{\Omega}(\mathbb{C}^{0}e(w),e(w))\,dx=-\int_{\Omega}(\mathbb{C}^{0}e(\bar{u}_{01}),e(w))\,dx.
\end{align}
On one hand, making use of \eqref{ellip} and the first Korn's inequality, we obtain
\begin{equation}\label{def_w}
\int_{\Omega}(\mathbb{C}^{0}e(w),e(w))\,dx\geq\frac{1}{C}\int_{\Omega}|e(w)|^{2}dx\geq\frac{1}{C}\int_{\Omega}|\nabla w|^{2}dx.
\end{equation}

On the other hand, it follows from the H\"{o}lder inequality that
\begin{align*}
\left|\int_{\Omega}(\mathbb{C}^{0}e(\bar{u}_{01}),e(w))\,dx\right|\leq&\left|\int_{\Omega_{R}}(\mathbb{C}^{0}e(\bar{u}_{01}),e(w))\,dx\right|+\left|\int_{\Omega\setminus\Omega_{R}}(\mathbb{C}^{0}e(\bar{u}_{01}),e(w))\,dx\right|\\
\leq&\left|\int_{\Omega_{R}}(\mathbb{C}^{0}e(\bar{u}_{01}),e(w))\,dx\right|+C\int_{\Omega\setminus\Omega_{R}}|\nabla\bar{u}_{01}||\nabla w|\,dx\\
\leq&\left|\int_{\Omega_{R}}(\mathbb{C}^{0}e(\bar{u}_{01}),e(w))\,dx\right|+C\|\varphi^{1}\|_{C^{1}}\|\nabla w\|_{L^{2}(\Omega\setminus\Omega_{R})}.
\end{align*}
Denote
\begin{align*}
\mathcal{A}=\int_{\Omega_{R}}(\mathbb{C}^{0}e(\bar{u}_{01}),e(w))\,dx.
\end{align*}
Recalling the definitions of $\mathbb{C}^{0}$ and $\bar{u}_{01}$, it follows from a direct calculation that
\begin{align*}
(\mathbb{C}^{0}e(\bar{u}_{01}),e(w))=&\lambda\partial_{1}\bar{u}_{01}^{1}\partial_{1}w^{1}+\mu\sum^{d}_{i=1}(\partial_{i}w^{1}+\partial_{1}w^{i})\partial_{i}\bar{u}_{01}^{1},
\end{align*}
where $\bar{u}_{01}^{1}=\varphi^{1}(x',h(x'))(1-\bar{v}).$ Then we decompose $\mathcal{A}$ into two parts as follows:
\begin{align*}
\mathcal{A}_{1}=&\int_{\Omega_{R}}\lambda\partial_{1}\bar{u}_{01}^{1}\partial_{1}w^{1}+\mu\sum^{d-1}_{i=1}(\partial_{i}w^{1}+\partial_{1}w^{i})\partial_{i}\bar{u}_{01}^{1},\\
\mathcal{A}_{2}=&\int_{\Omega_{R}}\mu(\partial_{d}w^{1}+\partial_{1}w^{d})\partial_{d}\bar{u}_{01}^{1}.
\end{align*}
From the H\"{o}lder inequality, we have
\begin{align}\label{ZH002}
|\mathcal{A}_{1}|\leq&C\|\nabla_{x'}\bar{u}_{01}^{1}\|_{L^{2}(\Omega_{R})}\|\nabla w\|_{L^{2}(\Omega_{R})}\leq C\|\varphi^{1}\|_{C^{1}}\|\nabla w\|_{L^{2}(\Omega_{R})},
\end{align}
while, in view of $\partial_{dd}\bar{v}=0$ in $\Omega_{R}$, it follows from the Sobolev trace embedding and integration by parts that
\begin{align}\label{ZH003}
|\mathcal{A}_{2}|\leq&\int\limits_{\scriptstyle |x'|={R},\atop\scriptstyle
h(x')<x_{d}<\varepsilon+h_1(x')\hfill}\mu\left| w^{1}\partial_{d}\bar{u}_{01}^{1}\nu_{d}+w^{d}\partial_{d}\bar{u}_{01}^{1}\nu_{1}-w^{d}\partial_{1}\bar{u}_{01}^{1}\nu_{d}\right|+\int_{\Omega_{R}}\left|\mu\partial_{1}\bar{u}_{01}^{1}\partial_{d}w^{d}\right|\notag\\
\leq&\int\limits_{\scriptstyle |x'|={R},\atop\scriptstyle
h(x')<x_{d}<\varepsilon+h_1(x')\hfill}C\|\varphi^{1}\|_{C^{1}}|w|+C\|\partial_{1}\bar{u}_{01}^{1}\|_{L^{2}(\Omega_{R})}\|\nabla w\|_{L^{2}(\Omega_{R})}\notag\\
\leq&C\|\varphi^{1}\|_{C^{1}}\|\nabla w\|_{L^{2}(\Omega)}.
\end{align}
Then combining with \eqref{ZH001}--\eqref{ZH003}, we arrive at
\begin{align*}
\int_{\Omega}|\nabla w|^{2}dx\leq&C\|\varphi^{1}\|_{C^{1}}\left(\int_{\Omega}|\nabla w|^{2}dx\right)^{\frac{1}{2}}.
\end{align*}
That is, \eqref{zzwad01} is proved.

{\bf Part 2.}
Proof of
\begin{align}\label{step2}
 \int_{\Omega_\delta(z')}|\nabla w|^2dx
 &\leq C\delta^{d-\frac{2}{1+\alpha}}\big(|\varphi^{1}(z',h(z'))|^2+\delta^{\frac{2}{1+\alpha}}\|\varphi^{1}\|_{C^{1}}^2\big),\quad|z'|\leq R.
\end{align}
For $0<t<s<R$, introduce a smooth cutoff function $\eta$ such that $0\leq\eta(x')\leq1$, $\eta(x')=1$ if $|x'-z'|<t$, $\eta(x')=0$ if $|x'-z'|>s$, and $|\nabla\eta(x')|\leq\frac{2}{s-t}$. Multiplying equation \eqref{ZWZW001} by $w\eta^{2}$, it follows from integration by parts that
\begin{align}\label{LMQ001}
\int_{\Omega_{s}(z')}(\mathbb{C}^{0}e(w),e(w\eta^{2}))\,dx=-\int_{\Omega_{s}(z')}(\mathbb{C}^{0}e(\bar{u}_{01})-\mathcal{M},e(w\eta^{2}))\,dx.
\end{align}
For the left hand side in \eqref{LMQ001}, it follows from \eqref{LAPXN001}, \eqref{ellip} and the first Korn's inequality that
\begin{align}\label{KLW01}
\int_{\Omega_{s}(z')}(\mathbb{C}^{0}e(w),e(w\eta^{2}))\,dx\geq&\frac{1}{C}\int_{\Omega_{s}(z')}|\nabla(w\eta)|^{2}dx-C\int_{\Omega_{s}(z')}|w|^{2}|\nabla\eta|^{2}dx,
\end{align}
while, for the right hand side in \eqref{LMQ001}, we see from the Young's inequality that for any $\zeta>0$,
\begin{align}\label{KLW02}
\left|\int_{\Omega_{s}(z')}(\mathbb{C}^{0}e(\bar{u}_{01})-\mathcal{M},e(w\eta^{2}))\,dx\right|\leq&\zeta\int_{\Omega_{s}(z')}\eta^{2}|\nabla w|^{2}dx+C\int_{\Omega_{s}(z')}|\nabla\eta|^{2}|w|^{2}dx\notag\\
&+\frac{C}{\zeta}\int_{\Omega_{s}(z')}|\mathbb{C}^{0}e(\bar{u}_{01})-\mathcal{M}|^{2}dx.
\end{align}
Then combining \eqref{KLW01} and \eqref{KLW02}, we have
\begin{align*}
\int_{\Omega_{t}(z')}|\nabla w|^{2}dx\leq\frac{C}{(s-t)^{2}}\int_{\Omega_{s}(z')}|w|^{2}dx+C\int_{\Omega_{s}(z')}|\mathbb{C}^{0}e(\bar{u}_{01})-\mathcal{M}|^{2}dx.
\end{align*}
Set
\begin{align*}
\mathcal{M}=\frac{1}{|\Omega_{s}(z')|}\int_{\Omega_{s}(z')}\mathbb{C}^{0}e(\bar{u}_{01}(y))\,dy.
\end{align*}

For $|z'|\leq R$, $0<s\leq\vartheta(\tau,\kappa_{1})\delta^{\frac{1}{1+\alpha}}$, $\vartheta(\tau,\kappa_{1})=\frac{1}{8\kappa_{1}\max\{1,\tau^{-\frac{\alpha}{1+\alpha}}\}}$, it follows from conditions ({\bf{S1}}) and ({\bf{S2}}) that for $(x',x_{d})\in\Omega_{s}(z')$,
\begin{align}\label{ASK001}
|\delta(x')-\delta(z')|\leq&|h_{1}(x')-h_{1}(z')|+|h(x')-h(z')|\notag\\
\leq&(|\nabla_{x'}h_{1}(x'_{\theta_{1}})|+|\nabla_{x'}h(x'_{\theta})|)|x'-z'|\notag\\
\leq&\kappa_{1}|x'-z'|(|x'_{\theta_{1}}|^{\alpha}+|x'_{\theta}|^{\alpha})\notag\\
\leq&2\kappa_{1}s(s^{\alpha}+|z'|^{\alpha})\notag\\
\leq&\frac{\delta(z')}{2}.
\end{align}
Then, we have
\begin{align}\label{QWN001}
\frac{1}{2}\delta(z')\leq\delta(x')\leq\frac{3}{2}\delta(z'),\quad\mathrm{in}\;\Omega_{s}(z').
\end{align}
In light of \eqref{QWN001}, it follows from a direct calculation that
\begin{align}\label{QWN002}
[\nabla\bar{u}_{01}]_{\alpha,\Omega_{s}(z')}\leq C\big(|\varphi^{1}(z',h(z'))|\delta^{-\frac{2+\alpha}{1+\alpha}}+\|\varphi^{1}\|_{C^{1}}\delta^{-1}\big)s^{1-\alpha}.
\end{align}
Since $w=0$ on $\partial\Omega$, we see from \eqref{QWN001} and \eqref{QWN002} that
\begin{align}\label{ADE007}
\int_{\Omega_{s}(z')}|w|^{2}\leq C\delta^{2}\int_{\Omega_{s}(z')}|\nabla w|^{2},
\end{align}
and
\begin{align}\label{ADE006}
&\int_{\Omega_{s}(z')}|\mathbb{C}^{0}e(\bar{u}_{01})-\mathcal{M}|^{2}dx\leq  Cs^{d+1}\delta^{-\frac{3+\alpha}{1+\alpha}}\big(|\varphi^{1}(z',h(z'))|^{2}+\delta^{\frac{2}{1+\alpha}}\|\varphi^{1}\|_{C^{1}}^{2}\big).
\end{align}

Denote
\begin{align*}
F(t):=\int_{\Omega_{t}(z')}|\nabla w|^{2}.
\end{align*}
Then we know from (\ref{ADE007})--(\ref{ADE006}) that
\begin{align}\label{ADE008}
F(t)\leq \left(\frac{c\delta}{s-t}\right)^2F(s)+Cs^{d+1}\delta^{-\frac{3+\alpha}{1+\alpha}}\big(|\varphi^{1}(z',h(z'))|^{2}+\delta^{\frac{2}{1+\alpha}}\|\varphi^{1}\|_{C^{1}}^{2}\big),
\end{align}
where $c$ and $C$ are universal constants.

Choose $k=\left[\frac{1}{4c\delta^{\frac{\alpha}{2(1+\alpha)}}}\right]+1$ and $t_{i}=\delta+2ci\delta,\;i=0,1,2,...,k$. Then, (\ref{ADE008}), together with $s=t_{i+1}$ and $t=t_{i}$, reads that
$$F(t_{i})\leq\frac{1}{4}F(t_{i+1})+C(i+1)^{n+1}\delta^{d-\frac{2}{1+\alpha}}\big(|\varphi^{1}(z',h(z'))|^{2}+\delta^{\frac{2}{1+\alpha}}\|\varphi^{1}\|^{2}_{C^{1}}\big).$$
After $k$ iterations, it follows from (\ref{zzwad01}) that for a sufficiently small $\varepsilon>0$,
\begin{align*}
F(t_{0})\leq C\delta^{d-\frac{2}{1+\alpha}}\big(|\varphi^{1}(z',h(z'))|^{2}+\delta^{\frac{2}{1+\alpha}}\|\varphi^{1}\|^{2}_{C^{1}}\big).
\end{align*}

{\bf Part 3.}
Proof of
\begin{align*}
|\nabla w(x)|\leq C\delta^{-\frac{1}{1+\alpha}}\big(|\varphi^{1}(x',h(x'))|+\delta^{\frac{1}{1+\alpha}}\|\varphi^{1}\|_{C^{1}}\big),\quad\mathrm{in}\;\Omega_{R}.
\end{align*}

Making use of a change of variables in $\Omega_{\delta}(z')$ as follows:
\begin{align*}
\begin{cases}
x'-z'=\delta y',\\
x_{d}=\delta y_{d},
\end{cases}
\end{align*}
$\Omega_{\delta}(z')$ becomes $Q_{1}$, where, for $0<r\leq 1$,
\begin{align*}
Q_{r}=\left\{y\in\mathbb{R}^{d}\,\Big|\,\frac{1}{\delta}h(\delta y'+z')<y_{d}<\frac{\varepsilon}{\delta}+\frac{1}{\delta}h_{1}(\delta y'+z'),\;|y'|<r\right\},
\end{align*}
with its top and bottom boundaries represented, respectively, by
\begin{align*}
\Gamma^{+}_{r}=&\left\{y\in\mathbb{R}^{d}\,\Big|\,y_{d}=\frac{\varepsilon}{\delta}+\frac{1}{\delta}h_{1}(\delta y'+z'),\;|y'|<r\right\},
\end{align*}
and
\begin{align*}
\Gamma^{-}_{r}=&\left\{y\in\mathbb{R}^{d}\,\Big|\,y_{d}=\frac{1}{\delta}h(\delta y'+z'),\;|y'|<r\right\}.
\end{align*}
In fact, $Q_{1}$ is of nearly unit size. Similarly as in \eqref{ASK001}, we deduce that for $x\in\Omega_{\delta}(z')$,
\begin{align*}
|\delta(x')-\delta(z')|
\leq&2\kappa_{1}\delta(\delta^{\alpha}+|z'|^{\alpha})\leq 4\kappa_{1}\max\{1,\tau^{-\frac{\alpha}{1+\alpha}}\}\delta^{\frac{1+2\alpha}{1+\alpha}},
\end{align*}
which indicates that
\begin{align}\label{KHW03}
\left|\frac{\delta(x')}{\delta(z')}-1\right|\leq 8\max\{1,\tau^{\frac{\alpha}{1+\alpha}}\}\kappa_{1}R^{\alpha}.
\end{align}
Since $R$ is a small positive constant, it follows from \eqref{KHW03} that $Q_{1}$ is of nearly unit size.
For $y\in Q_{1}$, define
\begin{align*}
W(y',y_{d}):=w(\delta y'+z',\delta y_{d}),\quad U(y',y_{d}):=\bar{u}_{01}(\delta y'+z',\delta y_{d}).
\end{align*}
From \eqref{Zww01}, we know that $W$ solves
\begin{align}\label{CL005}
\begin{cases}
-\partial_{j}(C_{ijkl}^{0}\partial_{l}W^{k})=\partial_{j}(C_{ijkl}^{0}\partial_{l}U^{k}),&in\; Q_{1},\\
W=0,&on\;\Gamma^{\pm}_{1}.
\end{cases}
\end{align}
Applying Theorems \ref{CL001} and \ref{CL002} for equation \eqref{CL005} with $f_{ij}=C_{ijkl}^{0}\partial_{l}U^{k}$ and utilizing the Poincar\'{e} inequality, we have
\begin{align*}
\|\nabla W\|_{L^{\infty}(Q_{1/4})}\leq&C\big(\|W\|_{L^{\infty}(Q_{1/2})}+[\nabla U]_{\alpha,Q_{1/2}}\big)\notag\\
\leq&C\left(\|\nabla W\|_{L^{2}(Q_{1})}+[\nabla U]_{\alpha,Q_{1}}\right),
\end{align*}
where we used the fact that $[C_{ijkl}^{0}\partial_{l}U^{k}]_{\alpha,Q_{1}}\leq[\nabla U]_{\alpha,Q_{1}}$.

Tracking back to $w$, we see
\begin{align}\label{GDA001}
\|\nabla w\|_{L^{\infty}(\Omega_{\delta/4}(z'))}\leq\frac{C}{\delta}\big(\delta^{1-\frac{d}{2}}\|\nabla w\|_{L^{2}(\Omega_{\delta}(z'))}+\delta^{1+\alpha}[\nabla\bar{u}_{01}]_{\alpha,\Omega_{\delta}(z')}\big).
\end{align}
Then substituting \eqref{step2} and \eqref{QWN002} into \eqref{GDA001}, we obtain that for $z\in\Omega_{R}$,
\begin{align*}
|\nabla w(z)|\leq \|\nabla w\|_{L^{\infty}(\Omega_{\delta/4})(z')}\leq C\delta^{-\frac{1}{1+\alpha}}\big(|\varphi^{1}(z',h(z'))|+\delta^{\frac{1}{1+\alpha}}\|\varphi^{1}\|_{C^{1}}\big).
\end{align*}
The proof is complete.

\end{proof}

In exactly the same way to the proof of Theorem \ref{lem89999}, we obtain the following corollary.
\begin{corollary}\label{thm86}
Assume as above. Let $u_{i}\in H^{1}(\Omega;\mathbb{R}^{d})$, $i=1,2,...,\frac{d(d+1)}{2}$ be a weak solution of (\ref{P2.005}). Then, for a sufficiently small $\varepsilon>0$ and $x\in\Omega_{R}$,
\begin{align}\label{Le2.025}
\nabla u_{i}=&\nabla\bar{u}_{i}+O(1)
\begin{cases}
\delta^{-\frac{1}{1+\alpha}},& i=1,2,...,d,\\
1,&i=d+1,...,\frac{d(d+1)}{2},
\end{cases}
\end{align}
where $\bar{u}_{i}$, $i=1,2,...,\frac{d(d+1)}{2}$, are defined in \eqref{GWQ}.
\end{corollary}

Secondly, we present the asymptotic expansions of $C^{i}$, $i=1,2,...,\frac{d(d+1)}{2}$ in the following theorem with its proof given in Section \ref{SEC005}.
\begin{theorem}\label{OMG123}
Assume as in Theorem \ref{Lthm066}, then for a sufficiently small $\varepsilon>0$,
\begin{itemize}
\item[$(a)$] if $d=2$, for $i=1,2$,
\begin{align*}
C^{i}=&\frac{\det\mathbb{B}_{i}^{\ast}[\varphi]}{a_{33}^{\ast}}\frac{\varepsilon^{\frac{\alpha}{1+\alpha}}(1+O(\varepsilon(\alpha,\beta)))}{\mathcal{L}_{d}^{i}\mathcal{M}_{\alpha,\tau}},\quad C^{3}=\frac{Q_{3}^{\ast}[\varphi]}{a_{33}^{\ast}}(1+O(\varepsilon^{\frac{\alpha}{2(1+2\alpha)}})),
\end{align*}
where $\mathcal{M}_{\alpha,\tau}$ is defined by \eqref{zwzh001}, $\mathcal{L}_{2}^{i}$ is defined in \eqref{AZ}, $\varepsilon(\alpha,\beta)$ is defined by \eqref{ZWZHAO01}.

\item[$(b)$] if $d\geq3$, for $i=1,2,...,\frac{d(d+1)}{2}$,
\begin{align*}
C^{i}=\frac{\det\mathbb{F}_{i}^{\ast}[\varphi]}{\det\mathbb{A}^{\ast}}(1+O(\varepsilon^{\frac{\alpha^{2}}{2(1+2\alpha)(1+\alpha)^{2}}})),
\end{align*}
where the blow-up factor matrices $\mathbb{A}^{\ast}$ and $\mathbb{F}_{i}^{\ast}[\varphi]$ are defined in \eqref{matrix01}--\eqref{LATP001}.
\end{itemize}
\end{theorem}

Combining the results above, we immediately give the proofs of Theorems \ref{Lthm066} and \ref{Lthm06666}.
\begin{proof}[Proofs of Theorems \ref{Lthm066} and \ref{Lthm06666}.]
In view of decomposition \eqref{Le2.015}, it follows from Theorem \ref{lem89999}, Corollary \ref{thm86} and Theorem \ref{OMG123} that

$(i)$ if $d=2$, then
\begin{align*}
\nabla u=&\sum^{2}_{i=1}\frac{\det\mathbb{B}_{i}^{\ast}[\varphi]}{a_{33}^{\ast}}\frac{\varepsilon^{\frac{\alpha}{1+\alpha}}(1+O(\varepsilon(\alpha,\beta)))}{\mathcal{L}_{2}^{i}\mathcal{M}_{\alpha,\tau}}(\nabla\bar{u}_{i}+O(\delta^{-\frac{1}{1+\alpha}}))\notag\\
&+\frac{Q_{3}^{\ast}[\varphi]}{a_{33}^{\ast}}(1+O(\varepsilon^{\frac{\alpha}{2(1+2\alpha)}}))(\nabla\bar{u}_{3}+O(1))+\nabla\bar{u}_{0}+O(1)\|\varphi\|_{C^{1}(\partial D)}\notag\\
=&\sum^{2}_{i=1}\frac{\det\mathbb{B}_{i}^{\ast}[\varphi]}{a_{33}^{\ast}}\frac{\varepsilon^{\frac{\alpha}{1+\alpha}}(1+O(\varepsilon(\alpha,\beta)))}{\mathcal{L}_{2}^{i}\mathcal{M}_{\alpha,\tau}}\nabla\bar{u}_{i}+\frac{Q_{3}^{\ast}[\varphi]}{a_{33}^{\ast}}(1+O(\varepsilon^{\frac{\alpha}{2(1+2\alpha)}}))\nabla\bar{u}_{3}\notag\\
&+\nabla\bar{u}_{0}+O(1)\delta^{-\frac{1-\alpha}{1+\alpha}}\|\varphi\|_{C^{1}(\partial D)};
\end{align*}

$(ii)$ if $d\geq3$, then
\begin{align*}
\nabla u=&\sum^{d}_{i=1}\frac{\det\mathbb{F}_{i}^{\ast}[\varphi]}{\det\mathbb{A}^{\ast}}(1+O(\varepsilon^{\frac{\alpha^{2}}{2(1+2\alpha)(1+\alpha)^{2}}}))(\nabla\bar{u}_{i}+O(\delta^{-\frac{1}{1+\alpha}}))\notag\\
&+\sum^{\frac{d(d+1)}{2}}_{i=d+1}\frac{\det\mathbb{F}_{i}^{\ast}[\varphi]}{\det\mathbb{A}^{\ast}}(1+O(\varepsilon^{\frac{\alpha^{2}}{2(1+2\alpha)(1+\alpha)^{2}}}))(\nabla\bar{u}_{i}+O(1))+\nabla\bar{u}_{0}+O(1)\|\varphi\|_{C^{1}(\partial D)}\notag\\
=&\sum_{i=1}^{\frac{d(d+1)}{2}}\frac{\det\mathbb{F}_{i}^{\ast}[\varphi]}{\det\mathbb{A}^{\ast}}(1+O(\varepsilon^{\frac{\alpha^{2}}{2(1+2\alpha)(1+\alpha)^{2}}}))\nabla\bar{u}_{i}+\nabla\bar{u}_{0}+O(1)\delta^{-\frac{1}{1+\alpha}}\|\varphi\|_{C^{1}(\partial D)}.
\end{align*}

Consequently, we accomplish the proofs of Theorems \ref{Lthm066} and \ref{Lthm06666}.

\end{proof}

\section{Proof of Theorem \ref{OMG123}}\label{SEC005}
Recalling decomposition (\ref{Le2.015}) and utilizing the fourth line of (\ref{La.002}), we arrive at
\begin{align}\label{Le3.078}
\sum^{\frac{d(d+1)}{2}}_{i=1}C^{i}a_{ij}=Q_{j}[\varphi],\quad j=1,2,...,\frac{d(d+1)}{2},
\end{align}
where, for $i,j=1,2,...,\frac{d(d+1)}{2}$,
\begin{align}\label{AKDM001}
a_{ij}:=-\int_{\partial D_{1}}\frac{\partial u_{i}}{\partial\nu_{0}}\Big|_{+}\cdot\psi_{j},\quad Q_{j}[\varphi]:=\int_{\partial D_{1}}\frac{\partial u_{0}}{\partial\nu_{0}}\Big|_{+}\cdot\psi_{j}.
\end{align}
From \eqref{Le3.078}, we reduce the proof of Theorem \ref{OMG123} to the establishments of the asymptotic expressions of $a_{ij}$ and $Q_{j}[\varphi]$.

\subsection{Asymptotics of $Q_{j}[\varphi]$, $j=1,2,...,\frac{d(d+1)}{2}$.}\label{subsec31}

Denote the unit outer normal of $D$ near the origin by
\begin{align}\label{normal}
\nu:=(\nu_{1},...,\nu_{d})=\left(\frac{\nabla_{x'}h}{\sqrt{1+|\nabla_{x'}h|^{2}}},\frac{-1}{\sqrt{1+|\nabla_{x'}h|^{2}}}\right).
\end{align}
\begin{lemma}\label{KM323}
Assume as in Theorem  \ref{Lthm066}. Then, for a sufficiently small $\varepsilon>0$,

\begin{align}\label{AAPP321}
Q_{j}[\varphi]&=Q_{j}^{\ast}[\varphi]+O(1)
\begin{cases}
\varepsilon^{\frac{(d-1-\alpha)\alpha}{d(1+2\alpha)}},&j=1,2,...,d,\\
\varepsilon^{\frac{(d-\alpha)(1+\alpha)}{(d+1)(1+2\alpha)}},&j=d+1,...,\frac{d(d+1)}{2}.
\end{cases}
\end{align}

\end{lemma}
\begin{remark}
The difference on the convergence rate between $Q_{j}[\varphi]$ and $Q_{j}^{\ast}[\varphi]$ in the case of $j=1,2,...,d$ and $j=d+1,...,\frac{d(d+1)}{2}$ arises from the difference of iterate results in \eqref{Le2.025}.
\end{remark}

\begin{proof}
We only give the proof of (\ref{AAPP321}) in the case of $j=1,2,...,d$, because the case of $j=d+1,...,\frac{d(d+1)}{2}$ is almost the same to the former with a slight modification. Recalling the definitions of $u_{0}$ and $u_{j}$, $j=1,2,...,d$, it follows from (\ref{Le2.01222}) that
\begin{align*}
Q_{j}[\varphi]-Q^{\ast}_{j}[\varphi]=\int_{\partial D}\frac{\partial(u_{j}-u_{j}^{\ast})}{\partial\nu_{0}}\Big|_{+}\cdot\varphi(x),
\end{align*}
where $u_{j}^{\ast}$ satisfies equation \eqref{l03.001}.

Introduce a family of auxiliary functions as follows: for $j=1,2,...,d$,
$$\bar{u}_{j}^{\ast}=\bar{v}^{\ast}e_{j},$$
where $\bar{v}^{\ast}$ satisfies $\bar{v}^{\ast}=1$ on $\partial D_{1}^{\ast}\setminus\{0\}$, $\bar{v}^{\ast}=0$ on $\partial D$, and
$$\bar{v}^{\ast}=\frac{x_{d}-h(x')}{h_{1}(x')-h(x')},\quad\mathrm{in}\;\,\Omega_{2R}^{\ast},\quad\;\,\|\bar{v}^{\ast}\|_{C^{2}(\Omega^{\ast}\setminus\Omega_{R}^{\ast})}\leq C,$$
where $\Omega^{\ast}_{r}:=\Omega^{\ast}\cap\{|x'|<r\},$ $0<r\leq2R$. In light of ({\bf{\em H3}}), we derive that for $x\in\Omega_{R}^{\ast}$,
\begin{align}\label{Lw3.003}
|\nabla_{x'}(\bar{u}_{j}-\bar{u}^{\ast }_{j})|\leq\frac{C}{|x'|},\quad|\partial_{x_{d}}(\bar{u}_{j}-\bar{u}^{\ast}_{j})|\leq\frac{C\varepsilon}{|x'|^{1+\alpha}(\varepsilon+|x'|^{1+\alpha})}.
\end{align}
Applying Corollary \ref{thm86} to (\ref{l03.001}), it follows that for $x\in\Omega_{R}^{\ast}$,
\begin{align}\label{Lw3.005}
|\nabla(u_{j}^{\ast}-\bar{u}_{j}^{\ast})|\leq\frac{C}{|x'|}.
\end{align}
For $0<r<R$, define
\begin{align*}
\mathcal{C}_{r}:=\left\{x\in\mathbb{R}^{d}\Big|\;|x'|<r,\,\frac{1}{2}\min_{|x'|\leq r}h(x')\leq x_{d}\leq\varepsilon+2\max_{|x'|\leq r}h_{1}(x')\right\}.
\end{align*}
We divide into two substeps to estimate the difference $|Q_{j}[\varphi]-Q^{\ast}_{j}[\varphi]|$ in the following.

{\bf Step 1.} Note that $u_{j}-u_{j}^{\ast}$ solves
\begin{align*}
\begin{cases}
\mathcal{L}_{\lambda,\mu}(u_{j}-u_{j}^{\ast})=0,&\mathrm{in}\;\,D\setminus(\overline{D_{1}\cup D_{1}^{\ast}}),\\
u_{j}-u_{j}^{\ast}=\psi_{j}-u_{j}^{\ast},&\mathrm{on}\;\,\partial D_{1}\setminus D_{1}^{\ast},\\
u_{j}-u_{j}^{\ast}=u_{j}-\psi_{j},&\mathrm{on}\;\,\partial D_{1}^{\ast}\setminus(D_{1}\cup\{0\}),\\
u_{j}-u_{j}^{\ast}=0,&\mathrm{on}\;\,\partial D.
\end{cases}
\end{align*}
First, we estimate $|u_{j}-u_{j}^{\ast}|$ on $\partial(D_{1}\cup D_{1}^{\ast})\setminus\mathcal{C}_{\varepsilon^{\gamma}}$, where $0<\gamma<1$ to be determined later. By the standard elliptic estimates, we have
$$|\partial_{x_{d}}u_{j}^{\ast}|\leq C,\quad\;\,\mathrm{in}\;\,\Omega^{\ast}\setminus\Omega^{\ast}_{R},$$
which reads that
\begin{align}\label{Lw3.007}
|u_{j}-u_{j}^{\ast}|\leq C\varepsilon,\quad\;\,\mathrm{for}\;\,x\in\partial D_{1}\setminus D_{1}^{\ast}.
\end{align}
From \eqref{Le2.025}, we see
\begin{align}\label{Lw3.008}
|u_{j}-u_{j}^{\ast}|\leq C\varepsilon^{1-(1+\alpha)\gamma},\quad\;\,\mathrm{on}\;\,\partial D_{1}^{\ast}\setminus(D_{1}\cup\mathcal{C}_{\varepsilon^{\gamma}}).
\end{align}
For $x\in\Omega_{R}^{\ast}\cap\{|x'|=\varepsilon^{\gamma}\}$, it follows from \eqref{Le2.025} and (\ref{Lw3.003})--(\ref{Lw3.005}) that
\begin{align*}
|\partial_{x_{d}}(u_{j}-u_{j}^{\ast})|\leq&|\partial_{x_{d}}(u_{j}-\bar{u}_{j})|+|\partial_{x_{d}}(\bar{u}_{j}-\bar{u}_{j}^{\ast})|+|\partial_{x_{d}}(u_{j}^{\ast}-\bar{u}_{j}^{\ast})|\\
\leq&C\left(\frac{1}{\varepsilon^{2(1+\alpha)\gamma-1}}+\frac{1}{\varepsilon^{\gamma}}\right).
\end{align*}
This, together with the fact that $\bar{u}_{j}-\bar{u}_{j}^{\ast}=0$ on $\partial D$, yields that
\begin{align}\label{Lw3.009}
|(u_{j}-u_{j}^{\ast})(x',x_{d})|=&|(u_{j}-u_{j}^{\ast})(x',x_{d})-(u_{j}-u_{j}^{\ast})(x',h(x'))|\notag\\
\leq&C\big(\varepsilon^{1-(1+\alpha)\gamma}+\varepsilon^{\alpha\gamma}\big).
\end{align}
Pick $\gamma=\frac{1}{1+2\alpha}$. Then combining (\ref{Lw3.007})--(\ref{Lw3.009}), we have
$$|u_{j}-u_{j}^{\ast}|\leq C\varepsilon^{\frac{\alpha}{1+2\alpha}},\quad\;\,\mathrm{on}\;\,\partial\big(D\setminus\big(\overline{D_{1}\cup D_{1}^{\ast}\cup\mathcal{C}_{\varepsilon^{\frac{1}{1+2\alpha}}}}\big)\big),$$
which, in combination with the maximum principle for Lam\'{e} system in \cite{MMN2007}, reads that
\begin{align}\label{ZZW0101}
|u_{j}-u_{j}^{\ast}|\leq C\varepsilon^{\frac{\alpha}{1+2\alpha}},\quad\;\,\mathrm{in}\;\,D\setminus\big(\overline{D_{1}\cup D_{1}^{\ast}\cup\mathcal{C}_{\varepsilon^{\frac{1}{1+2\alpha}}}}\big).
\end{align}
Therefore, utilizing the standard interior and boundary estimates for Lam\'{e} system, we know that for any $\frac{1}{(1+\alpha)(1+2\alpha)}<\tilde{\gamma}<\frac{1}{1+2\alpha}$,
$$|\nabla(u_{j}-u_{j}^{\ast})|\leq C\varepsilon^{(1+\alpha)\tilde{\gamma}-\frac{1}{1+2\alpha}},\quad\;\,\mathrm{on}\;\,D\setminus\big(\overline{D_{1}\cup D_{1}^{\ast}\cup\mathcal{C}_{\varepsilon^{\frac{1}{1+2\alpha}-\tilde{\gamma}}}}\big),$$
which implies that
\begin{align}\label{Lw3.010}
|\mathcal{A}^{\mathrm{I}}|:=\left|\int_{\partial D\setminus\mathcal{C}_{\varepsilon^{\frac{1}{1+2\alpha}-\tilde{\gamma}}}}\frac{\partial(u_{j}-u_{j}^{\ast})}{\partial\nu_{0}}\Big|_{+}\cdot\varphi(x)\right|\leq C\varepsilon^{(1+\alpha)\tilde{\gamma}-\frac{1}{1+2\alpha}}.
\end{align}

{\bf Step 2.} We now estimate the residual part as follows:
\begin{align*}
\mathcal{A}^{\mathrm{II}}:=&\int_{\partial D\cap\mathcal{C}_{\varepsilon^{\frac{1}{1+2\alpha}-\tilde{\gamma}}}}\frac{\partial(u_{j}-u_{j}^{\ast})}{\partial\nu_{0}}\Big|_{+}\cdot\varphi(x)\\
=&\int_{\partial D\cap\mathcal{C}_{\varepsilon^{\frac{1}{1+2\alpha}-\tilde{\gamma}}}}\frac{\partial(\bar{u}_{j}-\bar{u}_{j}^{\ast})}{\partial\nu_{0}}\Big|_{+}\cdot\varphi(x)+\int_{\partial D\cap\mathcal{C}_{\varepsilon^{\frac{1}{1+2\alpha}-\tilde{\gamma}}}}\frac{\partial(w_{j}-w_{j}^{\ast})}{\partial\nu_{0}}\Big|_{+}\cdot\varphi(x)\\
=&:\mathcal{A}_{1}^{\mathrm{II}}+\mathcal{A}_{2}^{\mathrm{II}},
\end{align*}
where $w_{j}=u_{j}-\bar{u}_{j}$, $w_{j}^{\ast}=u_{j}^{\ast}-\bar{u}_{j}^{\ast}$. With regard to $\mathcal{A}_{1}^{\mathrm{II}}$, on one hand, for $j=1,...,d-1$, we decompose it into two parts as follows:
\begin{align*}
\mathcal{A}_{11}^{\mathrm{II}}=&\int_{\partial D\cap\mathcal{C}_{\varepsilon^{\frac{1}{1+2\alpha}-\tilde{\gamma}}}}\bigg\{\lambda\sum^{d}_{i=1}\partial_{x_{j}}(\bar{u}^{j}_{j}-\bar{u}^{\ast j}_{j})\nu_{i}\varphi^{i}(x)+\\
&\qquad+\mu\sum^{d-1}_{i=1}\partial_{x_{i}}(\bar{u}^{j}_{j}-\bar{u}^{\ast j}_{j})[\nu_{j}\varphi^{i}(x)+\nu_{i}\varphi^{j}(x)]+\mu\partial_{x_{d}}(\bar{u}^{j}_{j}-\bar{u}^{\ast j}_{j})\nu_{j}\varphi^{d}(x)\bigg\},\\
\mathcal{A}_{12}^{\mathrm{II}}=&\int_{\partial D\cap\mathcal{C}_{\varepsilon^{\frac{1}{1+2\alpha}-\tilde{\gamma}}}}\mu\partial_{x_{d}}(\bar{u}^{j}_{j}-\bar{u}^{\ast j}_{j})\nu_{d}\varphi^{j}(x),
\end{align*}
where $\nu_{i}$, $i=1,2,...,d$ are defined in \eqref{normal}. From (\ref{Lw3.003}) and the Taylor expansion of $\varphi^{i}(x)$, we see
\begin{align*}
|\mathcal{A}_{11}^{\mathrm{II}}|\leq C\int_{\partial D\cap\mathcal{C}_{\varepsilon^{\frac{1}{1+2\alpha}-\tilde{\gamma}}}}\|\varphi\|_{C^{1}(\partial D)}\leq C\|\varphi\|_{C^{1}(\partial D)}\varepsilon^{(\frac{1}{1+2\alpha}-\tilde{\gamma})(d-1)},
\end{align*}
and
\begin{align*}
|\mathcal{A}_{12}^{\mathrm{II}}|\leq C\|\varphi\|_{C^{1}(\partial D)}\int_{0}^{\varepsilon^{\frac{1}{1+2\alpha}-\tilde{\gamma}}}s^{d-2-\alpha}\,ds\leq C\|\varphi\|_{C^{1}(\partial D)}\varepsilon^{(\frac{1}{1+2\alpha}-\tilde{\gamma})(d-1-\alpha)},
\end{align*}
which leads to that
\begin{align}\label{TCD001}
|\mathcal{A}_{1}^{\mathrm{II}}|\leq&C\|\varphi\|_{C^{1}(\partial D)}\varepsilon^{(\frac{1}{1+2\alpha}-\tilde{\gamma})(d-1-\alpha)}.
\end{align}
On the other hand, for $j=d$, we split $\mathcal{A}_{1}^{\mathrm{II}}$ into two parts as follows:
\begin{align*}
\mathcal{A}_{11}^{\mathrm{II}}=&\int_{\partial D\cap\mathcal{C}_{\varepsilon^{\frac{1}{1+2\alpha}-\tilde{\gamma}}}}\bigg\{\lambda\sum^{d-1}_{i=1}\partial_{x_{d}}(\bar{u}^{d}_{d}-\bar{u}^{\ast d}_{d})\nu_{i}\varphi^{i}(x)+\\
&\qquad+\mu\sum^{d-1}_{i=1}\partial_{x_{i}}(\bar{u}^{d}_{d}-\bar{u}^{\ast d}_{d})[\nu_{d}\varphi^{i}(x)+\nu_{i}\varphi^{d}(x)]\bigg\},\\
\mathcal{A}_{12}^{\mathrm{II}}=&\int_{\partial D\cap\mathcal{C}_{\varepsilon^{\frac{1}{1+2\alpha}-\tilde{\gamma}}}}(\lambda+2\mu)\partial_{x_{d}}(\bar{u}^{d}_{d}-\bar{u}^{\ast d}_{d})\nu_{d}\varphi^{d}(x).
\end{align*}
Using (\ref{Lw3.003}) and the Taylor expansion of $\varphi^{i}(x)$ again, we deduce that \eqref{TCD001} still holds.

We next estimate $\mathcal{A}_{2}^{\mathrm{II}}$. Making use of Corollary \ref{thm86}, we derive that for $0<|x'|\leq R$,
\begin{align}\label{Lw3.016}
|\nabla w_{j}|\leq C(\varepsilon+|x'|^{1+\alpha})^{-\frac{1}{1+\alpha}},\quad|\nabla w_{j}^{\ast}|\leq\frac{C}{|x'|}.
\end{align}
By definition,
\begin{align*}
\mathcal{A}_{2}^{\mathrm{II}}=&\int_{\partial D\cap\mathcal{C}_{\varepsilon^{\frac{1}{1+2\alpha}-\tilde{\gamma}}}}\bigg\{\lambda\sum^{d}_{k,l=1}\partial_{x_{k}}(w_{j}^{k}-w_{j}^{\ast k})\nu_{l}\varphi^{l}(x)\\
&\quad+\mu\sum^{d}_{k,l=1}[\partial_{x_{l}}(w_{j}^{k}-w_{j}^{\ast k})+\partial_{x_{k}}(w_{j}^{l}-w_{j}^{\ast l})]\nu_{l}\varphi^{k}(x)\bigg\}.
\end{align*}
Making use of (\ref{Lw3.016}) and the Taylor expansion of $\varphi^{j}(x)$, we obtain
\begin{align}\label{Lw3.017}
|\mathcal{A}_{2}^{\mathrm{II}}|\leq C\int_{\partial D\cap\mathcal{C}_{\varepsilon^{\frac{1}{1+2\alpha}-\tilde{\gamma}}}}\|\varphi\|_{C^{1}(\partial D)}\leq C\|\varphi\|_{C^{1}(\partial D)}\varepsilon^{(\frac{1}{1+2\alpha}-\tilde{\gamma})(d-1)}.
\end{align}
Taking $\tilde{\gamma}=\frac{d-\alpha}{d(1+2\alpha)}$, it follows from \eqref{Lw3.010}--\eqref{TCD001} and \eqref{Lw3.017} that
\begin{align*}
|Q_{j}[\varphi]-Q^{\ast}_{j}[\varphi]|\leq C\varepsilon^{\frac{(d-1-\alpha)\alpha}{d(1+2\alpha)}},\quad j=1,2,...,d.
\end{align*}
Therefore, we complete the proof of Lemma \ref{KM323}.
\end{proof}

\subsection{Asymptotics of $a_{ij}$, $i,j=1,2,...,\frac{d(d+1)}{2}$.}
%Similarly as \eqref{AZ}, in dimension three we denote
%\begin{align}\label{CON01}
%(\mathcal{L}_{3}^{1},...,\mathcal{L}_{3}^{d-1},\mathcal{L}_{3}^{d})=(\mu,...,\mu,\lambda+2\mu).
%\end{align}
Multiplying the first line of (\ref{P2.005}) by $u_{i}$, it follows from integration by parts that for $i,j=1,2,...,\frac{d(d+1)}{2}$,
\begin{align*}
a_{ij}=\int_{\Omega}(\mathbb{C}^{0}e(u_{i}),e(u_{j})).
\end{align*}
For brevity, write
\begin{align}\label{APDN001}
\tilde{\varepsilon}(\alpha,\beta)=&
\begin{cases}
\varepsilon^{\frac{\beta}{1+\alpha}},&\alpha>\beta,\\
\varepsilon^{\frac{\alpha}{1+\alpha}}|\ln\varepsilon|,&\alpha=\beta,\\
\varepsilon^{\frac{\alpha}{1+\alpha}},&0<\alpha<\beta.
\end{cases}
\end{align}

\begin{lemma}\label{lemmabc}
Assume as above. Then, for a sufficiently small $\varepsilon>0$,

$(i)$ for $i=1,2,...d$, if $d=2$,
\begin{align}\label{zzw001}
a_{ii}=\mathcal{L}_{2}^{i}\mathcal{M}_{\alpha,\tau}\varepsilon^{-\frac{\alpha}{1+\alpha}}(1+O(\tilde{\varepsilon}(\alpha,\beta))),
\end{align}
and if $d\geq3$,
\begin{align}\label{zzw002}
a_{ii}=a_{ii}^{\ast}+O(\bar{\varepsilon}(\alpha,d)),
\end{align}
where $\bar{\varepsilon}(\alpha,d)$ and $\tilde{\varepsilon}(\alpha,\beta)$ are defined by \eqref{NZKL001} and \eqref{APDN001}, respectively.

$(ii)$ for $i=d+1,...,\frac{d(d+1)}{2}$,
\begin{align}\label{zzw003}
a_{ii}=a_{ii}^{\ast}+O(1)\varepsilon^{\frac{\alpha}{2(1+2\alpha)}};
\end{align}

$(iii)$ if $d=2$, for $i,j=1,2,i\neq j$, then
\begin{align}\label{LVZQ001}
a_{12}=a_{21}=O(1)|\ln\varepsilon|,
\end{align}
and if $d\geq3$, for $i,j=1,2,...,d,\,i\neq j$, then
\begin{align}\label{ADzc}
a_{ij}=a_{ji}=&a_{ij}^{\ast}+O(1)
\begin{cases}
\varepsilon^{\frac{\alpha^{2}}{2(1+2\alpha)(1+\alpha)^{2}}},&d=3,\\
\varepsilon^{\frac{\alpha^{2}}{2(1+2\alpha)(1+\alpha)}},&d\geq4,
\end{cases}
\end{align}
and if $d\geq2$, for $i=1,2,...,d,\,j=d+1,...,\frac{d(d+1)}{2},$ then
\begin{align}\label{LVZQ0011gdw}
a_{ij}=a_{ji}=&a_{ij}^{\ast}+O(1)
\begin{cases}
\varepsilon^{\frac{\alpha^{2}}{2(1+2\alpha)(1+\alpha)^{2}}},&d=2,\\
\varepsilon^{\frac{\alpha^{2}}{2(1+2\alpha)(1+\alpha)}},&d\geq3,
\end{cases}
\end{align}
and if $d\geq3$, for $i,j=d+1,...,\frac{d(d+1)}{2},\,i\neq j$, then
\begin{align}\label{zzw007}
a_{ij}=a_{ij}^{\ast}+O(1)\varepsilon^{\frac{\alpha}{2(1+2\alpha)}}.
\end{align}

\end{lemma}
\begin{proof}
{\bf Step 1. Proofs of (\ref{zzw001})--(\ref{zzw002})}. Denote $\bar{\gamma}=\frac{\alpha^{2}}{2(1+2\alpha)(1+\alpha)^{2}}$. For $\varepsilon^{\bar{\gamma}}\leq|z'|\leq R$, making use of the change of variable
\begin{align*}
\begin{cases}
x'-z'=|z'|^{1+\alpha}y',\\
x_{d}=|z'|^{1+\alpha}y_{d},
\end{cases}
\end{align*}
then $\Omega_{|z'|+|z'|^{1+\alpha}}\setminus\Omega_{|z'|}$ and $\Omega_{|z'|+|z'|^{1+\alpha}}^{\ast}\setminus\Omega_{|z'|}^{\ast}$ become two nearly unit-size squares (or cylinders) $Q_{1}$ and $Q_{1}^{\ast}$, respectively. For $i=1,2,...,d$, let
\begin{align*}
U_{i}(y)=u_{i}(z'+|z'|^{1+\alpha}y',|z'|^{1+\alpha}y_{d}),\quad\mathrm{in}\;Q_{1},
\end{align*}
and
\begin{align*}
U_{i}^{\ast}(y)=u_{i}^{\ast}(z'+|z'|^{1+\alpha}y',|z'|^{1+\alpha}y_{d}),\quad\mathrm{in}\;Q_{1}^{\ast}.
\end{align*}
Due to the fact that $0<V_{i},V_{i}^{\ast}<1$, it follows from the standard elliptic estimates that
\begin{align*}
\|U_{i}\|_{C^{1,\alpha}(Q_{1})}\leq C,\quad\|U_{i}^{\ast}\|_{C^{1,\alpha}(Q_{1}^{\ast})}\leq C.
\end{align*}
Utilizing an interpolation with \eqref{ZZW0101}, we arrive at
\begin{align*}
|\nabla(U_{i}-U_{i}^{\ast})|\leq C\varepsilon^{\frac{\alpha}{1+2\alpha}(1-\frac{1}{1+\alpha})}\leq C\varepsilon^{\frac{\alpha^{2}}{(1+2\alpha)(1+\alpha)}}.
\end{align*}
Thus, rescaling it back to $u_{i}-u_{i}^{\ast}$ and in view of $\varepsilon^{\bar{\gamma}}\leq|z'|\leq R$, we know
\begin{align*}
|\nabla(u_{i}-u_{i}^{\ast})(x)|\leq C\varepsilon^{\frac{\alpha^{2}}{(1+2\alpha)(1+\alpha)}}|z'|^{-1-\alpha}\leq C\varepsilon^{(1+\alpha)\bar{\gamma}},\quad x\in\Omega^{\ast}_{|z'|+|z'|^{1+\alpha}}\setminus\Omega_{|z'|}^{\ast},
\end{align*}
which implies that
\begin{align}\label{con035}
|\nabla(u_{i}-u_{i}^{\ast})|\leq C\varepsilon^{(1+\alpha)\bar{\gamma}},\quad\mathrm{in}\;D\setminus\big(\overline{D_{1}\cup D_{1}^{\ast}\cup\mathcal{C}_{\varepsilon^{\bar{\gamma}}}}\big).
\end{align}

We now split $a_{ii}$ into three parts as follows:
\begin{align*}
a_{ii}=&\int_{\Omega\setminus\Omega_{R}}(\mathbb{C}^{0}e(u_{i}),e(u_{i}))+\int_{\Omega_{\varepsilon^{\bar{\gamma}}}}(\mathbb{C}^{0}e(u_{i}),e(u_{i}))+\int_{\Omega_{R}\setminus\Omega_{\varepsilon^{\bar{\gamma}}}}(\mathbb{C}^{0}e(u_{i}),e(u_{i}))\nonumber\\
=&:\mathrm{I}+\mathrm{II}+\mathrm{III}.
\end{align*}
With regard to the first term $\mathrm{I}$, in light of the boundedness of $|\nabla u_{i}|$ in $D_{1}^{\ast}\setminus(D_{1}\cup\Omega_{R})$ and $D_{1}\setminus D_{1}^{\ast}$ and the fact that the volume of $D_{1}^{\ast}\setminus(D_{1}\cup\Omega_{R})$ and $D_{1}\setminus D_{1}^{\ast}$ is of order $O(\varepsilon)$, it follows from (\ref{con035}) that
\begin{align}\label{YUA123}
\mathrm{I}=&\int_{D\setminus(D_{1}\cup D_{1}^{\ast}\cup\Omega_{R})}(\mathbb{C}^{0}e(u_{i}),e(u_{i}))+O(1)\varepsilon\notag\\
=&\int_{D\setminus(D_{1}\cup D_{1}^{\ast}\cup\Omega_{R})}(\mathbb{C}^{0}e(u^{\ast}_{i}),e(u^{\ast}_{i}))+2\int_{D\setminus(D_{1}\cup D_{1}^{\ast}\cup\Omega_{R})}(\mathbb{C}^{0}e(u_{i}-u^{\ast}_{i}),e(u^{\ast}_{i}))\notag\\
&+\int_{D\setminus(D_{1}\cup D_{1}^{\ast}\cup\Omega_{R})}(\mathbb{C}^{0}e(u_{i}-u^{\ast}_{i}),e(u_{i}-u^{\ast}_{i}))\notag\\
=&\int_{\Omega^{\ast}\setminus\Omega^{\ast}_{R}}(\mathbb{C}^{0}e(u^{\ast}_{i}),e(u^{\ast}_{i}))+O(1)\varepsilon^{(1+\alpha)\bar{\gamma}}.
\end{align}

For the second term $\mathrm{II}$, recalling the definition of $\bar{u}_{i}$ and using Corollary \ref{thm86}, we have
\begin{align}\label{YUA03365}
\mathrm{II}=&\int_{\Omega_{\varepsilon^{\bar{\gamma}}}}(\mathbb{C}^{0}e(\bar{u}_{i}),e(\bar{u}_{i}))+2\int_{\Omega_{\varepsilon^{\bar{\gamma}}}}(\mathbb{C}^{0}e(u_{i}-\bar{u}_{i}),e(\bar{u}_{i}))\notag\\
&+\int_{\Omega_{\varepsilon^{\bar{\gamma}}}}(\mathbb{C}^{0}e(u_{i}-\bar{u}_{i}),e(u_{i}-\bar{u}_{i}))\notag\\
=&\,\mathcal{L}_{d}^{i}\int_{|x'|<\varepsilon^{\bar{\gamma}}}\frac{dx'}{\varepsilon+h_{1}(x')-h(x')}+O(1)
\begin{cases}
|\ln\varepsilon|,&d=2,\\
\varepsilon^{(d-2)\bar{\gamma}},&d\geq3,
\end{cases}
\end{align}
where $\mathcal{L}_{d}^{i}$ is defined by (\ref{AZ}).

For the third term $\mathrm{III}$, we further decompose it into three parts as follows:
\begin{align*}
\mathrm{III}_{1}=&\int_{(\Omega_{R}\setminus\Omega_{\varepsilon^{\bar{\gamma}}})\setminus(\Omega^{\ast}_{R}\setminus\Omega^{\ast}_{\varepsilon^{\bar{\gamma}}})}(\mathbb{C}^{0}e(u_{i}),e(u_{i})),\\
\mathrm{III}_{2}=&\int_{\Omega^{\ast}_{R}\setminus\Omega^{\ast}_{\varepsilon^{\bar{\gamma}}}}(\mathbb{C}^{0}e(u_{i}-u_{i}^{\ast}),e(u_{i}-u_{i}^{\ast}))+2\int_{\Omega^{\ast}_{R}\setminus\Omega^{\ast}_{\varepsilon^{\bar{\gamma}}}}(\mathbb{C}^{0}e(u_{i}-u_{i}^{\ast}),e(u_{i}^{\ast})),\\
\mathrm{III}_{3}=&\int_{\Omega^{\ast}_{R}\setminus\Omega^{\ast}_{\varepsilon^{\bar{\gamma}}}}(\mathbb{C}^{0}e(u_{i}^{\ast}),e(u_{i}^{\ast})).
\end{align*}
Since the thickness of $(\Omega_{R}\setminus\Omega_{\varepsilon^{\bar{\gamma}}})\setminus(\Omega^{\ast}_{R}\setminus\Omega^{\ast}_{\varepsilon^{\bar{\gamma}}})$ is of order $O(\varepsilon)$, we see from (\ref{Le2.025}) that
\begin{align}\label{YUA0333355}
\mathrm{III}_{1}\leq\,C\varepsilon\int_{\varepsilon^{\bar{\gamma}}<|x'|<R}\frac{dx'}{|x'|^{2(1+\alpha)}}\leq C\varepsilon^{1+(d-3-2\alpha)\bar{\gamma}}.
\end{align}
By applying Corollary \ref{thm86} to (\ref{l03.001}), we deduce that for $i=1,2,...,d$,
\begin{align}\label{LGA01}
|\nabla_{x'}u^{\ast}_{i}|\leq\frac{C}{|x'|},\quad|\partial_{d}u^{\ast}_{i}|\leq\frac{C}{|x'|^{1+\alpha}},\quad|\nabla(u_{i}^{\ast}-\bar{u}_{i}^{\ast})|\leq\frac{C}{|x'|}.
\end{align}
A consequence of (\ref{con035}) and \eqref{LGA01} yields that
\begin{align}\label{YUA036}
|\mathrm{III}_{2}|\leq C\varepsilon^{(1+\alpha)\bar{\gamma}}.
\end{align}
As for $\mathrm{III}_{3}$, using (\ref{LGA01}) again, we derive that for $d=2$,
\begin{align}\label{QPQ01}
\mathrm{III}_{3}=&\int_{\Omega_{R}^{\ast}\setminus\Omega^{\ast}_{\varepsilon^{\bar{\gamma}}}}(\mathbb{C}^{0}e(\bar{u}_{i}^{\ast}),e(\bar{u}_{i}^{\ast}))+2\int_{\Omega_{R}^{\ast}\setminus\Omega^{\ast}_{\varepsilon^{\bar{\gamma}}}}(\mathbb{C}^{0}e(u_{i}^{\ast}-\bar{u}_{i}^{\ast}),e(\bar{u}_{i}^{\ast}))\notag\\
&+\int_{\Omega_{R}^{\ast}\setminus\Omega^{\ast}_{\varepsilon^{\bar{\gamma}}}}(\mathbb{C}^{0}e(u_{i}^{\ast}-\bar{u}_{i}^{\ast}),e(u_{i}^{\ast}-\bar{u}_{i}^{\ast}))\notag\\
=&\,\mathcal{L}_{2}^{i}\int_{\varepsilon^{\bar{\gamma}}<|x_{1}|<R}\frac{dx_{1}}{h_{1}(x_{1})-h(x_{1})}+O(1)|\ln\varepsilon|;
\end{align}
for $d\geq3$,
\begin{align}\label{YUAQPQ01}
\mathrm{III}_{3}=&\int_{\Omega_{R}^{\ast}\setminus\Omega^{\ast}_{\varepsilon^{\bar{\gamma}}}}(\mathbb{C}^{0}e(\bar{u}_{i}^{\ast}),e(\bar{u}_{i}^{\ast}))+2\int_{\Omega_{R}^{\ast}\setminus\Omega^{\ast}_{\varepsilon^{\bar{\gamma}}}}(\mathbb{C}^{0}e(u_{i}^{\ast}-\bar{u}_{i}^{\ast}),e(\bar{u}_{i}^{\ast}))\notag\\
&+\int_{\Omega_{R}^{\ast}\setminus\Omega^{\ast}_{\varepsilon^{\bar{\gamma}}}}(\mathbb{C}^{0}e(u_{i}^{\ast}-\bar{u}_{i}^{\ast}),e(u_{i}^{\ast}-\bar{u}_{i}^{\ast}))\notag\\
=&\mathcal{L}_{d}^{i}\int_{\varepsilon^{\bar{\gamma}}<|x'|<R}\frac{dx'}{h_{1}(x')-h(x')}-\int_{\Omega^{\ast}\setminus\Omega^{\ast}_{R}}(\mathbb{C}^{0}e(u_{i}^{\ast}),e(u_{i}^{\ast}))\notag\\
&+M^{\ast}_{i}+O(1)\varepsilon^{(d-2)\bar{\gamma}},
\end{align}
where
\begin{align*}
M_{i}^{\ast}=&\int_{\Omega^{\ast}\setminus\Omega^{\ast}_{R}}(\mathbb{C}^{0}e(u_{i}^{\ast}),e(u_{i}^{\ast}))+2\int_{\Omega_{R}^{\ast}}(\mathbb{C}^{0}e(u_{i}^{\ast}-\bar{u}_{i}^{\ast}),e(\bar{u}_{i}^{\ast}))\notag\\
&+\int_{\Omega_{R}^{\ast}}(\mathbb{C}^{0}e(u_{i}^{\ast}-\bar{u}_{i}^{\ast}),e(u_{i}^{\ast}-\bar{u}_{i}^{\ast}))\\
&+
\begin{cases}
\int_{\Omega_{R}^{\ast}}(\lambda+\mu)(\partial_{i}\bar{v}^{\ast})^{2}+\mu\sum\limits^{d-1}_{j=1}(\partial_{j}\bar{v}^{\ast})^{2},&i=1,...,d-1,\\
\int_{\Omega_{R}^{\ast}}\mu\sum\limits^{d-1}_{j=1}(\partial_{j}\bar{v}^{\ast})^{2},&i=d.
\end{cases}
\end{align*}
Then, from (\ref{YUA123})--(\ref{YUA0333355}) and (\ref{YUA036})--(\ref{YUAQPQ01}), we arrive at
\begin{align}\label{aaaa01}
a_{ii}=&\mathcal{L}_{d}^{i}\left(\int_{\varepsilon^{\bar{\gamma}}<|x'|<R}\frac{dx'}{h_{1}(x')-h(x')}+\int_{|x'|<\varepsilon^{\bar{\gamma}}}\frac{dx'}{\varepsilon+h_{1}(x')-h(x')}\right)\nonumber\\
&+
\begin{cases}
O(1)|\ln\varepsilon|,&d=2,\\
M^{\ast}_{i}+O(1)\varepsilon^{\bar{\gamma}\min\{1+\alpha,d-2\}},&d\geq3.
\end{cases}
\end{align}
Further, for $d=2$,
\begin{align}\label{ZZQ01}
&\int_{|x_{1}|<R}\frac{dx_{1}}{\varepsilon+h_{1}-h}+\int_{\varepsilon^{\bar{\gamma}}<|x_{1}|<R}\frac{\varepsilon}{(h_{1}-h)(\varepsilon+h_{1}-h)}\notag\\
&=\int_{|x_{1}|<R}\frac{1}{\varepsilon+\tau|x_{1}|^{1+\alpha}}+\int_{|x_{1}|<R}\left(\frac{1}{\varepsilon+h_{1}-h}-\frac{1}{\varepsilon+\tau|x_{1}|^{1+\alpha}}\right)+O(1)\varepsilon^{\frac{\alpha^{2}+4\alpha+2}{2(1+\alpha)^{2}}}\notag\\
&=2\int_{0}^{R}\frac{1}{\varepsilon+\tau s^{1+\alpha}}+O(1)\int_{0}^{R}\frac{s^{\beta}}{\varepsilon+\tau s^{1+\alpha}}\notag\\
&=\frac{2\Gamma_{\alpha}}{(1+\alpha)\tau^{\frac{1}{1+\alpha}}}\varepsilon^{-\frac{\alpha}{1+\alpha}}
\begin{cases}
1+O(1)\varepsilon^{\frac{\beta}{1+\alpha}},&\alpha>\beta,\\
1+O(1)\varepsilon^{\frac{\alpha}{1+\alpha}}|\ln\varepsilon|,&\alpha=\beta,\\
1+O(1)\varepsilon^{\frac{\alpha}{1+\alpha}},&0<\alpha<\beta;
\end{cases}
\end{align}
for $d\geq3$,
\begin{align}\label{ZZQ02}
&\int_{|x'|<R}\frac{dx'}{h_{1}-h}-\int_{|x'|<\varepsilon^{\bar{\gamma}}}\frac{\varepsilon\,dx'}{(h_{1}-h)(\varepsilon+h_{1}-h)}\notag\\
&=\int_{\Omega_{R}^{\ast}}|\partial_{d}\bar{u}^{\ast}_{i}|^{2}+O(1)\varepsilon^{(d-2-\alpha)\bar{\gamma}}.
\end{align}
Consequently, it follows from \eqref{aaaa01}--\eqref{ZZQ02} that (\ref{zzw001})--(\ref{zzw002}) hold.

{\bf Step 2. Proof of (\ref{zzw003})}. Note that for $i=d+1,...,\frac{d(d+1)}{2}$, there exist two indices $1\leq k_{i}<l_{i}\leq d$ such that $\psi_{i}=(0,...,0,x_{l_{i}},0,...,0,-x_{k_{i}},0,...,0)$. Take $\tilde{\gamma}=\frac{\alpha}{2(1+2\alpha)(1+\alpha)}$. For $i=d+1,...,\frac{d(d+1)}{2}$, similarly as above, we split $a_{ii}$ as follows:
\begin{align*}
a_{ii}=&\int_{\Omega\setminus\Omega_{R}}(\mathbb{C}^{0}e(u_{i}),e(u_{i}))+\int_{\Omega_{\varepsilon^{\tilde{\gamma}}}}(\mathbb{C}^{0}e(u_{i}),e(u_{i}))+\int_{\Omega_{R}\setminus\Omega_{\varepsilon^{\tilde{\gamma}}}}(\mathbb{C}^{0}e(u_{i}),e(u_{i}))\nonumber\\
=&:\mathrm{I}+\mathrm{II}+\mathrm{III}.
\end{align*}

To begin with, by applying \eqref{Lw3.003}--\eqref{ZZW0101} with a minor modification, we obtain that for $i=d+1,...,\frac{d(d+1)}{2}$,
\begin{align}\label{QNWE001}
|u_{i}-u_{i}^{\ast}|\leq C\varepsilon^{\frac{1+\alpha}{1+2\alpha}},\quad\;\,\mathrm{in}\;\,D\setminus\big(\overline{D_{1}\cup D_{1}^{\ast}\cup\mathcal{C}_{\varepsilon^{\frac{1}{1+2\alpha}}}}\big).
\end{align}
Analogously as before, in view of \eqref{QNWE001}, we deduce from the rescale argument, the interpolation inequality and the standard elliptic estimates that for $i=d+1,...,\frac{d(d+1)}{2}$,
\begin{align}\label{QNWE002}
|\nabla(u_{i}-u_{i}^{\ast})|\leq C\varepsilon^{(1+\alpha)\tilde{\gamma}},\quad\;\,\mathrm{in}\;\,D\setminus\big(\overline{D_{1}\cup D_{1}^{\ast}\cup\mathcal{C}_{\varepsilon^{\tilde{\gamma}}}}\big).
\end{align}

As for the first term $\mathrm{I}$, following the same argument used in \eqref{YUA123}, we see that
\begin{align}\label{QNWE003}
\mathrm{I}=&\int_{D\setminus(D_{1}\cup D_{1}^{\ast}\cup\Omega_{R})}(\mathbb{C}^{0}e(u_{i}),e(u_{i}))+O(1)\varepsilon\notag\\
=&\int_{\Omega^{\ast}\setminus\Omega^{\ast}_{R}}(\mathbb{C}^{0}e(u^{\ast}_{i}),e(u^{\ast}_{i}))+O(1)\varepsilon^{(1+\alpha)\tilde{\gamma}}.
\end{align}

For the second term $\mathrm{II}$, we further split it as follows:
\begin{align*}
\mathrm{II}
=&\int_{\Omega_{\varepsilon^{\bar{\gamma}}}}[(\mathbb{C}^{0}e(\bar{u}_{i}),e(\bar{u}_{i}))+2(\mathbb{C}^{0}e(\bar{u}_{i}),e(u_{i}-\bar{u}_{i}))+(\mathbb{C}^{0}e(u_{i}-\bar{u}_{i}),e(u_{i}-\bar{u}_{i}))].
\end{align*}
By a direct calculation, it follows that for $i=d+1,...,\frac{d(d+1)}{2}$,
\begin{align*}
(\mathbb{C}^{0}e(\bar{u}_{i}),e(\bar{u}_{i}))=&\mu(x_{k_{i}}^{2}+x_{l_{i}}^{2})\sum^{d}_{j=1}(\partial_{x_{j}}\bar{v})^{2}+(\lambda+\mu)(x_{l_{i}}\partial_{x_{k_{i}}}\bar{v}-x_{k_{i}}\partial_{x_{l_{i}}}\bar{v})^{2}.
\end{align*}
This, together with Corollary \ref{thm86}, yields that
\begin{align}\label{QNWE005}
\mathrm{II}=&O(1)\varepsilon^{(d-\alpha)\tilde{\gamma}}.
\end{align}

With regard to the third term $\mathrm{III}$, we further decompose it as follows:
\begin{align*}
\mathrm{III}_{1}=&\int_{(\Omega_{R}\setminus\Omega_{\varepsilon^{\tilde{\gamma}}})\setminus(\Omega^{\ast}_{R}\setminus\Omega^{\ast}_{\varepsilon^{\tilde{\gamma}}})}(\mathbb{C}^{0}e(u_{i}),e(u_{i}))+\int_{\Omega^{\ast}_{R}\setminus\Omega^{\ast}_{\varepsilon^{\tilde{\gamma}}}}(\mathbb{C}^{0}e(u_{i}-u_{i}^{\ast}),e(u_{i}-u_{i}^{\ast}))\notag\\
&+2\int_{\Omega^{\ast}_{R}\setminus\Omega^{\ast}_{\varepsilon^{\tilde{\gamma}}}}(\mathbb{C}^{0}e(u_{i}-u_{i}^{\ast}),e(u_{i}^{\ast})),\\
\mathrm{III}_{2}=&\int_{\Omega^{\ast}_{R}\setminus\Omega^{\ast}_{\varepsilon^{\tilde{\gamma}}}}(\mathbb{C}^{0}e(u_{i}^{\ast}),e(u_{i}^{\ast})).
\end{align*}
Since the thickness of $(\Omega_{R}\setminus\Omega_{\varepsilon^{\tilde{\gamma}}})\setminus(\Omega^{\ast}_{R}\setminus\Omega^{\ast}_{\varepsilon^{\tilde{\gamma}}})$ is $\varepsilon$, then we obtain from (\ref{Le2.025}), \eqref{LGA01} and \eqref{QNWE002} that
\begin{align}\label{QNWE006}
\mathrm{III}_{1}=O(1)\varepsilon^{(1+\alpha)\tilde{\gamma}}.
\end{align}
As for $\mathrm{III}_{2}$, following the same argument as in \eqref{QNWE005}, we obtain
\begin{align*}
\int_{\Omega^{\ast}_{\varepsilon^{\tilde{\gamma}}}}(\mathbb{C}^{0}e(u_{i}^{\ast}),e(u_{i}^{\ast}))=O(1)\varepsilon^{(d-\alpha)\tilde{\gamma}},
\end{align*}
which implies that
\begin{align}\label{QNWE007}
\mathrm{III}_{2}=&\int_{\Omega^{\ast}_{R}}(\mathbb{C}^{0}e(u_{i}^{\ast}),e(u_{i}^{\ast}))-\int_{\Omega^{\ast}_{\varepsilon^{\tilde{\gamma}}}}(\mathbb{C}^{0}e(u_{i}^{\ast}),e(u_{i}^{\ast}))\notag\\
=&\int_{\Omega^{\ast}_{R}}(\mathbb{C}^{0}e(u_{i}^{\ast}),e(u_{i}^{\ast}))+O(1)\varepsilon^{(d-\alpha)\tilde{\gamma}}.
\end{align}
Therefore, combining \eqref{QNWE003}--\eqref{QNWE007}, we obtain that for $i=d+1,...,\frac{d(d+1)}{2}$,
\begin{align*}
a_{ii}=&a_{ii}^{\ast}+O(1)\varepsilon^{(1+\alpha)\tilde{\gamma}}.
\end{align*}

{\bf Step 3. Proofs of (\ref{LVZQ001})--(\ref{zzw007})}. Because of symmetry, it suffices to consider the case of $i<j$. Let
\begin{align*}
\hat{\gamma}=&
\begin{cases}
\frac{\alpha^{2}}{2(1+2\alpha)(1+\alpha)^{2}},&i=1,2,...,d,\,j=1,2,...,\frac{d(d+1)}{2},\,i<j,\\
\frac{\alpha}{2(1+2\alpha)(1+\alpha)},&i,j=d+1,...,\frac{d(d+1)}{2},\,i<j.
\end{cases}
\end{align*}
Similarly as before, for $i,j=1,2,...,\frac{d(d+1)}{2}$, $i<j$, we decompose $a_{ij}$ into three parts as follows:
\begin{align*}
a_{ij}=&\int_{\Omega\setminus\Omega_{R}}(\mathbb{C}^{0}e(u_{i}),e(u_{j}))+\int_{\Omega_{\varepsilon^{\hat{\gamma}}}}(\mathbb{C}^{0}e(u_{i}),e(u_{j}))+\int_{\Omega_{R}\setminus\Omega_{\varepsilon^{\hat{\gamma}}}}(\mathbb{C}^{0}e(u_{i}),e(u_{j}))\nonumber\\
=&:\mathrm{I}+\mathrm{II}+\mathrm{III}.
\end{align*}

For the first part $\mathrm{I}$, similar to \eqref{YUA123}, we see that
\begin{align}\label{KKAA123}
\mathrm{I}=&\int_{D\setminus(D_{1}\cup D_{1}^{\ast}\cup\Omega_{R})}(\mathbb{C}^{0}e(u_{i}),e(u_{j}))+O(1)\varepsilon\notag\\
=&\int_{D\setminus(D_{1}\cup D_{1}^{\ast}\cup\Omega_{R})}(\mathbb{C}^{0}e(u^{\ast}_{i}),e(u^{\ast}_{j}))+\int_{D\setminus(D_{1}\cup D_{1}^{\ast}\cup\Omega_{R})}(\mathbb{C}^{0}e(u_{i}-u^{\ast}_{i}),e(u_{j}-u^{\ast}_{j}))\notag\\
&+\int_{D\setminus(D_{1}\cup D_{1}^{\ast}\cup\Omega_{R})}(\mathbb{C}^{0}e(u^{\ast}_{i}),e(u_{j}-u^{\ast}_{j}))+\int_{D\setminus(D_{1}\cup D_{1}^{\ast}\cup\Omega_{R})}(\mathbb{C}^{0}e(u_{i}-u^{\ast}_{i}),e(u^{\ast}_{j}))\notag\\
=&\int_{\Omega^{\ast}\setminus\Omega^{\ast}_{R}}(\mathbb{C}^{0}e(u^{\ast}_{i}),e(u^{\ast}_{j}))+O(1)\varepsilon^{(1+\alpha)\hat{\gamma}}.
\end{align}

As for the second part $\mathrm{II}$, we further split it as follows:
\begin{align}\label{NAT001}
\mathrm{II}=&\int_{\Omega_{\varepsilon^{\hat{\gamma}}}}(\mathbb{C}^{0}e(u_{i}),e(u_{j}))\notag\\
=&\int_{\Omega_{\varepsilon^{\hat{\gamma}}}}(\mathbb{C}^{0}e(\bar{u}_{i}),e(\bar{u}_{j}))+\int_{\Omega_{\varepsilon^{\hat{\gamma}}}}(\mathbb{C}^{0}e(u_{i}-\bar{u}_{i}),e(u_{j}-\bar{u}_{j}))\notag\\
&+\int_{\Omega_{\varepsilon^{\hat{\gamma}}}}(\mathbb{C}^{0}e(\bar{u}_{i}),e(u_{j}-\bar{u}_{j}))+\int_{\Omega_{\varepsilon^{\hat{\gamma}}}}(\mathbb{C}^{0}e(u_{i}-\bar{u}_{i}),e(\bar{u}_{j})).
\end{align}
A direct computation gives that

$(i)$ for $i,j=1,2,...,d,$ $i<j$, then
\begin{align}\label{NAT002}
(\mathbb{C}^{0}e(\bar{u}_{i}),e(\bar{u}_{j}))=(\lambda+\mu)\partial_{x_{i}}\bar{v}\partial_{x_{j}}\bar{v};
\end{align}

$(ii)$ for $i=1,2,...,d$, $j=d+1,...,\frac{d(d+1)}{2}$, there exist two indices $1\leq k_{j}<l_{j}\leq d$ such that
$\bar{u}_{j}=\psi_{j}\bar{v}=(0,...,0,x_{l_{j}}\bar{v},0,...,0,-x_{k_{j}}\bar{v},0,...,0)$. If $k_{j}\neq i,\,l_{j}\neq i$, then
\begin{align}\label{NAT003}
(\mathbb{C}^{0}e(\bar{u}_{i}),e(\bar{u}_{j}))=\lambda\partial_{x_{i}}\bar{v}(x_{l_{j}}\partial_{k_{j}}\bar{v}-x_{k_{j}}\partial_{x_{l_{j}}}\bar{v}),
\end{align}
and if $k_{j}=i,\,l_{j}\neq i$, then
\begin{align}\label{NAT005}
(\mathbb{C}^{0}e(\bar{u}_{i}),e(\bar{u}_{j}))=\mu x_{l_{j}}\sum^{d}_{s=1}(\partial_{x_{s}}\bar{v})^{2}+(\lambda+\mu)\partial_{x_{i}}\bar{v}(x_{l_{j}}\partial_{k_{j}}\bar{v}-x_{k_{j}}\partial_{x_{l_{j}}}\bar{v}),
\end{align}
and if $k_{j}\neq i,\,l_{j}=i$, then
\begin{align}\label{NAT006}
(\mathbb{C}^{0}e(\bar{u}_{i}),e(\bar{u}_{j}))=&-\mu x_{k_{j}}\sum^{d}_{s=1}(\partial_{x_{s}}\bar{v})^{2}+(\lambda+\mu)\partial_{x_{i}}\bar{v}(x_{l_{j}}\partial_{k_{j}}\bar{v}-x_{k_{j}}\partial_{x_{l_{j}}}\bar{v});
\end{align}

$(iii)$ for $i,j=d+1,...,\frac{d(d+1)}{2}$, $i<j$, there exist four indices $1\leq k_{i}<l_{i}\leq d$ and $1\leq k_{j}<l_{j}\leq d$ such that $\bar{u}_{i}=\psi_{i}\bar{v}=(0,...,0,x_{l_{i}}\bar{v},0,...,0,-x_{k_{i}}\bar{v},0,...,0)$
and
$\bar{u}_{j}=\psi_{j}\bar{v}=(0,...,0,x_{l_{j}}\bar{v},0,...,0,-x_{k_{j}}\bar{v},0,...,0)$. Without loss of generality, we let $l_{j}\leq l_{i}$. If $k_{i}\neq k_{j},\,l_{i}\neq l_{j},\,k_{i}\neq l_{j}$, then
\begin{align}\label{NAT007}
(\mathbb{C}^{0}e(\bar{u}_{i}),e(\bar{u}_{j}))=\lambda(x_{l_{i}}\partial_{x_{k_{i}}}\bar{v}-x_{k_{i}}\partial_{x_{l_{i}}}\bar{v})(x_{l_{j}}\partial_{x_{k_{j}}}\bar{v}-x_{k_{j}}\partial_{x_{l_{j}}}\bar{v}),
\end{align}
and if $k_{i}=k_{j},\,l_{i}\neq l_{j}$, then
\begin{align}\label{NAT008}
(\mathbb{C}^{0}e(\bar{u}_{i}),e(\bar{u}_{j}))=&\mu x_{l_{i}}x_{l_{j}}\sum^{d}_{s=1}(\partial_{x_{s}}\bar{v})^{2}+\mu x_{l_{i}}\partial_{x_{l_{j}}}\bar{v}(x_{l_{j}}\partial_{x_{l_{j}}}\bar{v}-x_{k_{i}}\partial_{x_{k_{i}}}\bar{v})\notag\\
&+(\lambda+\mu)(x_{l_{i}}\partial_{x_{k_{i}}}\bar{v}-x_{k_{i}}\partial_{x_{l_{i}}}\bar{v})(x_{l_{j}}\partial_{x_{k_{j}}}\bar{v}-x_{k_{j}}\partial_{x_{l_{j}}}\bar{v}),
\end{align}
and if $k_{i}\neq k_{j},\,l_{i}=l_{j}$, then
\begin{align}\label{NAT009}
(\mathbb{C}^{0}e(\bar{u}_{i}),e(\bar{u}_{j}))=&\mu x_{k_{i}}x_{k_{j}}\sum^{d}_{s=1}(\partial_{x_{s}}\bar{v})^{2}+\mu x_{k_{i}}\partial_{x_{k_{j}}}\bar{v}(x_{k_{j}}\partial_{x_{k_{j}}}\bar{v}-x_{l_{i}}\partial_{x_{l_{i}}}\bar{v})\notag\\
&+(\lambda+\mu)(x_{l_{i}}\partial_{x_{k_{i}}}\bar{v}-x_{k_{i}}\partial_{x_{l_{i}}}\bar{v})(x_{l_{i}}\partial_{x_{k_{j}}}\bar{v}-x_{k_{j}}\partial_{x_{l_{i}}}\bar{v}),
\end{align}
and if $k_{j}<l_{j}=k_{i}<l_{i}$, then
\begin{align}\label{NAT010}
(\mathbb{C}^{0}e(\bar{u}_{i}),e(\bar{u}_{j}))=&-\mu x_{k_{j}}x_{l_{i}}\sum^{d}_{s=1}(\partial_{x_{s}}\bar{v})^{2}+\mu x_{l_{i}}\partial_{x_{k_{j}}}\bar{v}(x_{k_{i}}\partial_{x_{k_{i}}}\bar{v}-x_{k_{j}}\partial_{x_{k_{j}}}\bar{v})\notag\\
&+(\lambda+\mu)(x_{l_{i}}\partial_{x_{k_{i}}}\bar{v}-x_{k_{i}}\partial_{x_{l_{i}}}\bar{v})(x_{k_{i}}\partial_{x_{k_{j}}}\bar{v}-x_{k_{j}}\partial_{x_{k_{i}}}\bar{v}).
\end{align}

Therefore, in view of \eqref{NAT001}--\eqref{NAT010}, by using Corollary \ref{thm86}, the symmetry of integral region, the parity of integrand and the fact that
\begin{align*}
\left|\int^{\varepsilon+h_{1}(x')}_{h_{2}(x')}x_{d}\,dx_{d}\right|\leq\frac{1}{2}|\varepsilon+(h_{1}+h)(x')|\delta(x')\leq C(\varepsilon+|x'|^{1+\alpha})^{2},\quad \mathrm{in}\; B'_{R},
\end{align*}
we derive
\begin{align}\label{MAH01}
\mathrm{II}=&O(1)
\begin{cases}
|\ln\varepsilon|,&d=2,\,i=1,j=2,\\
\varepsilon^{(d-2)\hat{\gamma}},&d\geq3,\,i,j=1,2,...,d,\,i<j,\\
\varepsilon^{(d-1)\hat{\gamma}},&d\geq2,\,i=1,2,...,d,\,j=d+1,...,\frac{d(d+1)}{2},\\
\varepsilon^{d\hat{\gamma}},&d\geq3,\,i,j=d+1,...,\frac{d(d+1)}{2},\,i<j.
\end{cases}
\end{align}

For the last part $\mathrm{III}$, we further split it into two parts as follows:
\begin{align*}
\mathrm{III}_{1}=&\int_{(\Omega_{R}\setminus\Omega_{\varepsilon^{\bar{\gamma}}})\setminus(\Omega^{\ast}_{R}\setminus\Omega^{\ast}_{\varepsilon^{\bar{\gamma}}})}(\mathbb{C}^{0}e(u_{i}),e(u_{j}))+\int_{\Omega^{\ast}_{R}\setminus\Omega^{\ast}_{\varepsilon^{\bar{\gamma}}}}(\mathbb{C}^{0}e(u_{i}-u_{i}^{\ast}),e(u_{j}-u_{j}^{\ast}))\notag\\
&+\int_{\Omega^{\ast}_{R}\setminus\Omega^{\ast}_{\varepsilon^{\bar{\gamma}}}}(\mathbb{C}^{0}e(u_{i}-u_{i}^{\ast}),e(u_{j}^{\ast}))+\int_{\Omega^{\ast}_{R}\setminus\Omega^{\ast}_{\varepsilon^{\bar{\gamma}}}}(\mathbb{C}^{0}e(u_{i}^{\ast}),e(u_{j}-u_{j}^{\ast})),\\
\mathrm{III}_{2}=&\int_{\Omega^{\ast}_{R}\setminus\Omega^{\ast}_{\varepsilon^{\bar{\gamma}}}}(\mathbb{C}^{0}e(u_{i}^{\ast}),e(u_{j}^{\ast})).
\end{align*}
In view of the fact that the thickness of $(\Omega_{R}\setminus\Omega_{\varepsilon^{\bar{\gamma}}})\setminus(\Omega^{\ast}_{R}\setminus\Omega^{\ast}_{\varepsilon^{\bar{\gamma}}})$ is $\varepsilon$, it follows from (\ref{Le2.025}), (\ref{con035}), \eqref{LGA01} and \eqref{QNWE002} that
\begin{align}\label{KAM001}
\mathrm{III}_{1}=O(1)\varepsilon^{(1+\alpha)\hat{\gamma}}.
\end{align}

As for $\mathrm{III}_{2}$, we deduce that

$(a)$ for $d=2$, $i=1,j=2$,
\begin{align*}
\mathrm{III}_{2}=&\int_{\Omega_{R}^{\ast}\setminus\Omega^{\ast}_{\varepsilon^{\hat{\gamma}}}}(\mathbb{C}^{0}e(\bar{u}_{1}^{\ast}),e(\bar{u}_{2}^{\ast}))+\int_{\Omega_{R}^{\ast}\setminus\Omega^{\ast}_{\varepsilon^{\hat{\gamma}}}}(\mathbb{C}^{0}e(u_{1}^{\ast}-\bar{u}_{1}^{\ast}),e(u_{2}^{\ast}-\bar{u}_{2}^{\ast}))\notag\\
&+\int_{\Omega_{R}^{\ast}\setminus\Omega^{\ast}_{\varepsilon^{\hat{\gamma}}}}(\mathbb{C}^{0}e(u_{1}^{\ast}-\bar{u}_{1}^{\ast}),e(\bar{u}_{2}^{\ast}))+\int_{\Omega_{R}^{\ast}\setminus\Omega^{\ast}_{\varepsilon^{\hat{\gamma}}}}(\mathbb{C}^{0}e(\bar{u}_{1}^{\ast}),e(u_{2}^{\ast}-\bar{u}_{2}^{\ast})),
\end{align*}
which, together with $(\mathbb{C}^{0}e(\bar{u}_{1}^{\ast}),e(\bar{u}_{2}^{\ast}))=(\lambda+\mu)\partial_{x_{1}}\bar{v}^{\ast}\partial_{x_{2}}\bar{v}^{\ast}$, leads to that
\begin{align}\label{PLAM001}
\mathrm{III}_{2}=O(1)|\ln\varepsilon|;
\end{align}

$(b)$ for $d\geq3,\,i,j=1,2,...,d,\,i<j$, for $d\geq2,\,i=1,2,...,d,\,j=d+1,...,\frac{d(d+1)}{2},i<j$, or for $d\geq3,\,i,j=d+1,...,\frac{d(d+1)}{2},\,i<j$, similar to \eqref{NAT001}, applying \eqref{NAT002}--\eqref{NAT010} with $\bar{v}$ replaced by $\bar{v}^{\ast}$ for $\int_{\Omega^{\ast}_{\varepsilon^{\hat{\gamma}}}}(\mathbb{C}^{0}e(u_{i}^{\ast}),e(u_{i}^{\ast}))$, we obtain
\begin{align}\label{KAM002}
&\mathrm{III}_{2}-\int_{\Omega^{\ast}_{R}}(\mathbb{C}^{0}e(u_{i}^{\ast}),e(u_{j}^{\ast}))\notag\\
=&-\int_{\Omega^{\ast}_{\varepsilon^{\hat{\gamma}}}}(\mathbb{C}^{0}e(u_{i}^{\ast}),e(u_{j}^{\ast}))\notag\\
=&O(1)
\begin{cases}
\varepsilon^{(d-2)\hat{\gamma}},&d\geq3,\,i,j=1,2,...,d,\,i<j,\\
\varepsilon^{(d-1)\hat{\gamma}},&d\geq2,\,i=1,2,...,d,\,j=d+1,...,\frac{d(d+1)}{2},\\
\varepsilon^{d\hat{\gamma}},&d\geq3,\,i,j=d+1,...,\frac{d(d+1)}{2},\,i<j.
\end{cases}
\end{align}
Then it follows from \eqref{KAM001}--\eqref{KAM002} that
\begin{align*}
\mathrm{III}=\int_{\Omega_{R}\setminus\Omega_{\varepsilon^{\hat{\gamma}}}}(\mathbb{C}^{0}e(u_{i}),e(u_{j}))=O(1)|\ln\varepsilon|,\quad d=2,\,i=1,j=2,
\end{align*}
and
\begin{align*}
&\mathrm{III}-\int_{\Omega^{\ast}_{R}}(\mathbb{C}^{0}e(u_{i}^{\ast}),e(u_{j}^{\ast}))\notag\\
=&O(1)
\begin{cases}
\varepsilon^{\hat{\gamma}\min\{1+\alpha,d-2\}},&d\geq3,\,i,j=1,2,...,d,\,i<j,\\
\varepsilon^{\hat{\gamma}\min\{1+\alpha,d-1\}},&d\geq2,\,i=1,2,...,d,\,j=d+1,...,\frac{d(d+1)}{2},\\
\varepsilon^{\hat{\gamma}d},&d\geq3,\,i,j=d+1,...,\frac{d(d+1)}{2},\,i<j.
\end{cases}
\end{align*}
This, in combination with \eqref{KKAA123} and \eqref{MAH01}, reads that
\begin{align*}
a_{12}=O(1)|\ln\varepsilon|,\quad d=2,
\end{align*}
and
\begin{align*}
a_{ij}=a_{ij}^{\ast}+&O(1)
\begin{cases}
\varepsilon^{\hat{\gamma}\min\{1+\alpha,d-2\}},&d\geq3,\,i,j=1,2,...,d,\,i<j,\\
\varepsilon^{\hat{\gamma}\min\{1+\alpha,d-1\}},&d\geq2,\,i=1,2,...,d,\,j=d+1,...,\frac{d(d+1)}{2},\\
\varepsilon^{\hat{\gamma}(1+\alpha)},&d\geq3,\,i,j=d+1,...,\frac{d(d+1)}{2},\,i<j.
\end{cases}
\end{align*}

\end{proof}

Before proving Theorem \ref{OMG123}, we first state a lemma with its proof seen in \cite{BJL2017}.
\begin{lemma}\label{FBC6}
There exists a positive universal constant $C$, independent of $\varepsilon$, such that
\begin{align}\label{BJ010}
\sum^{\frac{d(d+1)}{2}}_{i,j=1}a_{ij}\xi_{i}\xi_{j}\geq\frac{1}{C},\quad\;\,\forall\;\xi\in\mathbb{R}^{\frac{d(d+1)}{2}},\;|\xi|=1.
\end{align}
\end{lemma}

Now we have all the necessary ingredients to complete the proof of Theorem \ref{OMG123}.
\begin{proof}[Proof of Theorem \ref{OMG123}.]

To begin with, denote
\begin{align*}
X=(C^{1},C^{2},...,C^{\frac{d(d+1)}{2}})^{T},\quad Y=(Q_{1}[\varphi],Q_{2}[\varphi],...,Q_{\frac{d(d+1)}{2}}[\varphi])^{T},
\end{align*}
and
\begin{gather*}
\mathbb{A}=\begin{pmatrix} a_{11}&a_{12}&\cdots&a_{1\frac{d(d+1)}{2}} \\ a_{21}&a_{22}&\cdots&a_{2\frac{d(d+1)}{2}} \\ \vdots&\vdots&\ddots&\vdots\\a_{\frac{d(d+1)}{2}1}&a_{\frac{d(d+1)}{2}2}&\cdots&a_{\frac{d(d+1)}{2}\frac{d(d+1)}{2}}
\end{pmatrix}.
\end{gather*}
Then \eqref{Le3.078} becomes
\begin{equation}\label{PAKD001}
\mathbb{A}X=Y.
\end{equation}

We now divide into two cases to accomplish the proof.

$(i)$ If $d=2$, denote
\begin{align*}
\mathbb{B}_{i}[\varphi]=\begin{pmatrix} Q_{i}[\varphi]&a_{i3}\\  Q_{3}[\varphi]&a_{33}\\  \end{pmatrix},\quad i=1,2.
\end{align*}
Therefore, utilizing \eqref{AAPP321}, \eqref{zzw003} and \eqref{LVZQ0011gdw}, we obtain
\begin{align*}
\det\mathbb{B}_{i}[\varphi]=\det\mathbb{B}_{i}^{\ast}[\varphi]+O(\varepsilon^{\min\{\frac{\alpha^{2}}{2(1+2\alpha)(1+\alpha)^{2}},\frac{(1-\alpha)\alpha}{2(1+2\alpha)}\}}),\quad i=1,2,
\end{align*}
which yields that
\begin{align}\label{QGH01}
\frac{\det\mathbb{B}_{i}[\varphi]
}{a_{33}}=&\frac{\det\mathbb{B}_{i}^{\ast}[\varphi]}{a_{33}^{\ast}}\frac{1}{1-\frac{a_{33}^{\ast}-a_{33}}{a_{33}^{\ast}}}+\frac{\det\mathbb{B}_{i}[\varphi]-\det\mathbb{B}_{i}^{\ast}[\varphi]}{a_{33}}\notag\\
=&\frac{\det\mathbb{B}_{i}^{\ast}[\varphi]}{a_{33}^{\ast}}(1+O(\varepsilon^{\min\{\frac{\alpha^{2}}{2(1+2\alpha)(1+\alpha)^{2}},\frac{(1-\alpha)\alpha}{2(1+2\alpha)}\}})),
\end{align}
and
\begin{align}\label{QGH02}
\frac{Q_{3}[\varphi]
}{a_{33}}=&\frac{Q_{3}^{\ast}[\varphi]}{a_{33}^{\ast}}\frac{1}{1-\frac{a_{33}^{\ast}-a_{33}}{a_{33}^{\ast}}}+\frac{Q_{3}[\varphi]-Q_{3}^{\ast}[\varphi]}{a_{33}}=\frac{Q_{3}^{\ast}[\varphi]}{a_{33}^{\ast}}(1+O(\varepsilon^{\frac{\alpha}{2(1+2\alpha)}})).
\end{align}
In light of \eqref{zzw001}, we obtain that for $i=1,2,$
\begin{align}\label{AJDC0}
\frac{1}{a_{ii}}=&\frac{\varepsilon^{\frac{\alpha}{1+\alpha}}}{\mathcal{L}_{2}^{i}\mathcal{M}_{\alpha,\tau}}\frac{1}{1-\frac{\mathcal{L}_{2}^{i}\mathcal{M}_{\alpha,\tau}-\varepsilon^{\frac{\alpha}{1+\alpha}}a_{ii}}{\mathcal{L}_{2}^{i}\mathcal{M}_{\alpha,\tau}}}=\frac{\varepsilon^{\frac{\alpha}{1+\alpha}}(1+O(\tilde{\varepsilon}(\alpha,\beta)))}{\mathcal{L}_{2}^{i}\mathcal{M}_{\alpha,\tau}},
\end{align}
where $\tilde{\varepsilon}(\alpha,\beta)$ is defined by \eqref{APDN001}.

Then applying the Cramer's rule for \eqref{PAKD001}, it follows from \eqref{QGH01}--\eqref{AJDC0} that for $i=1,2,$
\begin{align*}
C^{i}=&\frac{\prod\limits_{j\neq i}^{2}a_{jj}\det\mathbb{B}_{i}[\varphi]}{\prod\limits_{j=1}^{3}a_{jj}}(1+O(\varepsilon^{\frac{\alpha}{1+\alpha}}|\ln\varepsilon|))=\frac{\det\mathbb{B}_{i}^{\ast}[\varphi]}{a_{33}^{\ast}}\frac{\varepsilon^{\frac{\alpha}{1+\alpha}}(1+O(\varepsilon(\alpha,\beta)))}{\mathcal{L}_{d}^{i}\mathcal{M}_{\alpha,\tau}},
\end{align*}
and
\begin{align*}
C^{3}=&\frac{Q_{3}[\varphi]}{a_{33}}(1+O(\varepsilon^{\frac{\alpha}{1+\alpha}}))=\frac{Q_{3}^{\ast}[\varphi]}{a_{33}^{\ast}}(1+O(\varepsilon^{\frac{\alpha}{2(1+2\alpha)}})),
\end{align*}
where $\varepsilon(\alpha,\beta)$ is defined in \eqref{ZWZHAO01}.

We now claim that $a_{33}^{\ast}>0$. In fact, when we choose $\xi=(0,0,1)$ in \eqref{BJ010}, we see from \eqref{zzw003} that
\begin{align*}
a_{33}=a_{33}^{\ast}+O(1)\varepsilon^{\frac{\alpha}{2(1+2\alpha)}}\geq\frac{1}{C},
\end{align*}
where the constant $C$ is independent of $\varepsilon$. That is, $a_{33}^{\ast}\geq\frac{1}{C}>0$.

$(ii)$ If $d\geq3$, for $i=1,2,...,\frac{d(d+1)}{2}$, we replace the elements of $i$-th column in the matrix $\mathbb{A}$ by column vector $Y$ and then denote this new matrix by $\mathbb{F}_{i}[\varphi]$ as follows:
\begin{gather*}
\mathbb{F}_{i}[\varphi]=
\begin{pmatrix}
a_{11}&\cdots&Q_{1}[\varphi]&\cdots&a_{1\frac{d(d+1)}{2}} \\\\ \vdots&\ddots&\vdots&\ddots&\vdots\\\\a_{\frac{d(d+1)}{2}1}&\cdots&Q_{\frac{d(d+1)}{2}}[\varphi]&\cdots&a_{\frac{d(d+1)}{2}\frac{d(d+1)}{2}}
\end{pmatrix}.
\end{gather*}
Then it follows from Lemmas \ref{KM323} and \ref{lemmabc} that for $i=1,2,...,\frac{d(d+1)}{2}$,
\begin{align*}
\frac{\det\mathbb{F}_{i}[\varphi]
}{\det\mathbb{A}}=&\frac{\det\mathbb{F}_{i}^{\ast}[\varphi]}{\det\mathbb{A}^{\ast}}\frac{1}{1-\frac{\det\mathbb{A}^{\ast}-\det\mathbb{A}}{\det\mathbb{A}^{\ast}}}+\frac{\det\mathbb{F}_{i}[\varphi]-\det\mathbb{F}_{i}^{\ast}[\varphi]}{\det\mathbb{A}}\notag\\
=&\frac{\det\mathbb{F}_{i}^{\ast}[\varphi]}{\det\mathbb{A}^{\ast}}(1+O(\bar{\varepsilon}(\alpha,d))),
\end{align*}
which, in combination with the Cramer's rule, yields that
\begin{align*}
C^{i}=\frac{\det\mathbb{F}_{i}[\varphi]}{\det\mathbb{A}}=\frac{\det\mathbb{F}_{i}^{\ast}[\varphi]}{\det\mathbb{A}^{\ast}}(1+O(\bar{\varepsilon}(\alpha,d))),
\end{align*}
where $\bar{\varepsilon}(\alpha,d)$ is defined in \eqref{NZKL001}.

We now demonstrate that $\det\mathbb{A}^{\ast}>0$. Similarly as before, it follows from Lemma \ref{lemmabc} and \eqref{zzw003} that
\begin{align*}
\sum^{\frac{d(d+1)}{2}}_{i,j=1}a_{ij}\xi_{i}\xi_{j}=\sum^{\frac{d(d+1)}{2}}_{i,j=1}a_{ij}^{\ast}\xi_{i}\xi_{j}+O(\varepsilon^{\frac{\alpha^{2}}{2(1+2\alpha)(1+\alpha)^{2}}})\geq\frac{1}{C},\quad\;\,\forall\;\xi\in\mathbb{R}^{\frac{d(d+1)}{2}},\;|\xi|=1,
\end{align*}
which implies that
\begin{align*}
\sum^{\frac{d(d+1)}{2}}_{i,j=1}a_{ij}^{\ast}\xi_{i}\xi_{j}\geq\frac{1}{C}.
\end{align*}
That is, the matrix $\mathbb{A}^{\ast}$ is positive definite and we thus have
\begin{align*}
\det\mathbb{A}^{\ast}>0.
\end{align*}

\end{proof}

\section{Proof of Example \ref{CORO001}}\label{SEC006}
\begin{lemma}
Assume as in Example \ref{CORO001}. Then, for a sufficiently small $\varepsilon>0$,

$(i)$ for $i=1,2,$
\begin{align}\label{KAN001}
a_{ii}=&\mathcal{L}_{2}^{i}\mathcal{M}_{\alpha,\tau_{0}}\varepsilon^{-\frac{\alpha}{1+\alpha}}+\mathcal{K}^{\ast}+O(1)|\ln\varepsilon|;
\end{align}

$(ii)$ for $i=3$,
\begin{align}\label{KAN002}
a_{33}=a_{33}^{\ast}+O(1)\varepsilon^{\frac{\alpha}{2(1+2\alpha)}},
\end{align}
where $\mathcal{M}_{\alpha,\tau_{0}}$ is defined by \eqref{zwzh001} with $\tau=\tau_{0}$, $\mathcal{L}_{2}^{i}$, $i=1,2$ are defined in \eqref{AZ} with $d=2$, $\mathcal{K}^{\ast}$ is defined by \eqref{LKM} below.

\end{lemma}

\begin{proof}
The proof of \eqref{KAN002} is contained in the proof of \eqref{zzw003} and thus omitted here. We next prove \eqref{KAN001}. Pick $\bar{\gamma}=\frac{\alpha^{2}}{2(1+2\alpha)(1+\alpha)^{2}}$. By the same argument as in \eqref{aaaa01}, we deduce that for $i=1,2$,
\begin{align*}
a_{ii}=&\mathcal{L}_{2}^{i}\left(\int_{\varepsilon^{\bar{\gamma}}<|x_{1}|<r_{0}}\frac{dx_{1}}{h_{1}(x_{1})-h(x_{1})}+\int_{|x_{1}|<\varepsilon^{\bar{\gamma}}}\frac{dx_{1}}{\varepsilon+h_{1}(x_{1})-h(x_{1})}\right)+O(1)|\ln\varepsilon|.
\end{align*}

Observe that by utilizing Taylor expansion, we derive
\begin{align}\label{ZCZ009}
h_{1}(x_{1})-h(x_{1})=\tau_{0}|x_{1}|^{1+\alpha}+O(|x_{1}|^{2+2\alpha}),\quad\mathrm{in}\;B'_{r_{0}},
\end{align}
where $\tau_{0}$ is defined by \eqref{AKLGJ001}. Then it follows from \eqref{ZCZ009} that
\begin{align*}
\int_{\varepsilon^{\bar{\gamma}}<|x_{1}|<r_{0}}\left(\frac{1}{h_{1}-h}-\frac{1}{\tau_{0}|x_{1}|^{1+\alpha}}\right)dx_{1}=\int_{\varepsilon^{\bar{\gamma}}<|x_{1}|<r_{0}}O(1)dx_{1}=C^{\ast}+O(1)\varepsilon^{\bar{\gamma}},
\end{align*}
where $C^{\ast}$ depends on $\tau_{0},r_{0}$, but not on $\varepsilon$. Then
\begin{align*}
\int_{\varepsilon^{\bar{\gamma}}<|x_{1}|<r_{0}}\frac{dx_{1}}{h_{1}-h}=&\int_{\varepsilon^{\bar{\gamma}}<|x_{1}|<r_{0}}\frac{dx_{1}}{\tau_{0}|x_{1}|^{1+\alpha}}+C^{\ast}+O(1)\varepsilon^{\bar{\gamma}}.
\end{align*}
Similarly, we derive
\begin{align*}
\int_{|x_{1}|<\varepsilon^{\bar{\gamma}}}\frac{dx_{1}}{\varepsilon+h_{1}-h}=&\int_{|x_{1}|<\varepsilon^{\bar{\gamma}}}\frac{dx_{1}}{\varepsilon+\tau_{0}|x_{1}|^{1+\alpha}}+O(1)\varepsilon^{\bar{\gamma}}.
\end{align*}
Then the energy $a_{ii}$ becomes
\begin{align*}
a_{ii}=&\mathcal{L}_{2}^{i}\left(\int_{\varepsilon^{\bar{\gamma}}<|x_{1}|<r_{0}}\frac{dx_{1}}{\tau_{0}|x_{1}|^{1+\alpha}}+\int_{|x_{1}|<\varepsilon^{\bar{\gamma}}}\frac{dx_{1}}{\varepsilon+\tau_{0}|x_{1}|^{1+\alpha}}\right)+C^{\ast}+O(1)|\ln\varepsilon|.
\end{align*}
Since
\begin{align*}
&\int_{\varepsilon^{\bar{\gamma}}<|x_{1}|<r_{0}}\frac{dx_{1}}{\tau_{0}|x_{1}|^{1+\alpha}}+\int_{|x_{1}|<\varepsilon^{\bar{\gamma}}}\frac{dx_{1}}{\varepsilon+\tau_{0}|x_{1}|^{1+\alpha}}\notag\\
=&\int_{-\infty}^{+\infty}\frac{1}{\varepsilon+\tau_{0}|x_{1}|^{1+\alpha}}-\int_{|x_{1}|>r_{0}}\frac{dx_{1}}{\tau_{0}|x_{1}|^{1+\alpha}}+\int_{|x_{1}|>\varepsilon^{\bar{\gamma}}}\frac{\varepsilon}{\tau_{0}|x_{1}|^{1+\alpha}(\varepsilon+\tau_{0}|x_{1}|^{1+\alpha})}\notag\\
=&\mathcal{M}_{\alpha,\tau_{0}}\varepsilon^{-\frac{\alpha}{1+\alpha}}-\frac{2}{\alpha\tau_{0}r_{0}^{\alpha}}+O(1)\varepsilon^{1-(1+2\alpha)\bar{\gamma}},
\end{align*}
then we obtain that for $i=1,2,$
\begin{align*}
a_{ii}=&\mathcal{L}_{2}^{i}\mathcal{M}_{\alpha,\tau_{0}}\varepsilon^{-\frac{\alpha}{1+\alpha}}+\mathcal{K}^{\ast}+O(1)|\ln\varepsilon|,
\end{align*}
where
\begin{align}\label{LKM}
\mathcal{K}^{\ast}=C^{\ast}-\frac{2\mathcal{L}_{2}^{i}}{\alpha\tau_{0}r_{0}^{\alpha}}.
\end{align}

\end{proof}

\begin{proof}[Proof of Example \ref{CORO001}]

Denote
\begin{align}\label{QKLP001}
\mathcal{G}^{\ast}_{i}=\frac{\mathcal{K}^{\ast}}{\mathcal{L}_{2}^{i}\mathcal{M}_{\alpha,\tau_{0}}},\quad i=1,2.
\end{align}
In light of \eqref{KAN001}, we deduce
\begin{align}\label{DYA001}
\frac{1}{a_{ii}}=&\frac{\varepsilon^{\frac{\alpha}{1+\alpha}}}{\mathcal{L}_{2}^{i}\mathcal{M}_{\alpha,\tau_{0}}}\frac{1}{1-\frac{\mathcal{L}_{2}^{i}\mathcal{M}_{\alpha,\tau_{0}}-\varepsilon^{\frac{\alpha}{1+\alpha}}a_{ii}}{\mathcal{L}_{2}^{i}\mathcal{M}_{\alpha,\tau_{0}}}}=\frac{\varepsilon^{\frac{\alpha}{1+\alpha}}}{\mathcal{L}_{2}^{i}\mathcal{M}_{\alpha,\tau_{0}}}\frac{1}{1+\mathcal{G}^{\ast}_{i}\varepsilon^{\frac{\alpha}{1+\alpha}}+O(\varepsilon^{\frac{\alpha}{1+\alpha}}|\ln\varepsilon|)}\notag\\
=&\frac{\varepsilon^{\frac{\alpha}{1+\alpha}}}{\mathcal{L}_{2}^{i}\mathcal{M}_{\alpha,\tau_{0}}}\frac{1+O(\varepsilon^{\frac{\alpha}{1+\alpha}}|\ln\varepsilon|)}{1+\mathcal{G}^{\ast}_{i}\varepsilon^{\frac{\alpha}{1+\alpha}}}.
\end{align}
Similarly as in \eqref{QGH01}--\eqref{QGH02}, utilizing \eqref{KAN001}--\eqref{KAN002} and \eqref{DYA001}, it follows from the Cramer's rule that for $i=1,2,$
\begin{align*}
C^{i}=&\frac{\prod\limits_{j\neq i}^{2}a_{jj}\det\mathbb{B}_{i}[\varphi]}{\prod\limits_{j=1}^{3}a_{jj}}(1+O(\varepsilon^{\frac{\alpha}{1+\alpha}}|\ln\varepsilon|))\notag\\
=&\frac{\det\mathbb{B}_{i}^{\ast}[\varphi]}{a_{33}^{\ast}}\frac{\varepsilon^{\frac{\alpha}{1+\alpha}}}{\mathcal{L}_{2}^{i}\mathcal{M}_{\alpha,\tau_{0}}}\frac{1+O(\varepsilon^{\min\{\frac{\alpha^{2}}{2(1+2\alpha)(1+\alpha)^{2}},\frac{(1-\alpha)\alpha}{2(1+2\alpha)}\}})}{1+\mathcal{G}^{\ast}_{i}\varepsilon^{\frac{\alpha}{1+\alpha}}},
\end{align*}
and
\begin{align*}
C^{3}=&\frac{Q_{3}[\varphi]}{a_{33}}(1+O(\varepsilon^{\frac{\alpha}{1+\alpha}}))=\frac{Q_{3}^{\ast}[\varphi]}{a_{33}^{\ast}}(1+O(\varepsilon^{\frac{\alpha}{2(1+2\alpha)}})).
\end{align*}
This, together with decomposition \eqref{Le2.015}, Theorems \ref{lem89999} and \ref{thm86}, yields that
\begin{align*}
\nabla u=&\sum^{2}_{i=1}\frac{\det\mathbb{B}_{i}^{\ast}[\varphi]}{a_{33}^{\ast}}\frac{\varepsilon^{\frac{\alpha}{1+\alpha}}}{\mathcal{L}_{2}^{i}\mathcal{M}_{\alpha,\tau_{0}}}\frac{1+O(\varepsilon^{\min\{\frac{\alpha^{2}}{2(1+2\alpha)(1+\alpha)^{2}},\frac{(1-\alpha)\alpha}{2(1+2\alpha)}\}})}{1+\mathcal{G}^{\ast}_{i}\varepsilon^{\frac{\alpha}{1+\alpha}}}(\nabla\bar{u}_{i}+O(\delta^{-\frac{1}{1+\alpha}}))\notag\\
&+\frac{Q_{3}^{\ast}[\varphi]}{a_{33}^{\ast}}(1+O(\varepsilon^{\frac{\alpha}{2(1+2\alpha)}}))(\nabla\bar{u}_{3}+O(1))+\nabla\bar{u}_{0}+O(1)\|\varphi\|_{C^{1}(\partial D)}\notag\\
=&\sum^{2}_{i=1}\frac{\det\mathbb{B}_{i}^{\ast}[\varphi]}{a_{33}^{\ast}}\frac{\varepsilon^{\frac{\alpha}{1+\alpha}}}{\mathcal{L}_{2}^{i}\mathcal{M}_{\alpha,\tau_{0}}}\frac{1+O(\varepsilon^{\min\{\frac{\alpha^{2}}{2(1+2\alpha)(1+\alpha)^{2}},\frac{(1-\alpha)\alpha}{2(1+2\alpha)}\}})}{1+\mathcal{G}^{\ast}_{i}\varepsilon^{\frac{\alpha}{1+\alpha}}}\nabla\bar{u}_{i}\notag\\
&+\frac{Q_{3}^{\ast}[\varphi]}{a_{33}^{\ast}}(1+O(\varepsilon^{\frac{\alpha}{2(1+2\alpha)}}))\nabla\bar{u}_{3}+\nabla\bar{u}_{0}+O(1)\delta^{-\frac{1-\alpha}{1+\alpha}}\|\varphi\|_{C^{1}(\partial D)}.
\end{align*}
This completes the proof.

\end{proof}

\section{Appendix:\,The proofs of Theorems \ref{CL001} and \ref{CL002}}

\subsection{$C^{1,\alpha}$ estimates.}
Let $Q\subseteq\mathbb{R}^{d}$ be a Lipschitz domain and introduce the Campanato space $\mathcal{L}^{2,\lambda}(Q)$, $\lambda\geq0$ as follows:
\begin{align*}
\mathcal{L}^{2,\lambda}(Q):=\bigg\{u\in L^{2}(Q):\,\sup\limits_{\scriptstyle x_{0}\in Q\atop\scriptstyle\;\;
\rho>0\hfill}\frac{1}{\rho^{\lambda}}\int_{B_{\rho}(x_{0})\cap Q}|u-u_{x_{0},\rho}|^{2}dx<+\infty \bigg\},
\end{align*}
where $u_{x_{0},\rho}:=\frac{1}{|Q\cap B_{\rho}(x_{0})|}\int_{Q\cap B_{\rho}(x_{0})}u(x)\,dx$. We endow the Campanato space  $\mathcal{L}^{2,\lambda}(Q)$ with the semi-norm
\begin{align*}
[u]^{2}_{\mathcal{L}^{2,\lambda}(Q)}:=\sup\limits_{\scriptstyle x_{0}\in Q\atop\scriptstyle\;\;
\rho>0\hfill}\frac{1}{\rho^{\lambda}}\int_{B_{\rho}(x_{0})\cap Q}|u-u_{x_{0},\rho}|^{2}dx,
\end{align*}
and the norm
\begin{align*}
\|u\|_{\mathcal{L}^{2,\lambda}(Q)}:=\|u\|_{L^{2}(Q)}+[u]_{\mathcal{L}^{2,\lambda}(Q)}.
\end{align*}
It is well known that the Campanato space $\mathcal{L}^{2,\lambda}(Q)$ is equivalent to the H\"{o}lder space $C^{0,\alpha}(Q)$ in the case of $d<\lambda\leq d+2$ and $\alpha=\frac{\lambda-d}{2}$.

We first state a classical result in Theorem 5.14 of \cite{GM2013}.
\begin{theorem}\label{ASDL0}
Let $Q\subset\mathbb{R}^{d}$ be a Lipschitz domain. Let $w\in H^{1}(Q;\mathbb{R}^{d})$ be a solution of
\begin{align*}
-\partial_{j}(C^{0}_{ijkl}\partial_{l}w^{k})=\partial_{j}f_{ij}, \quad in\; Q,
\end{align*}
with $f_{ij}\in C^{0,\alpha}(Q)$, $0<\alpha<1$, and constant coefficients $C_{ijkl}^{0}$ satisfying \eqref{coeffi001}. Then $\nabla w\in\mathcal{L}^{2,d+2\alpha}_{loc}(Q)$ and for $B_{R}:=B_{R}(x_{0})\subset Q$,
\begin{align*}
\|\nabla w\|_{\mathcal{L}^{2,d+2\alpha}(B_{R/2})}\leq C(\|\nabla w\|_{L^{2}(B_{R})}+[F]_{\mathcal{L}^{2,d+2\alpha}(B_{R})}),
\end{align*}
where $F=(f_{ij})$ and $C=C(d,\alpha,R)$.
\end{theorem}

In light of the equivalence between the Campanato space and the H\"{o}lder space, it follows from the proof of Theorem \ref{ASDL0} (Theorem 5.14 of \cite{GM2013}) that
\begin{corollary}\label{CO001}
Assume as in Lemma \ref{CL001}. Let $w$ be the solution of \eqref{ADCo1}. Then for $B_{R}:=B_{R}(x_{0})\subset Q$,
\begin{align}\label{LAG001}
[\nabla w]_{\alpha,B_{R/2}}\leq C\Big(\frac{1}{R^{1+\alpha}}\|w\|_{L^{\infty}(B_{R})}+[F]_{\alpha,B_{R}} \Big),
\end{align}
where $C=C(d,\alpha,R)$.

\end{corollary}

\begin{proof}[Proof of Lemma \ref{CL001}]
In light of the definition of a $C^{1,\alpha}$ domain, at each point $x_{0}\in\Gamma$ there is a neighbourhood $U$ of $x_{0}$ and a homeomorphism $\Psi\in C^{1,\alpha}(U)$ that straightens the boundary in $U$, that is,
\begin{align*}
\Psi(U\cap Q)=\mathcal{B}^{+}_{1}=\{y\in\mathcal{B}_{R}(0):y_{d}>0\},\quad \Psi(U\cap \Gamma)=\partial\mathcal{B}^{+}_{1}=\{y\in\mathcal{B}_{R}(0):y_{d}=0\},
\end{align*}
where $\mathcal{B}_{1}(0):=\{y\in\mathbb{R}^{d}:|y|<1\}$. Under the mapping $y=\Psi(x)=(\Psi^{1}(x),...,\Psi^{d}(x))$, we define
\begin{align*}
\mathcal{W}(y):=w(\Phi^{-1}(y)),\quad \mathcal{J}:=\frac{\partial((\Psi^{-1})^{1},...,(\Psi^{-1})^{d})}{\partial(y^{1},...,y^{d})},\quad |\mathcal{J}(y)|:=\det\mathcal{J}(y),
\end{align*}
and
\begin{align*}
\mathcal{C}^{0}_{ijkl}(y):=&C^{0}_{i\hat{j}k\hat{l}}|\mathcal{J}(y)|(\partial_{\hat{l}}(\Psi^{-1})^{l}(y))^{-1}\partial_{\hat{j}}\Psi^{j}(\Psi^{-1}(y)),\\
\mathcal{F}_{ij}(y):=&|\mathcal{J}(y)|\partial_{\hat{l}}\Psi^{j}(\Psi^{-1}(y))f_{i\hat{l}}(\Psi^{-1}(y)).
\end{align*}
Recalling \eqref{ADCo1} and by virtue of this transformation, then $\mathcal{W}$ satisfies
\begin{align}\label{LAH001}
\begin{cases}
-\partial_{j}(\mathcal{C}^{0}_{ijkl}(y)\partial_{l}\mathcal{W}^{k})=\partial_{j}\mathcal{F}_{ij},&\quad\mathrm{in}\;\mathcal{B}^{+}_{R},\\
\mathcal{W}=0,&\quad\mathrm{on}\;\partial \mathcal{B}_{R}^{+}\cap\partial\mathbb{R}^{d}_{+},
\end{cases}
\end{align}
where $0<R\leq1$. Denote $y_{0}=\Psi(x_{0})$. By freezing the coefficients, we rewrite equation \eqref{LAH001} as follows:
\begin{align*}
-\partial_{j}(\mathcal{C}^{0}_{ijkl}(y_{0})\partial_{l}\mathcal{W}^{k})=\partial_{j}((\mathcal{C}^{0}_{ijkl}(y)-\mathcal{C}^{0}_{ijkl}(y_{0}))\partial_{l}\mathcal{W}^{k})+\partial_{j}\mathcal{F}_{ij}.
\end{align*}
In view of the equivalence between the Campanato space and the H\"{o}lder space and using the proof of Theorem 7.1 (Theorem 5.14 of \cite{GM2013}) again, we deduce that
\begin{align*}
[\nabla\mathcal{W}]_{\alpha,\mathcal{B}^{+}_{R/2}}\leq& C\Big(\frac{1}{R^{1+\alpha}}\|\mathcal{W}\|_{L^{\infty}(\mathcal{B}^{+}_{R})}+[\mathcal{F}]_{\alpha,\mathcal{B}^{+}_{R}}\Big)\notag\\
&+C[(\mathcal{C}^{0}_{ijkl}(y)-\mathcal{C}^{0}_{ijkl}(y_{0}))\partial_{l}\mathcal{W}^{k}]_{\alpha,\mathcal{B}^{+}_{R}},
\end{align*}
where $\mathcal{F}:=(\mathcal{F}_{ij})$. Since $\mathcal{C}^{0}_{ijkl}(y)\in C^{0,\alpha}$, then we have
\begin{align*}
[(\mathcal{C}^{0}_{ijkl}(y)-\mathcal{C}^{0}_{ijkl}(y_{0}))\partial_{l}\mathcal{W}^{k}]_{\alpha,\mathcal{B}^{+}_{R}}\leq C(R^{\alpha}[\nabla\mathcal{W}]_{\alpha,\mathcal{B}^{+}_{R}}+\|\nabla\mathcal{W}\|_{L^{\infty}(\mathcal{B}^{+}_{R})}).
\end{align*}
Applying the interpolation inequality (for example, see Lemma 6.32 in \cite{GT1998}), we obtain
\begin{align*}
\|\nabla\mathcal{W}\|_{L^{\infty}(\mathcal{B}^{+}_{R})}\leq R^{\alpha}[\nabla\mathcal{W}]_{\alpha,\mathcal{B}^{+}_{R}}+\frac{C}{R}\|\mathcal{W}\|_{L^{\infty}(\mathcal{B}^{+}_{R})},
\end{align*}
where $C=C(d)$. Then, we have
\begin{align}\label{HMA001}
[\nabla\mathcal{W}]_{\alpha,\mathcal{B}^{+}_{R/2}}\leq&C\Big(\frac{1}{R^{1+\alpha}}\|\mathcal{W}\|_{L^{\infty}(\mathcal{B}^{+}_{R})}+R^{\alpha}[\nabla\mathcal{W}]_{\alpha,\mathcal{B}^{+}_{R}}+[\mathcal{F}]_{\alpha,\mathcal{B}^{+}_{R}}\Big).
\end{align}
Since $\Psi$ is a homeomorphism, then the norms in \eqref{HMA001} defined on $\mathcal{B}^{+}_{R}$ are equivalent to those on $\mathcal{N}=\Psi^{-1}(\mathcal{B}^{+}_{R})$. Then back to $w$, we have
\begin{align*}
[\nabla w]_{\alpha,\mathcal{N}^{'}}\leq&C\Big(\frac{1}{R^{1+\alpha}}\|w\|_{L^{\infty}(\mathcal{N})}+R^{\alpha}[\nabla w]_{\alpha,\mathcal{N}}+[F]_{\alpha,\mathcal{N}}\Big),
\end{align*}
where $\mathcal{N}'=\Psi^{-1}(\mathcal{B}^{+}_{R/2})$ and $C=C(d,\alpha,\Psi)$. Moreover, there exists a constant $0<\sigma<1$, independent of $R$, such that $B_{\sigma R}(x_{0})\cap Q\subset\mathcal{N}'$.

Then for any domain $Q'\subset\subset Q\cup\Gamma$ and each $x_{0}\in Q'\cap\Gamma$, there exist $\mathcal{R}_{0}:=\mathcal{R}_{0}(x_{0})$ and $C_{0}=C_{0}(d,\alpha,x_{0})$ such that
\begin{align}\label{HMA003}
[\nabla w]_{\alpha,B_{\mathcal{R}_{0}}(x_{0})\cap Q'}\leq&C_{0}\Big(\mathcal{R}_{0}^{\alpha}[\nabla w]_{\alpha,Q'}+\frac{1}{\mathcal{R}_{0}^{1+\alpha}}\|w\|_{L^{\infty}(Q)}+[F]_{\alpha,Q}\Big).
\end{align}
Therefore, using the finite covering theorem for the collection $\{B_{\mathcal{R}_{0}/2}(x_{0})|\,x_{0}\in\Gamma\cap Q'\}$, we obtain that there exist finite $B_{\mathcal{R}_{i}}(x_{i})$, $i=1,2,...,K$, covering $\Gamma\cap Q'$. Denote the constant in \eqref{HMA003} corresponding to $x_{i}$ by $C_{i}$. Define
\begin{align*}
\overline{C}:=\max\limits_{1\leq i\leq K}\{C_{i}\},\quad \overline{\mathcal{R}}:=\min\limits_{1\leq i\leq K}\big\{\frac{\mathcal{R}_{i}}{2}\big\}.
\end{align*}
Then for every $x_{0}\in\Gamma\cap Q'$, there exists some $1\leq i_{0}\leq K$ such that $B_{\overline{R}}(x_{0})\subset B_{\overline{R}_{i_{0}}}(x_{i_{0}})$ and
\begin{align}\label{HMA005}
[\nabla w]_{\alpha,B_{\overline{R}}(x_{0})\cap Q'}\leq&\overline{C}\Big(\overline{R}^{\alpha}[\nabla w]_{\alpha,Q'}+\frac{1}{\overline{R}^{1+\alpha}}\|w\|_{L^{\infty}(Q)}+[F]_{\alpha,Q}\Big).
\end{align}

We now proceed to establish the corresponding estimates on $Q'$. Denote the constant in \eqref{LAG001} of Corollary \ref{CO001} by $\widetilde{C}$. Set
\begin{align*}
\widehat{C}:=\max\{\overline{C},\widetilde{C}\},\quad\widehat{\mathcal{R}}:=\min\{(3\widehat{C})^{-1/\alpha},\overline{\mathcal{R}}\}.
\end{align*}
For any $x_{1},x_{2}\in Q'$, then we have
\begin{enumerate}
\item[$(i)$]
$|x_{1}-x_{2}|\geq\frac{\widehat{\mathcal{R}}}{2}$;
\item[$(ii)$]
there exists some $1\leq i_{0}\leq K$ such that $x_{1},x_{2}\in B_{\widehat{R}/2}(x_{i_{0}})\cap Q'$;
\item[$(iii)$]
$x_{1},x_{2}\in B_{\widehat{R}/2}\subset Q'$.
\end{enumerate}
In the case of $(i)$, we obtain
\begin{align*}
\frac{|\nabla w(x_{1})-\nabla w(x_{2})|}{|x_{1}-x_{2}|^{\alpha}}\leq \frac{2^{1+\alpha}}{\widehat{R}^{\alpha}}\|\nabla w\|_{L^{\infty}(Q')}.
\end{align*}
In the case of $(ii)$, we see from \eqref{HMA005} that
\begin{align*}
\frac{|\nabla w(x_{1})-\nabla w(x_{2})|}{|x_{1}-x_{2}|^{\alpha}}\leq&[\nabla w]_{\alpha,B_{\widehat{R}/2(x_{i_{0}})\cap Q'}}\notag\\
\leq&\widehat{C}\Big(\widehat{R}^{\alpha}[\nabla w]_{\alpha,Q'}+\frac{1}{\widehat{R}^{1+\alpha}}\|w\|_{L^{\infty}(Q)}+[F]_{\alpha,Q}\Big).
\end{align*}
In the case of $(iii)$, it follows from Corollary \ref{CO001} that
\begin{align*}
\frac{|\nabla w(x_{1})-\nabla w(x_{2})|}{|x_{1}-x_{2}|^{\alpha}}\leq&[\nabla w]_{\alpha,B_{\widehat{R}/2}}\leq\widehat{C}\Big(\frac{1}{\widehat{R}^{1+\alpha}}\|w\|_{L^{\infty}(Q)}+[F]_{\alpha,Q}\Big).
\end{align*}
Consequently, we derive
\begin{align*}
[\nabla w]_{\alpha,Q'}\leq\widehat{C}\Big(\widehat{R}^{\alpha}[\nabla w]_{\alpha,Q'}+\frac{1}{\widehat{R}^{1+\alpha}}\|w\|_{L^{\infty}(Q)}+[F]_{\alpha,Q}\Big)+\frac{2^{1+\alpha}}{\widehat{R}^{\alpha}}\|\nabla w\|_{L^{\infty}(Q')}.
\end{align*}
Using the interpolation inequality (see Lemma 6.32 in \cite{GT1998}) again, we have
\begin{align*}
\frac{2^{1+\alpha}}{\widehat{R}^{\alpha}}\|\nabla w\|_{L^{\infty}(Q')}\leq \frac{1}{3}[\nabla w]_{\alpha,Q'}+\frac{C}{\widehat{R}^{1+\alpha}}\|w\|_{L^{\infty}(Q')},
\end{align*}
where $C=C(d,\alpha)$. Then in view of $\widehat{R}\leq(3\widehat{C})^{-1/\alpha}$, we obtain
\begin{align*}
[\nabla w]_{\alpha,Q'}\leq C(\|w\|_{L^{\infty}(Q)}+[F]_{\alpha,Q}),
\end{align*}
where $C=C(d,\alpha,Q',Q)$. By utilizing the interpolation inequality, we prove that \eqref{GNA001} holds.

\end{proof}

\subsection{$W^{1,p}$ estimates}
\begin{proof}[Proof of Lemma \ref{CL002}]
We first establish the $W^{1,p}$ interior estimates. In view of the fact that $w\neq0$ on $\partial B_{R}$ for any $B_{R}\subset Q$, we introduce a cut-off function $\eta\in C_{0}^{\infty}(B_{R})$ satisfying
\begin{align*}
0\leq\eta\leq1,\quad\eta=1\;\mathrm{in}\; B_{\rho},\quad|\nabla\eta|\leq\frac{C}{R-\rho}.
\end{align*}
From equation \eqref{ADCo1}, we see that $\eta w$ verifies
\begin{align*}
\int_{B_{R}}C^{0}_{ijkl}\partial_{l}(\eta w^{k})\partial_{j}\varphi^{i}=\int_{B_{R}}(G_{i}\varphi^{i}+\widetilde{F}\partial_{j}\varphi^{i}),\quad\forall\varphi\in C^{\infty}_{0}(B_{R};\mathbb{R}^{d}),
\end{align*}
where
\begin{align*}
G_{i}:=f_{ij}\partial_{j}\eta-C^{0}_{ijkl}\partial_{l}w^{k}\partial_{j}\eta,\quad\widetilde{F}_{ij}:=f_{ij}\eta+C^{0}_{ijkl}w^{k}\partial_{l}\eta.
\end{align*}
Let $v\in H^{1}_{0}(B_{R};\mathbb{R}^{d})$ be the weak solution of
\begin{align}\label{LLAM0}
-\Delta v^{i}=G_{i}.
\end{align}
Then $\eta w$ satisfies
\begin{align*}
\int_{B_{R}}C^{0}_{ijkl}\partial_{l}(\eta w^{k})\partial_{j}\varphi^{i}=\int_{B_{R}}\widehat{F}\partial_{j}\varphi^{i},\quad\forall\varphi\in C^{\infty}_{0}(B_{R};\mathbb{R}^{d}),
\end{align*}
where $\widehat{F}_{ij}:=\widetilde{F}_{ij}+\partial_{j}v^{i}$.

In light of $f_{ij}\in C^{0,\alpha}$, we know that $f_{ij}\in L^{p}(B_{R})$ for any $d\leq p<\infty$. Assume that $w\in W^{1,q}(B_{R};\mathbb{R}^{d})$, $q\geq2$. Then we obtain
\begin{align}\label{MBL001}
G_{i}&\in L^{p\wedge q}(B_{R}),\quad\mathrm{where}\;p\wedge q:=\min\{p,q\},
\end{align}
and
\begin{align}\label{MBL002}
\widetilde{F}_{ij}\in L^{p\wedge q^{\ast}}(B_{R}),\quad\mathrm{where}\;q^{\ast}:=&
\begin{cases}
\frac{dq}{d-q},&q<d,\\
2q,&q\geq d.
\end{cases}
\end{align}
Applying $L^{2}$ estimate to equation \eqref{LLAM0}, we see $\nabla^{2}v\in L^{2}(B_{R})$ and
\begin{align*}
-\Delta(\partial_{j}v^{i})=\partial_{j}G_{i}.
\end{align*}
Then in view of \eqref{MBL001}, it follows from Theorem 7.1 of \cite{GM2013} that $\nabla(\partial_{j}v^{i})\in L^{p\wedge q}(B_{R})$. Using the Sobolev embedding theorem, we derive that $\partial_{j}v^{i}\in L^{(p\wedge q)^{\ast}}$. This, together with \eqref{MBL002}, yields that $\widehat{F}_{ij}\in L^{p\wedge q^{\ast}}(B_{R})$. Further, utilizing Theorem 7.1 of \cite{GM2013} again, we deduce
\begin{align*}
\|\nabla(\eta w)\|_{L^{p\wedge q^{\ast}}(B_{R})}\leq C\|\widehat{F}\|_{L^{p\wedge q^{\ast}}(B_{R})},
\end{align*}
where $C=C(d,\lambda,\mu,p,q)$ and $\widehat{F}:=(\widehat{F}_{ij})$, $i,j=1,2,...,d$. Recalling the definition of $G_{i}$ and $\widetilde{F}_{ij}$, it follows from \eqref{MBL001}--\eqref{MBL002} that
\begin{align}\label{AGTQ001}
\|\nabla w\|_{L^{p\wedge q^{\ast}}(B_{\rho})}\leq\frac{C}{R-\rho}(\|w\|_{W^{1,p}(B_{R})}+\|F\|_{L^{p}(B_{R})}),
\end{align}
where $C=C(d,\lambda,\mu,p,q)$.

We proceed to demonstrate that $\nabla w\in L^{p}(B_{R/2})$. Pick a series of balls with radii as follows:
\begin{align*}
R_{k}=R\Big(\frac{1}{2}+\frac{1}{2^{k+1}}\Big),\quad k\geq0.
\end{align*}
First, taking $\rho=R_{1}$ and $q=2$ in \eqref{AGTQ001}, we then have
\begin{align*}
\|\nabla w\|_{L^{p\wedge 2^{\ast}}(B_{R_{1}})}\leq\frac{C}{R}(\|w\|_{W^{1,2}(B_{R})}+\|F\|_{L^{p}(B_{R})}).
\end{align*}
If $p\leq2^{\ast}$, then the proof is finished. If $p>2^{\ast}$, then $\nabla w\in L^{2^{\ast}}(B_{R_{1}})$ and
\begin{align*}
\|\nabla w\|_{L^{2^{\ast}}(B_{R_{1}})}\leq\frac{C}{R}(\|w\|_{W^{1,2}(B_{R})}+\|F\|_{L^{p}(B_{R})}).
\end{align*}
This, in combination with choosing $R=R_{1}$, $\rho=R_{2}$ and $q=2^{\ast}$ in \eqref{AGTQ001}, reads that
\begin{align*}
\|\nabla w\|_{L^{p\wedge2^{\ast\ast}}(B_{R_{2}})}\leq&\frac{C}{R}(\|w\|_{W^{1,2^{\ast}}(B_{R_{1}})}+\|F\|_{L^{p}(B_{R_{1}})})\notag\\
\leq&\frac{C}{R^{2}}(\|w\|_{W^{1,2}(B_{R})}+\|F\|_{L^{p}(B_{R})}).
\end{align*}
If $p\leq2^{\ast\ast}$, then the proof is finished. If $p>2^{\ast\ast}$, repeating the above argument with finite steps, we deduce that $\nabla w\in L^{p}(B_{R/2})$ and
\begin{align}\label{JANT001}
\|\nabla w\|_{L^{p}(B_{R/2})}\leq C(\|w\|_{H^{1}(B_{R})}+\|F\|_{L^{p}(B_{R})}),
\end{align}
where $C=C(d,\lambda,\mu,p,\mathrm{dist}(B_{R},\partial Q))$.

Next, we prove the $W^{1,p}$ estimates near boundary $\Gamma$ by using the method of locally flattening the boundary, which is the same to the proof in Lemma \ref{CL001}. For brevity, we employ the same notations as before. Therefore, we know that $\mathcal{W}(y):=w(\Psi^{-1}(y))\in H^{1}(\mathcal{B}^{+}_{R},\mathbb{R}^{d})$ verifies
\begin{align*}
\int_{\mathcal{B}^{+}_{R}}\mathcal{C}^{0}_{ijkl}(y)\partial_{l}\mathcal{W}^{k}\partial_{j}\varphi^{i}dy=\int_{\mathcal{B}^{+}_{R}}\mathcal{F}_{ij}\partial_{j}\varphi^{i}dy,\quad\forall\varphi\in H^{1}_{0}(\mathcal{B}^{+}_{R},\mathbb{R}^{d}).
\end{align*}
By following the proof of Theorem 7.2 of \cite{GM2013}, we deduce that for any $d\leq p<\infty$,
\begin{align*}
\|\nabla\mathcal{W}\|_{L^{p}(\mathcal{B}^{+}_{R/2})}\leq C(\|\mathcal{W}\|_{H^{1}(\mathcal{B}^{+}_{R})}+\|\mathcal{F}\|_{L^{p}(\mathcal{B}^{+}_{R})}),
\end{align*}
where $C=C(\lambda,\mu,p,R,\Psi)$. Then back to $w$, we have
\begin{align*}
\|\nabla w\|_{L^{p}(\mathcal{N}^{'})}\leq C(\|\mathcal{W}\|_{H^{1}(\mathcal{N})}+\|\mathcal{F}\|_{L^{p}(\mathcal{N})}),
\end{align*}
where $\mathcal{N}'=\Psi^{-1}(\mathcal{B}^{+}_{R/2})$, $\mathcal{N}=\Psi^{-1}(\mathcal{B}^{+}_{R})$ and $C=C(\lambda,\mu,p,R,\Psi)$. Moreover, there exists a constant $0<\sigma<1$, independent of $R$, such that $B_{\sigma R}\cap Q\subset\mathcal{N}'$.

Then for any $x_{0}\in Q'\cap\Gamma$, there exists $R_{0}:=R_{0}(x_{0})>0$ such that
\begin{align}\label{JANT002}
\|\nabla w\|_{L^{p}(B_{\sigma R_{0}}(x_{0})\cap Q')}\leq C(\|w\|_{H^{1}(Q)}+\|F\|_{L^{p}(Q)}),
\end{align}
where $C=C(\lambda,\mu,p,x_{0},R)$. In view of \eqref{JANT001}--\eqref{JANT002}, it follows from the finite covering theorem that
\begin{align*}
\|\nabla w\|_{L^{p}(Q')}\leq C(\|w\|_{H^{1}(Q)}+\|F\|_{L^{p}(Q)}),
\end{align*}
where $C=C(\lambda,\mu,p,Q',Q)$. This, together with the Poincar\'{e} inequality, yields that \eqref{LNZ001} holds.

Note that for any constant matrix $\mathcal{M}=(\mathfrak{a}_{ij})$, $i,j=1,2,...,d$, $w$ satisfies \eqref{ADCo1} with $F$ replaced by $F-\mathcal{M}$. Then using the continuous injection that $W^{1,p}\hookrightarrow C^{0,\gamma}$, $0<\gamma\leq1-d/p$, we deduce that \eqref{LNZ002} holds.

\end{proof}

%\noindent{\bf{\large Acknowledgements.}} X. Hao is greatly indebted to Professor HaiGang Li for his constant encouragement and support. Z.W. Zhao would like to thank School of Mathematical Sciences at Beijing Normal University for the stimulating environment. Z.W. Zhao was partially supported by NSFC (11971061) and BJNSF (1202013).

\bibliographystyle{plain}

\begin{thebibliography}{10}

\bibitem{ADN1959} S. Agmon, A. Douglis, L. Nirenberg, Estimates near the boundary for solutions of elliptic partial differential equations satisfying general boundary conditions. I, Comm. Pure Appl. Math. 12 (1959) 623-727.
%
\bibitem{ADN1964} S. Agmon, A. Douglis, L. Nirenberg, Estimates near the boundary for solutions of elliptic partial differential equations satisfying general boundary conditions. II, Comm. Pure Appl. Math. 17 (1964) 35-92.

\bibitem{ABTV2015}  H. Ammari, E. Bonnetier, F. Triki, M. Vogelius, Elliptic estimates in composite media with smooth inclusions: an integral equation approach, Ann. Sci. \'{E}c. Norm. Sup\'{e}r. (4) 48 (2) (2015) 453-495.

\bibitem{ACKLY2013} H. Ammari, G. Ciraolo, H. Kang, H. Lee, K. Yun, Spectral analysis of the Neumann-Poincar\'{e} operator and characterization of the stress concentration in anti-plane elasticity, Arch. Ration. Mech. Anal. 208 (2013) 275-304.

\bibitem{AKL2005} H. Ammari, H. Kang, M. Lim, Gradient estimates to the conductivity problem. Math. Ann. 332 (2005) 277-286.

\bibitem{AKLLL2007} H. Ammari, H. Kang, H. Lee, J. Lee, M. Lim, Optimal estimates for the electrical field in two dimensions, J. Math. Pures Appl. 88 (2007) 307-324.

\bibitem{BASL1999} I. Babu\u{s}ka, B. Andersson, P. Smith, K. Levin, Damage analysis of fiber composites. I. Statistical analysis on fiber scale, Comput. Methods Appl. Mech. Engrg. 172 (1999) 27-77.

\bibitem{BJL2017} J.G. Bao, H.J. Ju, H.G. Li, Optimal boundary gradient estimates for Lam\'{e} systems with partially infinite coefficients, Adv. Math. 314 (2017) 583-629.

\bibitem{BLL2015} J.G. Bao, H.G. Li, Y.Y. Li, Gradient estimates for solutions of the Lam\'{e} system with partially infinite coefficients, Arch. Ration. Mech. Anal. 215 (1) (2015) 307-351.

\bibitem{BLL2017} J.G. Bao, H.G. Li, Y.Y. Li, Gradient estimates for solutions of the Lam\'{e} system with partially infinite coefficients in dimensions greater than two, Adv. Math. 305 (2017) 298-338.

\bibitem{BLY2009} E.S. Bao, Y.Y. Li, B. Yin, Gradient estimates for the perfect conductivity problem, Arch. Ration. Mech. Anal. 193 (2009) 195-226.

\bibitem{BLY2010} E.S. Bao, Y.Y. Li, B. Yin, Gradient estimates for the perfect and insulated conductivity problems with multiple inclusions, Comm.
    Partial Differential Equations 35 (2010) 1982-2006.

\bibitem{BV2000} E. Bonnetier, M. Vogelius, An elliptic regularity result for a composite medium with ``touching'' fibers of circular cross-section, SIAM J. Math. Anal. 31 (2000) 651-677.

\bibitem{CL2019} Y. Chen, H.G. Li, Estimates and Asymptotics for the stress concentration between closely spaced stiff $C^{1,\gamma}$ inclusions in linear elasticity, J. Funct. Anal. 281 (2) (2021) 109038.

\bibitem{CS2019} G. Ciraolo, A. Sciammetta, Gradient estimates for the perfect conductivity problem in anisotropic media, J. Math. Pures Appl. 127 (2019) 268-298.

\bibitem{CS20192} G. Ciraolo, A. Sciammetta, Stress concentration for closely located inclusions in nonlinear perfect conductivity problems, J. Differential Equations 266 (2019) 6149-6178.

\bibitem{DL2019} H.J. Dong, H.G. Li, Optimal estimates for the conductivity problem by Green's function method, Arch. Ration. Mech. Anal. 231 (3) (2019) 1427-1453.

\bibitem{GB2005} Y. Gorb, L. Berlyand, Asymptotics of the effective conductivity of composites with closely spaced inclusions of optimal shape, Quart. J. Mech. Appl. Math. 58 (1) (2005) 84-106.

\bibitem{GN2012} Y. Gorb, A. Novikov, Blow-up of solutions to a $p$-Laplace equation, Multiscale Model. Simul. 10 (3) (2012) 727-743.

\bibitem{G2015} Y. Gorb, Singular behavior of electric field of high-contrast concentrated composites, Multiscale Model. Simul. 13 (2015) 1312-1326.

\bibitem{GM2013} M. Giaquinta, L. Martinazzi, An introduction to the regularity theory for elliptic systems, harmonic maps and minimal graphs, Springer Science Business Media, 2013.

\bibitem{GT1998} D. Gilbarg, N.S. Trudinger, Elliptic partial differential equations of second order, Springer 1998.

\bibitem{HZ2021} X. Hao, Z.W. Zhao, The asymptotics for the perfect conductivity problem with stiff $C^{1,\alpha}$-inclusions. J. Math. Anal. Appl. 501 (2021), no. 2, Paper No. 125201, 27 pp.

\bibitem{KLY2013}  H. Kang, M. Lim, K. Yun, Asymptotics and computation of the solution to the conductivity equation in the presence of adjacent inclusions with extreme conductivities, J. Math. Pures Appl. (9) 99 (2013) 234-249.

\bibitem{KLY2014} H. Kang, M. Lim, K. Yun, Characterization of the electric field concentration between two adjacent spherical perfect conductors,  SIAM J. Appl. Math. 74 (2014) 125-146.

\bibitem{KLY2015} H. Kang, H. Lee, K. Yun, Optimal estimates and asymptotics for the stress concentration between closely located stiff inclusions, Math. Ann. 363 (3-4) (2015) 1281-1306.

\bibitem{KY2019} H. Kang, S. Yu, Quantitative characterization of stress concentration in the presence of closely spaced hard inclusions in two-dimensional linear elasticity, Arch. Ration. Mech. Anal. 232 (2019) 121-196.

\bibitem{KY2020} H. Kang, S. Yu, A proof of the Flaherty-Keller formula on the effective property of densely packed elastic composites, Calc. Var. Partial Differential Equations 59 (1) (2020) Paper No. 22, 13 pp.

\bibitem{KY2019002} H. Kang, K. Yun, Optimal estimates of the field enhancement in presence of a bow-tie structure of perfectly conducting inclusions in two dimensions. J. Differential Equations 266 (2019), no. 8, 5064-5094.

\bibitem{KL2019} J. Kim, M. Lim, Electric field concentration in the presence of an inclusion with eccentric core-shell geometry, Math. Ann. 373 (1-2) (2019) 517-551.

\bibitem{LLBY2014} H.G. Li, Y.Y. Li, E.S. Bao, B. Yin, Derivative estimates of solutions of elliptic systems in narrow regions, Quart. Appl. Math.  72 (3) (2014) 589-596.

\bibitem{l2020} H.G. Li, Asymptotics for the electric field concentration in the perfect conductivity problem, SIAM J. Math. Anal. 52 (4) (2020) 3350-3375.

\bibitem{LLY2019} H.G. Li, Y.Y. Li, Z.L. Yang, Asymptotics of the gradient of solutions to the perfect conductivity problem, Multiscale Model. Simul. 17 (3) (2019) 899-925.

\bibitem{LZ2020} H.G. Li, Z.W. Zhao, Boundary blow-up analysis of gradient estimates for Lam\'{e} systems in the presence of $m$-convex hard inclusions. SIAM J. Math. Anal. 52 (4) (2020) 3777-3817.

\bibitem{LN2003} Y.Y. Li, L. Nirenberg, Estimates for elliptic system from composite material, Comm. Pure Appl. Math. 56 (2003) 892-925.

\bibitem{LV2000} Y.Y. Li, M. Vogelius, Gradient stimates for solutions to divergence form elliptic equations with discontinuous coefficients, Arch. Rational Mech. Anal. 153 (2000) 91-151.

\bibitem{LY2009} M. Lim, K. Yun, Blow-up of electric fields between closely spaced spherical perfect conductors, Comm. Partial Differential Equations 34 (2009) 1287-1315.

\bibitem{MMN2007} V.G. Maz'ya, A.B. Movchan, M.J. Nieves, Uniform asymptotic formulae for Green's tensors in elastic singularly perturbed domains, Asymptot. Anal. 52 (2007) 173-206.

\bibitem{MZ202102} C.X. Miao, Z.W. Zhao, Singular analysis of the stress concentration in the narrow regions between the inclusions and the matrix boundary, arXiv:2109.04394.

\bibitem{Y2007} K. Yun, Estimates for electric fields blown up between closely adjacent conductors with arbitrary shape, SIAM J. Appl. Math. 67 (2007) 714-730.

\bibitem{Y2009} K. Yun, Optimal bound on high stresses occurring between stiff fibers with arbitrary shaped cross-sections, J. Math. Anal. Appl. 350 (2009) 306-312.

\bibitem{ZH202101} Z.W. Zhao and X. Hao, Asymptotics for the concentrated field between closely located hard inclusions in all dimensions, Commun. Pure Appl. Anal. 20 (2021) 2379-2398.

\end{thebibliography}

\end{document}